\documentclass{amsart}
\usepackage{amssymb,amsmath}
\usepackage{stmaryrd}
\usepackage{yhmath}
\usepackage{latexsym}
\usepackage{amsthm}
\usepackage{amscd}
\usepackage{mathrsfs}
\usepackage[all]{xy}
\usepackage{mathtools}
\usepackage{arydshln}
\usepackage{bbm}
\usepackage{bm}
\usepackage{url}
\usepackage{color}

\newcommand{\sym}{\mathrm{sym}}
\newcommand{\ur}{\mathrm{ur}}
\newcommand{\ram}{\mathrm{ram}}

\newcommand{\red}{\mathrm{red}}

\newcommand{\Kal}{\mathrm{Kal}}
\newcommand{\Tam}{\mathrm{Tam}}
\newcommand{\Yu}{\mathrm{Yu}}
\newcommand{\KY}{\mathrm{KY}}

\newcommand{\HT}{\mathrm{HT}}
\newcommand{\BH}{\mathrm{BH}}

\newcommand{\rec}{\mathrm{rec}}

\newcommand{\ol}{\overline}

\newcommand{\Jac}[2]{\begin{pmatrix}\frac{#1}{#2}\end{pmatrix}}

\newcommand{\mcO}{\mathcal{O}}
\newcommand{\mfp}{\mathfrak{p}}
\newcommand{\mcT}{\mathcal{T}}
\newcommand{\bfT}{\mathbf{T}}

\newcommand{\Z}{\mathbb{Z}}
\newcommand{\F}{\mathbb{F}}
\newcommand{\R}{\mathbb{R}}

\newcommand{\C}{\mathbb{C}}
\newcommand{\Gm}{\mathbb{G}_{\mathrm{m}}}
\newcommand{\G}{\mathbf{G}}
\newcommand{\J}{\mathbf{J}}
\newcommand{\bfS}{\mathbf{S}}
\newcommand{\bfB}{\mathbf{B}}

\newcommand{\B}{\mathcal{B}}

\newcommand{\bfH}{\mathbf{H}}

\newcommand{\mfs}{\mathfrak{s}}
\newcommand{\mfg}{\mathfrak{g}}

\newcommand{\mfn}{\mathfrak{n}}

\newcommand{\mfA}{\mathfrak{A}}
\newcommand{\mfP}{\mathfrak{P}}
\newcommand{\mfU}{\mathfrak{U}}
\newcommand{\mfV}{\mathfrak{V}}
\newcommand{\mfW}{\mathfrak{W}}

\newcommand{\bfi}{\mathbf{i}}
\newcommand{\spl}{\mathbf{spl}}

\DeclareMathOperator{\tr}{tr}
\DeclareMathOperator{\Nr}{Nr}
\DeclareMathOperator{\Tr}{Tr}

\DeclareMathOperator{\depth}{depth}
\DeclareMathOperator{\sgn}{sgn}

\DeclareMathOperator{\ord}{ord}
\DeclareMathOperator{\val}{val}

\DeclareMathOperator{\GL}{GL}

\DeclareMathOperator{\SL}{SL}
\DeclareMathOperator{\Sp}{Sp}

\DeclareMathOperator{\Lie}{Lie}

\DeclareMathOperator{\cInd}{c-Ind}
\DeclareMathOperator{\Ind}{Ind}

\DeclareMathOperator{\Hom}{Hom}
\DeclareMathOperator{\End}{End}
\DeclareMathOperator{\Aut}{Aut}
\DeclareMathOperator{\Ker}{Ker}
\DeclareMathOperator{\Res}{Res}

\DeclareMathOperator{\Int}{Int}

\DeclareMathOperator{\Stab}{Stab}

\DeclareMathOperator{\id}{id}

\DeclareMathOperator{\LLC}{LLC}

\theoremstyle{plain}
\newtheorem{thm}{Theorem}[section]
\newtheorem*{thm*}{Theorem}
\newtheorem{prop}[thm]{Proposition}
\newtheorem{lem}[thm]{Lemma}
\newtheorem{cor}[thm]{Corollary}

\theoremstyle{definition}
\newtheorem{defn}[thm]{Definition}

\theoremstyle{remark}
\newtheorem{rem}[thm]{Remark}
\newtheorem*{claim*}{Claim}

\SelectTips{cm}{11}

\title{Local Langlands correspondence for regular supercuspidal representations of $\GL(n)$}
\author{Masao Oi}
\address{Department of Mathematics (Hakubi center), Kyoto University, Kitashirakawa, Oiwake-cho, Sakyo-ku, Kyoto 606-8502, Japan.}
\email{masaooi@math.kyoto-u.ac.jp}

\author{Kazuki Tokimoto}
\address{Institute of Mathematics, Academia Sinica, Astronomy-Mathematics Building, No.\ 1, Sec.\ 4, Roosevelt Road, Taipei 10617, Taiwan.}
\email{tokimoto@gate.sinica.edu.tw}

\begin{document}

\begin{abstract}
In this paper, we prove the coincidence of Kaletha's recent construction of the local Langlands correspondence for regular supercuspidal representations with Harris--Taylor's one in the case of general linear groups.
The keys are Bushnell--Henniart's essentially tame local Langlands correspondence and Tam's result on Bushnell--Henniart's rectifiers.
By combining them, our problem is reduced to an elementary root-theoretic computation on the difference between Kaletha's and Tam's $\chi$-data.
\end{abstract}

\subjclass[2010]{Primary: 22E50; Secondary: 11S37, 11F70}
\keywords{local Langlands correspondence, regular supercuspidal representation, essentially tame supercuspidal representation}

\maketitle

\section{Introduction}\label{sec:intro}

One fundamental objective in representation theory of $p$-adic reductive groups is to establish the conjectural \textit{local Langlands correspondence}, which predicts the existence of a natural connection between \textit{$L$-packets} (finite sets consisting of irreducible smooth representations) and \textit{$L$-parameters} for connected reductive groups over a $p$-adic field.
From the early days of representation theory of $p$-adic reductive groups, a number of results on the local Langlands correspondence have been obtained.
Among them, Harris--Taylor's construction (\cite{MR1876802} and also by Henniart in \cite{MR1738446}) of the correspondence for general linear groups, which are the most typical examples of connected reductive groups, has a particularly significant meaning.
Since it was established, it has been playing an important role as an indispensable foundation in a lot of studies.

Recently, several attempts are being made to construct the correspondence for more general connected reductive groups beyond Harris--Taylor's work on general linear groups.
One possible approach to constructing the correspondence for general groups is to restrict the class of representations.
For example, in their striking paper \cite{MR2480618}, DeBacker and Reeder constructed $L$-packets consisting of so-called regular depth zero supercuspidal representations and their corresponding $L$-parameters for unramified connected reductive groups.
At present, the most general result in this direction is Kaletha's construction for \textit{regular supercuspidal representations} (\cite{MR4013740}).
Regular supercuspidal representations are supercuspidal representations which are obtained by Yu's construction (\cite{MR1824988}) and satisfy a certain regularity condition.
As we can see in the fact that every supercuspidal representation of $\GL_{n}$ is regular when $p$ does not divide $n$, the regularity condition is not so restrictive compared to, for example, the depth zero condition mentioned above.
Although not all supercuspidal representations are regular in general, most supercuspidal representations do have this property.
By focusing on such representations, Kaletha constructed $L$-packets and their corresponding $L$-parameters, under the assumption that the connected reductive group is tamely ramified.
(Recently, a construction for further general supercuspidal representations was announced by him in \cite{Kaletha:2019}.)
Furthermore, he succeeded in proving various important properties such as the stability, the standard endoscopic character relation, and so on (under the assumption of the ``torality'').

However, it is not so obvious from his construction whether Kaletha's correspondence recovers Harris--Taylor's local Langlands correspondence in the case of general linear groups.
The goal of this paper is to check it, that is, to prove the following theorem:
\begin{thm}[Thereom \ref{thm:main}]\label{thm:intro}
Let $F$ be a $p$-adic field with odd residual characteristic.
Then the local Langlands correspondence of Kaletha coincides with that of Harris--Taylor for regular supercuspidal representations of $\GL_{n}(F)$.
\end{thm}
Here we remark that the oddness assumption on the residual characteristic is needed for making Kaletha's construction of the correspondence work.

We explain the outline of the proof of this result.
In the following, we let $F$ be a $p$-adic field with odd residual characteristic and consider the general linear group $\GL_{n}$ over $F$.
The key ingredients are Bushnell--Henniart's consecutive work on the \textit{essentially tame local Langlands correspondence} (\cite{MR2138141,MR2148193,MR2679700}) and Tam's reinterpretation of their work (\cite{MR3509939}).

In the case of $\GL_{n}$, regular supercuspidal representations are nothing but so-called essentially tame supercuspidal representations, which have been intensively studied by Bushnell--Henniart.
Such representations can be parametrized by pairs $(E,\xi)$ (called $F$-admissible pairs) consisting of a degree $n$ tamely ramified extension $E$ of $F$ and a character $\xi$ of $E^{\times}$ satisfying a condition called $F$-admissibility.
Then Bushnell--Henniart showed that the essentially tame supercuspidal representation $\pi_{(E,\xi)}^{\BH}$ corresponding to a pair $(E,\xi)$ maps to $\Ind_{W_{E}}^{W_{F}}(\xi\mu_{\rec}^{-1})$ (as an $n$-dimensional representation of the Weil group $W_{F}$ of $F$) under Harris--Taylor's local Langlands correspondence (let us write $\LLC_{\GL_{n}}^{\HT}$ for it):
\[
\LLC_{\GL_{n}}^{\HT}(\pi_{(E,\xi)}^{\BH})
=
\Ind_{W_{E}}^{W_{F}}(\xi\mu_{\rec}^{-1}),
\]
where $\mu_{\rec}$ is a certain tamely ramified character of $E^{\times}$ which depends on $(E,\xi)$ and is called the \textit{rectifier} of $(E,\xi)$, and gave an explicit formula for the rectifier.

On the other hand, Kaletha's construction of the local Langlands correspondence is briefly explained as follows.
Let $\G$ be a tamely ramified connected reductive group over $F$.
He first gave a reinterpretation to Yu's theory of construction of supercuspidal representations of such a group.
According to it, from a pair $(\bfS,\xi)$ called a ``tame elliptic regular pair'', which consists of a tamely-ramified elliptic maximal torus $\bfS$ of $\G$ and its character $\xi$ satisfying some regularity condition, we obtain an irreducible supercuspidal representation $\pi_{(\bfS,\xi)}^{\KY}$ of $\G(F)$.
When the group $\G$ is $\GL_{n}$ over $F$, a tamely ramified elliptic maximal torus corresponds to a tamely ramified extension of $F$ of degree $n$.
Then, in fact, a tame elliptic regular pair is nothing but an $F$-admissible pair.
More precisely, if an elliptic maximal torus $\bfS$ corresponds to an extension $E$ of $F$, then a pair $(\bfS,\xi)$ is tame elliptic regular if and only if $(E,\xi)$ is $F$-admissible.
Furthermore, Yu's construction recovers Bushnell--Henniart's construction of essentially tame supercuspidal representations, i.e., we have $\pi^{\BH}_{(E,\xi)}\cong\pi_{(\bfS,\xi)}^{\KY}$ (see Appendix \ref{sec:BH-Kal}).

To a tame elliptic regular pair $(\bfS,\xi)$, Kaletha associated an $L$-parameter $\phi$ in the following way.
First, by the local Langlands correspondence for $\bfS$, we get an $L$-parameter $\phi_{\xi}$ of $\bfS$.
As $\phi_{\xi}$ is a homomorphism from the Weil group $W_{F}$ to the $L$-group ${}^{L}\bfS$ of $\bfS$, we regard it as an $L$-parameter of $\G$ if we can find an $L$-embedding of ${}^{L}\bfS$ into ${}^{L}\G$.
In fact, a general procedure for creating such an embedding from certain auxiliary data called ``\textit{$\chi$-data}'' is introduced by Langlands--Shelstad.
The point here is that the choice of the set of $\chi$-data may not be unique, hence Langlands--Shelstad's construction does not give us an $L$-embedding in a canonical way.
Then Kaletha constructed a set of $\chi$-data (let us write $\chi_{\Kal}$ for it) carefully from the information of the pair $(E,\xi)$ and defined an $L$-parameter $\phi$ of $\G$ by using Langlands--Shelstad's $L$-embedding ${}^{L}\!j_{\chi_{\Kal}}$ determined by $\chi_{\Kal}$:
\[
\xymatrix{
&{}^{L}\G\\
W_{F}\ar^-{\phi}[ru]\ar^-{\phi_{\xi}}[r]&{}^{L}\bfS\ar@{^{(}->}_-{\text{${}^{L}\!j_{\chi_{\Kal}}$ the $L$-embedding obtained from $\chi_{\Kal}$}}[u]
}
\]

Therefore, in order to investigate the relation between Harris--Taylor's and Kaletha's constructions, it is necessary to understand the $L$-parameter obtained in Kaletha's way as an $n$-dimensional representation of $W_{F}$.
In fact, what Tam pursued in his paper \cite{MR3509939} is exactly this point.
In his paper \cite{MR3509939}, he first described the $L$-parameter constructed as above explicitly as an $n$-dimensional representation of $W_{F}$.
More precisely, for a pair $(E,\xi)$ and a set $\chi$ of $\chi$-data, he proved that the $L$-parameter of $\GL_{n}$ obtained as above is given by $\Ind_{W_{E}}^{W_{F}}\xi\mu_{\chi}$ with a character $\mu_{\chi}$ of $E^{\times}$ determined by $\chi$ (see \cite[Proposition 6.5]{MR3509939} or Proposition \ref{prop:Tam} in this paper for details).
Second, he indeed constructed a set of $\chi$-data (let us write $\chi_{\Tam}$ for it) realizing $\mu_{\rec}$ as $\mu_{\chi_{\Tam}}$ from a pair $(E,\xi)$.

Then, where is the nontrivial point left in our problem of comparing the correspondences of Kaletha and Harris--Taylor?
In fact, Kaletha's construction of the correspondence requires some additional twist on $\pi_{(\bfS,\xi)}^{\KY}$ by a certain character.
The key object at this point is the character ``$\epsilon$'' defined by DeBacker--Spice in \cite{MR3849622} (and also by Kaletha in \cite{MR4013740}).
By applying the theory of DeBacker and Spice to our setting, we may attach a tamely ramified character $\epsilon$ of $\bfS(F)$ to a tame elliptic regular pair $(\bfS,\xi)$.
This character is defined according to a root-theoretic property of the pair $(\bfS,\xi)$ and appears naturally in the context of an explicit character formula for the supercuspidal representation $\pi_{(\bfS,\xi)}^{\KY}$ (which was established by Adler--Spice in \cite{MR2543925} first and deepened later by DeBacker--Spice and Kaletha).
Then Kaletha associated the $L$-parameter ${}^{L}\!j_{\chi_{\Kal}}\circ\phi_{\xi}$ to the representation $\pi_{(\bfS,\xi\epsilon)}^{\KY}$, that is, the regular supercuspidal representation arising from the twisted pair $(\bfS,\xi\epsilon)$ (note that this is again tame elliptic regular):
\[
\LLC_{\G}^{\Kal}(\pi_{(\bfS,\epsilon\xi)}^{\KY})
=
{}^{L}\!j_{\chi_{\Kal}}\circ\phi_{\xi}.
\] 

Therefore, in fact, Kaletha's $\chi$-data is not exactly the same as Tam's.
Our task is to prove that the difference between them (i.e., the ratio of $\mu_{\chi_{\Kal}}$ to $\mu_{\chi_{\Tam}}^{-1}$) is given by DeBacker--Spice's character $\epsilon$.
To do this, we have to compare the language used in Adler--Spice--DeBacker and Kaletha's work with the one in Bushnell--Henniart and Tam's work.
In the former one, every notion is defined in a sophisticated way according to the general structure theory of connected reductive groups over $p$-adic fields.
On the other hand, in the latter one, since everything can be described explicitly when specialized to $\GL_{n}$, the corresponding notions are defined concretely in a Galois-theoretic way.
Thus, in order to compare $\chi_{\Kal}$ with $\chi_{\Tam}$, we first have to grasp the relation between these two languages precisely.
If we do this appropriately, our problem is reduced to a simple case-by-case computation based on a classification of root systems with Galois actions arising from elliptic maximal tori of $\GL_{n}$.
Then the result follows from elementary properties of several fundamental arithmetic invariants such as the Jacobi symbol, Gauss sums, Langlands constants, and so on.

Thus we would like to conclude this introduction by emphasizing that the theoretically difficult part of our problem can be completed almost just by referring to Tam's work.

\medbreak
\noindent{\bfseries Organization of this paper.}\quad
In Section \ref{sec:notation}, we explain our notation on fundamental objects and several invariants appearing in this paper.
In Section \ref{sec:HT}, we recall Bushnell--Henniart's work on Harris--Taylor's local Langlands correspondence for $\GL_{n}$ and its reinterpretation due to Tam.
In Section \ref{sec:Kaletha}, we review Kaletha's construction of the local Langlands correspondence and describe it in the case of $\GL_{n}$.
In Section \ref{sec:pre}, we introduce some preliminary results needed for our comparison of two $\chi$-data $\chi_{\Kal}$ and $\chi_{\Tam}$.
What we will do in this section is basically to compare Kaletha's language, which is abstract and available for general tamely ramified groups, with Bushnell--Henniart and Tam's one, which is explicit but specialized to $\GL_{n}$.
Although we believe most of the content in this section are well-known to experts, we justify them here for the sake of completeness.
In Section \ref{sec:main}, we determine the difference between $\chi_{\Kal}$ and $\chi_{\Tam}$ by a case-by-case computation and complete the proof of our main result.
In Appendix \ref{sec:BH-Kal}, we check that when we regard a tame elliptic regular pair of $\GL_{n}$ as an $F$-admissible pair, Kaletha--Yu's and Bushnell--Henniart's constructions of supercuspidal representations give rise to exactly the same representation.

\medbreak
\noindent{\bfseries Acknowledgment.}\quad
The authors would like to thank Tasho Kaletha for his encouragement and for answering their questions.
They are also grateful to Wen-Wei Li and Alexander Bertoloni Meli for their comments.
Moreover, they also express their sincere gratitude to Geo Kam-Fai Tam for a lot of detailed comments and advice on many technical points.
The authors were able to complete this work thanks to his help.

The first author was supported by the Program for Leading Graduate Schools, MEXT, Japan and JSPS Research Fellowship for Young Scientists, and KAKENHI Grant Number 17J05451 (DC2) and 19J00846 (PD).
The second author was supported by the Research Institute for Mathematical Sciences, a Joint Usage/Research Center located in Kyoto University, Iwanami Fujukai Foundation, and JSPS KAKENHI Grant Number 19K14503.

\setcounter{tocdepth}{2}
\tableofcontents

\section{Notation}\label{sec:notation}
First we explain our notation for several notions used in this paper and list their basic properties.

\begin{description}
\item[Prime number]
In this paper, we always assume that $p$ is an odd prime number.

\item[$p$-adic field]
We fix a $p$-adic field $F$.
Let $\mcO_{F}$, $\mfp_{F}$, and $k_{F}$ denote the ring of integers, the maximal ideal of $\mcO_{F}$, and the residue field $\mcO_{F}/\mfp_{F}$, respectively.
Let $q$ be the order of $k_{F}$.
For $x \in \mcO_{F}$, $\bar{x}$ denotes the image of $x$ in $k_{F}$.
We write $\mu_{F}$ for the set of roots of unity in $F$ of order prime to $p$.
We often regard an element of $k_{F}^{\times}$ as an element of $\mu_{F}$ by the Teichm{\"u}ller lift.
We write $\val_{F}$ for the valuation of $F$ normalized so that $\val_{F}(F^{\times})=\Z$.
By extending it to $\overline{F}$ so that $\val_{F}(F^{\times})=\Z$ again holds, we always regard $\val_{F}$ as a valuation of $\overline{F}$.

We fix an algebraic closure $\ol{F}$ of $F$ and write $\ol{k}_{F}$ for the residue field of $\ol{F}$.
We consider every finite extension of $F$ or $k_{F}$ within these closures.
Let $\Gamma_{F}$ and $W_{F}$ denote the absolute Galois group of $F$ and the Weil group of $F$, respectively.

For a finite extension $E$ of $F$, we use similar notations to above such as $\mcO_{E}$, $\mfp_{E}$, and so on.
We write $e(E/F)$ and $f(E/F)$ for the ramification index and the residue degree of the extension $E/F$, respectively.

\item[Norm $1$ subgroup for a finite field]
For a finite field $k$ which contains a subfield $k'$ such that $k$ is quadratic over $k'$ (note that then such a $k'$ is unique), we put $k^{1}$ to be the kernel of the norm map with respect to $k/k'$:
\[
k^{1}:=\Ker(\Nr_{k/k'}\colon k^{\times}\twoheadrightarrow {k'}^{\times}).
\]

\item[Filtration on additive and multiplicative groups]
For a finite extension $E$ of $F$, we put
\[
E_{r}:=\{x\in E \mid \val_{F}(x)\geq r\}
\quad
\text{and}
\quad
E_{r+}:=\{x\in E \mid \val_{F}(x)> r\}
\]
for each $r\in\R$.
Similarly, for each $r\in\R_{>0}$, we put
\begin{align*}
E_{r}^{\times}&:=\{x\in 1+\mfp_{E} \mid \val_{F}(x-1)\geq r\}
\quad
\text{and}\\
E_{r+}^{\times}&:=\{x\in 1+\mfp_{E} \mid \val_{F}(x-1)> r\}.
\end{align*}
On the other hand, for a positive integer $s\in\Z_{>0}$, we put
\[
U_{E}^{s}:=1+\mfp_{E}^{s}.
\]
Here we note that, if we put $e:=e(E/F)$ to be the ramification index of the extension $E/F$, then we have
\[
E_{\frac{s}{e}}^{\times}=U_{E}^{s}
\]
for each $s\in\Z_{>0}$.

\begin{rem}\label{rem:norm}
It is well-known that, for finite tamely ramified extensions $E\supset K\supset F$, the norm map $\Nr_{E/K}$ gives a surjection
\[
\Nr_{E/K}\colon U_{E}^{e(E/K)s}\twoheadrightarrow U_{K}^{s}
\]
for any integer $s\in\Z_{>0}$ (see, for example, \cite[Chapter V, Section 6, Corollary 3]{MR554237}).
If we use the above normalizations of filtrations of multiplicative groups, this can be expressed as: for any $r\in\R_{>0}$, we have
\[
\Nr_{E/K}\colon E_{r}^{\times}\twoheadrightarrow K_{r}^{\times}.
\]
\end{rem}

\item[Level and depth of a multiplicative character]
For a finite tamely ramified extension $E$ of $F$ and a multiplicative character $\xi$ of $E^{\times}$, we say
\begin{itemize}
\item
$\xi$ is of $E$-level $s$ if $\xi|_{U_{E}^{s+1}}$ is trivial and $\xi|_{U_{E}^{s}}$ is nontrivial, and
\item
$\xi$ is of depth $r$ if $\xi|_{E_{r+}^{\times}}$ is trivial and $\xi|_{E_{r}^{\times}}$ is nontrivial.
\end{itemize}
Note that if we put the depth of $\xi$ to be $r$, then its $E$-level is given by $e(E/F)r$.

\item[Additive character]
Throughout this paper, we fix an additive character $\psi_{F}$ on $F$ of level one, i.e., $\psi_{F}$ is trivial on $\mfp_{F}$, but not trivial on $\mcO_{F}$.
For a finite tamely ramified extension $E$ of $F$, let $\psi_{E}$ denote the additive character $\psi_{F}\circ\Tr_{E/F}$ on $E$.
Note that, by the tameness assumption, this is again of level one, hence induces a nontrivial additive character of the residue field $k_{E}$ of $E$.
By abuse of notation, we again write $\psi_{E}$ for this induced additive character of $k_{E}$.

\item[Quadratic character]
For a finite cyclic group $C$ of even order, we write 
\[
\begin{pmatrix}
\frac{\cdot}{C}
\end{pmatrix}
\]
for the unique nontrivial quadratic character of $C$.
Note that if we identify $C$ as a subgroup of $\C^{\times}$, then we have an equality of sign
\[
\begin{pmatrix}
\frac{c}{C}
\end{pmatrix}
=
c^{\frac{|C|}{2}}
\]
for any $c\in C$.

\item[Jacobi symbol]
Let
\[
\begin{pmatrix}
\frac{\cdot}{\cdot}
\end{pmatrix}
\]
denote the Jacobi symbol.
We note that, for any finite field $k$ with odd order $q_{k}$, we have $\begin{pmatrix}\frac{m}{k^{\times}}\end{pmatrix}=\begin{pmatrix}\frac{m}{q_{k}}\end{pmatrix}$ for any integer $m$ which is prime to $q_{k}$ (here the left-hand side is the quadratic character defined above and the right-hand side is the Jacobi symbol).

\item[Gauss sum]
For a nontrivial additive character $\psi$ of a finite field $k$ with odd order $q_{k}$, we define its Gauss sum $\mfg(\psi)$ by
\[
\mfg(\psi)
:=
\sum_{x\in k^{\times}}\begin{pmatrix}\frac{x}{k^{\times}}\end{pmatrix}\psi(x).
\]
We define the normalized Gauss sum $\mfn(\psi)$ by
\[
\mfn(\psi)
:=
q_{k}^{-\frac{1}{2}}\mfg(\psi).
\]
Recall that we have the following well-known identity:
\[
\mfn(\psi)^{2}
=
\begin{pmatrix}
\frac{-1}{k^{\times}}
\end{pmatrix}
=
\begin{pmatrix}
\frac{-1}{q_{k}}
\end{pmatrix}
\]
(see, e.g., \cite[(23.6.3)]{MR2234120}).

\item[Langlands constant]
For a finite extension $E/K$ of $p$-adic fields and a nontrivial additive character $\psi$ of $K$, we write $\lambda_{E/K}(\psi)$ for the Langlands constant, that is, the ratio of local root numbers:
\[
\lambda_{E/K}(\psi)
:=
\frac{\varepsilon(\Ind_{W_{E}}^{W_{K}}\mathbbm{1}_{E},\frac{1}{2},\psi)}{\varepsilon(\mathbbm{1}_{E},\frac{1}{2},\psi\circ\Tr_{E/K})}.
\]
When $K$ is a finite extension of the fixed $p$-adic field $F$ and $\psi$ is given by $\psi_{K}(=\psi_{F}\circ\Tr_{K/F})$, we simply write $\lambda_{E/K}$ for $\lambda_{E/K}(\psi_{K})$.

\item[General linear group]
In this paper, we always write $\G$ for the general linear group $\GL_{n}$ over $F$.
Note that the Langlands dual group $\widehat{\G}$ is then given by $\GL_{n}$ over the complex number field $\C$.
We fix an $F$-splitting $\spl_{\G}:=(\bfB,\bfT,\{X\})$ of $\G$ and a splitting $\spl_{\widehat{\G}}:=(\B,\mcT,\{\mathcal{X}\})$ of $\widehat{\G}$.
See, for example, \cite[Section 1.1]{MR1687096} for the precise definition of a splitting.
We also fix an isomorphism between the based root data of $\G$ determined by $\spl_{\G}$ and the dual of that of $\widehat{\G}$ determined by $\spl_{\widehat{\G}}$.

\end{description}

\section{Description of Harris--Taylor's LLC due to Bushnell--Henniart and Tam}\label{sec:HT}

In this section, we recall Bushnell--Henniart's work on Harris--Taylor's local Langlands correspondence for $\GL_{n}$ and its reinterpretation due to Tam.

\subsection{$\chi$-data and Langlands--Shelstad's construction of $L$-embeddings}\label{subsec:LS}

Let $\bfS$ be a maximal torus of $\G$ defined over $F$.
Then we have the set $\Phi(\bfS,\G)$ of (absolute) roots of $\bfS$ in $\G$, which has an action of $\Gamma_{F}$.
We recall the following terminology of Adler--DeBacker--Spice for roots in $\Phi(\bfS,\G)$:
\begin{defn}\label{defn:root}
For each $\alpha\in\Phi(\bfS,\G)$, we put $\Gamma_{\alpha}$ (resp.\ $\Gamma_{\pm\alpha}$) to be the stabilizer of $\alpha$ (resp.\ $\{\pm\alpha\}$) in $\Gamma_{F}$.
Let $F_{\alpha}$ (resp.\ $F_{\pm\alpha}$) be the subfield of $\ol{F}$ fixed by $\Gamma_{\alpha}$ (resp.\ $\Gamma_{\pm\alpha}$): 
\[
F\subset F_{\pm\alpha}\subset F_{\alpha}
\quad
\longleftrightarrow
\quad
\Gamma_{F}\supset \Gamma_{\pm\alpha}\supset \Gamma_{\alpha}.
\]
\begin{itemize}
\item
When $F_{\alpha}=F_{\pm\alpha}$, we say $\alpha$ is an \textit{asymmetric} root.
\item
When $F_{\alpha}\supsetneq F_{\pm\alpha}$, we say $\alpha$ is a \textit{symmetric} root.
Note that, in this case, the extension $F_{\alpha}/F_{\pm\alpha}$ is necessarily quadratic.
Furthermore,
\begin{itemize}
\item
when $F_{\alpha}/F_{\pm\alpha}$ is unramified, we say $\alpha$ is \textit{symmetric unramified}, and
\item
when $F_{\alpha}/F_{\pm\alpha}$ is ramified, we say $\alpha$ is \textit{symmetric ramified}.
\end{itemize}
\end{itemize}
Note that a root $\alpha$ is symmetric if and only if the $\Gamma_{F}$-orbit of $\alpha$ contains $-\alpha$.
We write $\Phi(\bfS,\G)^{\sym}$, $\Phi(\bfS,\G)_{\sym}$, $\Phi(\bfS,\G)_{\sym,\ur}$, and $\Phi(\bfS,\G)_{\sym,\ram}$ for the set of asymmetric roots, symmetric roots, symmetric unramified roots, and symmetric ramified roots, respectively.
\end{defn}

\begin{rem}\label{rem:root}
In this paper, we follow the definition of symmetric unramified and ramified roots given by Adler--DeBacker--Spice (see \cite[Definition 2.6]{MR3849622} or also \cite[Section 2.2]{MR4013740}).
Note that this definition is different from Tam's one used in his paper (see \cite[1710 page]{MR3509939}).
In Section \ref{subsec:root}, we investigate the relation between Tam's notion of symmetric unramified and ramified roots and Adler--DeBacker--Spice's one in the case of $\GL_{n}$.
\end{rem}

We next recall the notion of $\chi$-data.
A set of characters $\{\chi_{\alpha}\colon F_{\alpha}^{\times}\rightarrow \C^{\times}\}_{\alpha\in\Phi(\bfS,\G)}$ indexed by the elements of $\Phi(\bfS,\G)$ is called \textit{a set of $\chi$-data} if it satisfies the following conditions:
\begin{itemize}
\item
For every $\alpha\in\Phi(\bfS,\G)$ and $\sigma\in\Gamma_{F}$, we have $\chi_{\alpha}^{-1}=\chi_{-\alpha}$ and $\chi_{\sigma(\alpha)}=\chi_{\alpha}\circ\sigma^{-1}$.
\item
For every $\alpha\in\Phi(\bfS,\G)_{\sym}$, the restriction $\chi_{\alpha}|_{F_{\pm\alpha}^{\times}}$ of $\chi_{\alpha}$ to $F_{\pm\alpha}^{\times}$ is equal to the quadratic character corresponding to the quadratic extension $F_{\alpha}/F_{\pm\alpha}$.
\end{itemize}

By using a set of $\chi$-data with respect to $(\bfS,\G)$, we can construct an embedding of the $L$-group ${}^{L}\bfS$ into that ${}^{L}\G$ of $\G$ in  the following manner.
First, we take an element $g\in\G(\ol{F})$ which induces an isomorphism $\Int(g)\colon \bfS\cong\bfT$ between $\bfS$ and the maximal torus $\bfT$ in the fixed splitting $\spl_{\G}$ by conjugation (note that this isomorphism is over $\ol{F}$).
Then, by considering its dual, we get an isomorphism between the dual torus $\widehat{\bfS}$ of $\bfS$ and $\mcT$ belonging to the splitting $\spl_{\widehat{\G}}$.
We write $\widehat{j}$ for the embedding 
\[
\widehat{j}\colon\widehat{\bfS}\cong\mcT\hookrightarrow\widehat{\G}
\]
obtained in this way (note that this embedding is defined canonically up to $\widehat{\G}$-conjugation).
Then, if we have a set $\chi:=\{\chi_{\alpha}\}$ of $\chi$-data, we can extend $\widehat{j}$ to an $L$-embedding according to Langlands--Shelstad's construction.
More precisely, we get a $\widehat{\G}$-conjugacy class of $L$-embeddings of ${}^{L}\bfS$ into ${}^{L}\G$.
Here we do not recall the construction of this extension.
See \cite[Section 2.6]{MR909227} for the details (see also \cite[Section 6.2]{MR3509939}).
We write ${}^{L}\!j_{\chi}$ for this extended $L$-embedding
\[
{}^{L}\!j_{\chi}\colon{}^{L}\bfS\hookrightarrow{}^{L}\G
\]
and call it \textit{Langlands--Shelstad's $L$-embedding} determined by the $\chi$-data $\chi=\{\chi_{\alpha}\}$.

\begin{rem}
Strictly speaking, this $L$-embedding ${}^{L}\!j_{\chi}$ is well-defined only up to $\widehat{\G}$-conjugation.
However, since our purpose is to construct an $L$-parameter of $\G$ by using this $L$-embedding, we do not have to take care of this ambiguity.
More concretely, taking $\widehat{\G}$-conjugation does not change the isomorphism class of an $n$-dimensional Weil--Deligne representation.
\end{rem}

\subsection{Tam's description of Langlands--Shelstad's $L$-embeddings}\label{subsec:Tam}
Let us assume that $\bfS$ is elliptic in $\G$.
Then it is well-known that $\bfS$ corresponds to a finite extension $E$ of $F$  which is of degree $n$, that is, $\bfS$ is isomorphic to $\Res_{E/F}\Gm$.
Recall that, by using this field $E$, we can describe the set $\Gamma_{F}\backslash\Phi(\bfS,\G)$ of $\Gamma_{F}$-orbits of absolute roots of $\bfS$ in $\G$ in the following manner (see \cite[Section 3.1]{MR3509939} for the details).
First we fix a set $\{g_{1},\ldots,g_{n}\}$ of representatives of the quotient $\Gamma_{F}/\Gamma_{E}$ such that $g_{1}=\id$.
Then we get an isomorphism
\[
\bfS(\overline{F})
\xrightarrow{\cong}
\prod_{i=1}^{n} \overline{F}^{\times}
\]
which maps $x\in E^{\times}\cong\bfS(F)$ to $(g_{1}(x),\ldots,g_{n}(x))$.
Then the projections 
\[
\delta_{i}
\colon 
\bfS(\overline{F})\xrightarrow{\cong}\prod_{i=1}^{n} \overline{F}^{\times}\rightarrow \overline{F}^{\times}
;\quad
(x_{1},\ldots,x_{n})\mapsto x_{i}
\]
form a $\Z$-basis of the group $X^{\ast}(\bfS)$ of (absolute) characters of $\bfS$.
Under this usage of notation, the set $\Phi(\bfS,\G)$ of absolute roots of $\bfS$ in $\G$ is given by
\[
\biggl\{
\begin{bmatrix}g_{i}\\g_{j}\end{bmatrix}:=\delta_{i}-\delta_{j}
\,\bigg\vert\,
1\leq i \neq j \leq n
\biggr\}
\]
and the set $\Gamma_{F}\backslash\Phi(\bfS,\G)$ is described as follows:
\[
(\Gamma_{E}\backslash\Gamma_{F}/\Gamma_{E})'
\xrightarrow{1:1}
\Gamma_{F}\backslash\Phi(\bfS,\G)
;\quad
\Gamma_{E}g_{i}\Gamma_{E}\mapsto \Gamma_{F}\cdot\begin{bmatrix}1\\g_{i}\end{bmatrix},
\tag{$\ast$}
\]
where $(\Gamma_{E}\backslash\Gamma_{F}/\Gamma_{E})'$ is the set of nontrivial double-$\Gamma_{E}$-cosets in $\Gamma_{F}$.

If we fix a subset $\mathcal{D}$ of $\{g_{1},\ldots,g_{n}\}$ such that the double-$\Gamma_{E}$-cosets represented by elements of $\mathcal{D}$ map to $\Gamma_{F}\backslash\Phi(\bfS,\G)$ bijectively under the map $(\ast)$, then we have
\[
\Gamma_{F}\backslash\Phi(\bfS,\G)
=\biggl\{
\Gamma_{F}\cdot\begin{bmatrix}1\\g\end{bmatrix}
\,\bigg\vert\,
g\in\mathcal{D}
\biggr\}.
\]
In this paper, we choose such a $\mathcal{D}$ in the same way as explained in \cite[1708 page]{MR3509939}.

Now let us recall Tam's description of Langlands--Shelstad's $L$-embeddings determined by $\chi$-data.
We suppose that we have a set $\chi$ of $\chi$-data with respect to $\bfS$.
Then, as we explained in Section \ref{subsec:LS}, we obtain an $L$-embedding ${}^{L}\!j_{\chi}\colon{}^{L}\bfS\hookrightarrow{}^{L}\G$ from $\chi$.
On the other hand, if we have a character $\xi$ of $E^{\times} (\cong\bfS(F))$, then we can attach to it an $L$-parameter $\phi_{\xi}\colon W_{F}\rightarrow{}^{L}\bfS$ by the local Langlands correspondence for tori (see, for example, \cite{MR2508725}).
By composing it with the $L$-embedding ${}^{L}\!j_{\chi}$, we may regard it as an $L$-parameter of $\G$:
\[
{}^{L}\!j_{\chi}\circ\phi_{\xi}
\colon
W_{F} \xrightarrow{\phi_{\xi}} {}^{L}\bfS \xrightarrow{{}^{L}\!j_{\chi}} {}^{L}\G.
\]
As the $L$-group ${}^{L}\G$ of $\G$ is given by $\GL_{n}(\C)\times W_{F}$, by considering the projection from ${}^{L}\G$ to $\GL_{n}(\C)$, we may furthermore regard ${}^{L}\!j_{\chi}\circ\phi_{\xi}$ as an $n$-dimensional representation of $W_{F}$.
Then it can be described as follows:
\begin{prop}[{\cite[Proposition 6.5]{MR3509939}}]\label{prop:Tam}
We put $\mu_{\chi}$ to be the character of $E^{\times}$ defined by
\[
\mu_{\chi}:=
\prod_{[\alpha]\in\Gamma_{F}\backslash\Phi(\bfS,\G)} \chi_{\alpha}|_{E^{\times}}.
\]
Here, for each $[\alpha]\in\Gamma_{F}\backslash\Phi(\bfS,\G)$, we take its representative $\alpha\in\Phi(\bfS,\G)$ to be of the form $\begin{bmatrix}1\\g\end{bmatrix}$ for some (unique) $g\in\mathcal{D}$.
Then the $L$-parameter ${}^{L}\!j_{\chi}\circ\phi_{\xi}$ is isomorphic to $\Ind_{W_{E}}^{W_{F}} (\xi\mu_{\chi})$ as an $n$-dimensional representation of $W_{F}$.
\end{prop}

Here we note that, if we suppose that a root $\alpha\in\Phi(\bfS,\G)$ is of the form $\begin{bmatrix}g_{i}\\g_{j}\end{bmatrix}=\delta_{i}-\delta_{j}$, then the field $F_{\alpha}$ is given by the composite field $g_{i}(E)\cdot g_{j}(E)$ of $g_{i}(E)$ and $g_{j}(E)$.
In particular, under the choice of a representative $\alpha$ of $[\alpha]$ as in Proposition \ref{prop:Tam}, the source $F_{\alpha}^{\times}$ of the character $\chi_{\alpha}$ contains $E^{\times}$.
Thus restricting $\chi_{\alpha}$ to $E^{\times}$ makes sense.

\subsection{Bushnell--Henniart's essentially tame LLC}\label{subsec:BH}
In \cite{MR1876802}, Harris and Taylor (and also Henniart in \cite{MR1738446}) constructed a bijective map, which is called the local Langlands correspondence for $\GL_{n}$, from the set $\Pi(\G)$ of equivalence classes of irreducible smooth representations of $\G(F)$ to the set $\Phi(\G)$ of $\widehat{\G}$-conjugacy classes of $L$-parameters of $\G$ (or equivalence classes of $n$-dimensional semisimple Weil--Deligne representations).
We write $\LLC_{\G}^{\HT}$ for this map:
\[
\LLC_{\G}^{\HT}\colon\Pi(\G)\xrightarrow{1:1}\Phi(\G).
\]
Then Bushnell--Henniart's \textit{essentially tame local Langlands correspondence} gives an explicit description of this map $\LLC_{\G}^{\HT}$ for so-called ``essentially tame supercuspidal representations''.
Let us recall what these are briefly.

First, from an $F$-admissible pair $(E,\xi)$, which is a pair of 
\begin{itemize}
\item
a finite tamely ramified extension $E$ of $F$ of degree $n$ and
\item
an $F$-admissible character $\xi$ (see \cite[Definition in 686 page]{MR2138141} for the definition of the $F$-admissibility of a character),
\end{itemize}
we obtain an irreducible supercuspidal representation $\pi_{(E,\xi)}^{\BH}$ of $\G(F)$ according to the construction of Bushnell--Henniart (see \cite[Section 2]{MR2138141}, \cite[Section 5]{MR3509939}, or Appendix of this paper).
We call a representation obtained in this way an \textit{essentially tame supercuspidal representation}.

\begin{rem}
Precisely speaking, Bushnell and Henniart first defined essentially tame supercuspidal representations by using of the notion of the torsion number of supercuspidal representations (see the beginning of \cite[Section 2]{MR2138141}).
Then they gave a characterization of essentially tame supercuspidal representations in terms of $F$-admissible pairs, that is, they proved that essentially tame supercuspidal representations are parametrized by $F$-admissible pairs (\cite[Theorem 2.3]{MR2138141}).
In this paper we do not go back to the original definition since only the latter characterization is important from our viewpoint.
\end{rem}

For every essentially tame supercuspidal representation, Bushnell--Henniart proved that its image under the map $\LLC_{\G}^{\HT}$ can be described by using a character $\mu_{\rec}$ which depends on each $F$-admissible pair $(E,\xi)$ and is called the \textit{rectifier} of $(E,\xi)$ as follows:

\begin{thm}[{\cite{MR2138141, MR2148193, MR2679700}}]\label{thm:BH}
There exists a tamely ramified character $\mu_{\rec}$ of $E^{\times}$, which can be described explicitly in terms of $E$ and $\xi$, satisfying 
\[
\LLC_{\G}^{\HT}(\pi_{(E,\xi\mu_{\rec})}^{\BH})
=
\Ind_{W_{E}}^{W_{F}}\xi,
\]
or equivalently, 
\[
\LLC_{\G}^{\HT}(\pi_{(E,\xi)}^{\BH})
=
\Ind_{W_{E}}^{W_{F}}(\xi\mu_{\rec}^{-1}).
\]
\end{thm}

Note that although we simply write $\mu_{\rec}$ for the rectifier, it is defined for each $F$-admissible pair $(E,\xi)$ and depends on $(E,\xi)$.

What Tam noticed is that this mysterious character $\mu_{\rec}$ of $E^{\times}$ has a factorization in terms of $\chi$-data.
Let us take a tamely ramified elliptic maximal torus $\bfS$ of $\G$ which is isomorphic to $\Res_{E/F}\Gm$.
Then Tam produced a set $\chi_{\Tam}=\{\chi_{\Tam,\alpha}\}_{\alpha}$ of $\chi$-data with respect to $(\bfS,\G)$ from an $F$-admissible pair $(E,\xi)$ and proved the following:

\begin{thm}[{\cite[Theorem 7.1]{MR3509939}}]\label{thm:Tam}
Let $\mu_{\rec}$ be the rectifier with respect to an $F$-admissible pair $(E,\xi)$.
Then we have
\[
\mu_{\rec}
=
\mu_{\chi_{\Tam}}.
\]
Here the right-hand side is a character of $E^{\times}$ determined by $\chi_{\Tam}$ in the manner of Proposition \ref{prop:Tam}.
\end{thm}

The definition of Tam's $\chi$-data $\chi_{\Tam}$ is given in \cite[Theorem 7.1]{MR3509939}.
Later we will recall it precisely (Sections \ref{subsec:asymm}, \ref{subsec:symmur}, and \ref{subsec:symmram}).
As a consequence of Theorems \ref{thm:BH} and \ref{thm:Tam}, we get the following corollary:

\begin{cor}\label{cor:BH}
Let $\bfS$ be an elliptic maximal torus of $\G$ which is isomorphic to $\Res_{E/F}\Gm$ for a finite tamely ramified extension $E$ of $F$ of degree $n$.
Let $\xi$ be an $F$-admissible character of $\bfS(F)\cong E^{\times}$ and $\pi_{(E,\xi)}^{\BH}$ the irreducible supercuspidal representation of $\GL_{n}(F)$ arising from the pair $(E,\xi)$.
Then we have
\[
\LLC_{\G}^{\HT}(\pi_{(E,\xi)}^{\BH})
=
\Ind_{W_{E}}^{W_{F}}(\xi\mu_{\chi_{\Tam}}^{-1}).
\]
\end{cor}

\section{Kaletha's LLC for regular supercuspidal representations}\label{sec:Kaletha}
In this section, we quickly review Kaletha's construction of the local Langlands correspondence in the case of $\GL_{n}$.
In \cite{MR4013740}, he constructed the local Langlands correspondence for supercuspidal representations satisfying some regularity condition, which are called ``regular supercuspidal representations'', of tamely ramified connected reductive groups, under a mild restriction on the residual characteristic $p$.
In particular, his construction works for $\G=\GL_{n}$ whenever $p$ is an odd prime number.

\subsection{Regular supercuspidal representations}\label{subsec:rsc}
In this section, let us temporarily consider a general situation where $\G$ is a tamely ramified connected reductive group over $F$.
In general, \textit{regular supercuspidal representations} are defined to be supercuspidal representations obtained from ``\textit{regular Yu-data}'' according to Yu's construction.
Here recall that a ``(regular) reduced generic cuspidal $\G$-datum'' consists of 
\begin{itemize}
\item
a sequence $\G^{0}\subsetneq\G^{1}\subsetneq\cdots\subsetneq\G^{d}=\G$ of tame twisted Levi subgroups of $\G$ over $F$,
\item
an irreducible depth zero supercuspidal representation $\pi_{-1}$ of $\G^{0}(F)$, and
\item
characters $\phi_{i}\colon\G^{i}(F)\rightarrow\C^{\times}$ for $0\leq i\leq d$
\end{itemize}
satisfying several conditions.
We do not recall the precise definition of a (regular) reduced generic cuspidal $\G$-datum or Yu's construction of supercuspidal representations.
See \cite[Section 3.5 and Definition 3.7.3]{MR4013740} for the definition of a (regular) reduced generic cuspidal $\G$-datum.
In the following, we call a (regular) reduced generic cuspidal $\G$-datum simply a (regular) Yu-datum.
For a given Yu-datum $\Psi=(\G^{0}\subsetneq\G^{1}\subsetneq\cdots\subsetneq\G^{d},\pi_{-1},(\phi_{0},\ldots,\phi_{d}))$, let $\pi_{\Psi}^{\Yu}$ denote the irreducible supercuspidal representation obtained from $\Psi$ by Yu's construction.

One of the important discoveries of Kaletha is the following fact saying that regular Yu-data bijectively correspond to ``\textit{tame elliptic regular pairs}'', which are much simpler objects consisting only of tamely ramified elliptic tori and their characters (see \cite[Definition 3.7.5]{MR4013740} for the precise definition of a tame elliptic regular pair):

\begin{prop}[{\cite[Proposition 3.7.8]{MR4013740}}]\label{prop:reparametrization}
Suppose that we have a regular Yu-datum $(\G^{0}\subsetneq\G^{1}\subsetneq\cdots\subsetneq\G^{d},\pi_{-1},(\phi_{0},\ldots,\phi_{d}))$.
Then we can construct a tamely ramified elliptic maximal torus $\bfS$ of $\G$ contained in $\G^{0}$ and a character $\phi_{-1}$ on it such that the pair $(\bfS,\prod_{i=-1}^{d}\phi_{i}|_{\bfS(F)})$ is tame elliptic regular.
Moreover, this procedure
\[
\bigl(\G^{0}\subsetneq\G^{1}\subsetneq\cdots\subsetneq\G^{d},\pi_{-1},(\phi_{0},\ldots,\phi_{d})\bigr)
\mapsto
\biggl(\bfS,\prod_{i=-1}^{d}\phi_{i}|_{\bfS(F)}\biggr)
\]
gives a bijection from the set of $\G$-equivalence classes of regular Yu-data to the set of $\G(F)$-conjugacy classes of tame elliptic regular pairs.
\end{prop}

Here, ``$\G$-equivalence'' is the equivalence relation for Yu-data introduced by Hakim--Murnaghan in \cite{MR2431732}.
Its important nature is that it describes the ``fibers'' of Yu's construction, that is, for any two Yu-data $\Psi_{1}$ and $\Psi_{2}$, we have $\pi_{\Psi_{1}}^{\Yu}\cong\pi_{\Psi_{2}}^{\Yu}$ if and only if $\Psi_{1}$ is $\G$-equivalent to $\Psi_{2}$.
See \cite[Definition 6.3]{MR2431732} or \cite[Section 3.5]{MR4013740} for the definition of $\G$-equivalence.
In particular, by Proposition \ref{prop:reparametrization} and the definition of regular supercuspidal representations, we can parametrize the set of equivalence classes of regular supercuspidal representations by the set of $\G(F)$-conjugacy classes of tame elliptic regular pairs of $\G$.
When a regular supercuspidal representation $\pi$ arises from a tame elliptic regular pair $(\bfS,\xi)$, we write $\pi=\pi_{(\bfS,\xi)}^{\KY}$.
Thus, if $(\bfS,\xi)$ corresponds to a $\G$-equivalence class represented by a Yu-datum $\Psi$, then we have $\pi_{(\bfS,\xi)}^{\KY}=\pi_{\Psi}^{\Yu}$ by definition.
\[
\xymatrix{
&\{\text{irred.\ s.c.\ rep'ns of $\G(F)$}\}/{\sim}\\
\{\text{Yu-data}\}/\text{$\G$-eq.} \ar@{->}[r]^-{1:1}_-{\text{Yu's construction}} & \{\text{Yu's s.c.\ rep'ns of $\G(F)$}\}/{\sim}\ar@{^{(}->}[u] \\
\{\text{regular Yu-data}\}/\text{$\G$-eq.}\ar@{^{(}->}[u] \ar@{->}[r]^-{1:1} & \{\text{regular s.c.\ rep'ns of $\G(F)$}\}/{\sim}\ar@{^{(}->}[u]\\
\{\text{tame elliptic regular pairs}\}/\text{$\G(F)$-conj.}\ar@{<->}[u]^-{1:1}_-{\text{(Prop.\ \ref{prop:reparametrization})}} \ar@{->}[ur]_-{\quad\quad\quad(\bfS,\xi)\mapsto\pi^{\KY}_{(\bfS,\xi)}}&
}
\]

The key point which we have to keep in our mind here is the following connection with Bushnell--Henniart's construction in the case where $\G=\GL_{n}$.
In this case, since every tamely ramified elliptic maximal torus of $\G$ is isomorphic to $\Res_{E/F}\Gm$ for some finite tamely ramified extension $E$ of $F$ of degree $n$, we obtain a pair $(E,\xi)$ of such a field $E$ and a character $\xi$ on $E^{\times}$ from each tame elliptic regular pair $(\bfS,\xi)$ of $\G$.
In fact, the pair $(E,\xi)$ obtained in this way is $F$-admissible and, conversely, every $F$-admissible pair is obtained from a tame elliptic regular pair (see \cite[Lemma 3.7.7]{MR4013740}).
Here note that the proof of \cite[Lemma 3.7.7]{MR4013740} works for any odd prime number $p$ although the condition $p\nmid n$ is imposed in \cite[Lemma 3.7.7]{MR4013740}.
Then the supercuspidal representation $\pi_{(E,\xi)}^{\BH}$ constructed from this pair $(E,\xi)$ in Bushnell--Henniart's way coincides with $\pi_{(\bfS,\xi)}^{\KY}$.
Although this coincidence seems to be well-known to experts, we will explain a proof in Appendix \ref{sec:BH-Kal} for the sake of completeness.

In the following, when a tame elliptic regular pair $(\bfS,\xi)$ of $\GL_{n}$ corresponds to an $F$-admissible pair $(E,\xi)$, we simply write
\[
\pi_{(\bfS,\xi)}
:=\pi_{(\bfS,\xi)}^{\KY}
\cong\pi_{(E,\xi)}^{\BH}.
\]

\begin{rem}
When $p$ does not divide $n$, every irreducible supercuspidal representation of $\GL_{n}(F)$ is essentially tame.
This is almost trivial if we see the original definition of essentially tame supercuspidal representations via the torsion number (see \cite[Section 2]{MR2138141}).
\end{rem}

\subsection{DeBacker--Spice's sign and Kaletha's toral invariant}\label{subsec:epsilon}
In this section, we recall two invariants obtained from Yu-data and used in Kaletha's construction of the local Langlands correspondence for regular supercuspidal representations.
(Also in this section, we can take $\G$ to be a general tamely ramified connected reductive group.)

Let us take a Yu-datum $(\G^{0}\subsetneq\G^{1}\subsetneq\cdots\subsetneq\G^{d},\pi_{-1},(\phi_{0},\ldots,\phi_{d}))$ of $\G$.
Let $\bfS$ be an elliptic maximal torus of $\G^{0}$ obtained by Proposition \ref{prop:reparametrization}.

The first invariant which we need is DeBacker--Spice's sign character $\epsilon^{\sym}\cdot\epsilon_{\sym,\ur}$ introduced in \cite[Section 4.3]{MR3849622}.
This is the product of two characters $\epsilon^{\sym}$ and $\epsilon_{\sym,\ur}$ of $\bfS(F)$.
In order to define these two characters, we first define a character $\epsilon_{\alpha}$ of $\bfS(F)$ for an asymmetric or symmetric unramified root $\alpha$ of $\bfS$ in $\G$.
For a root $\alpha\in\Phi(\bfS,\G)$, there exists a unique index $-1\leq i\leq d-1$ such that $\alpha$ belongs to
\[
\Phi_{i+1}^{i}:=\Phi(\bfS,\G^{i+1})-\Phi(\bfS,\G^{i}).
\]
Here we put $\G^{-1}:=\bfS$.
When $i$ is equal to $-1$, we define $\epsilon_{\alpha}$ to be the trivial character $\mathbbm{1}$ of $\bfS(F)$.
When we have $0\leq i\leq d-1$, we first put $r_{i}$ to be the depth of the character $\phi_{i}$.
On the other hand, for each $\alpha\in\Phi(\bfS,\G)$, we can define a set $\ord_{x}(\alpha)$ of real numbers as in \cite[Definition 3.6]{MR3849622} (we will recall its definition later, see Definition \ref{defn:ord}).
Then the character $\epsilon_{\alpha}\colon\bfS(F)\rightarrow\C^{\times}$ is defined as
\[
\epsilon_{\alpha}(\gamma)
:=
\begin{cases}
\Jac{\overline{\alpha(\gamma)}}{k_{F_{\alpha}}^{\times}} &\text{if $\frac{r_{i}}{2}\in\ord_{x}(\alpha)$,}\\
1&\text{otherwise,}
\end{cases}
\]
when $\alpha$ is asymmetric, and
\[
\epsilon_{\alpha}(\gamma)
:=
\begin{cases}
\Jac{\overline{\alpha(\gamma)}}{k_{F_{\alpha}}^{1}}&\text{if $\frac{r_{i}}{2}\in\ord_{x}(\alpha)$,}\\
1&\text{otherwise,}
\end{cases}
\]
when $\alpha$ is symmetric unramified.
Here, in both cases, $\alpha(\gamma)$ belongs to $\mcO_{F_{\alpha}}^{\times}$ by the ellipticity of $\bfS$.
Thus we may consider its reduction by $\mfp_{F_{\alpha}}$.
We let $\overline{\alpha(\gamma)}$ denote its image in $k_{F_{\alpha}}^{\times}$.
Note that, in the latter case, the definition makes sense since $\alpha(\gamma)$ always belongs to the kernel of the norm map $\Nr_{F_{\alpha}/F_{\pm\alpha}}\colon F_{\alpha}^{\times}\rightarrow F_{\pm\alpha}^{\times}$.
Indeed, if we take the nontrivial element $\tau_{\alpha}$ of the Galois group of $F_{\alpha}/F_{\pm\alpha}$, then we have $\tau_{\alpha}(\alpha)=-\alpha$.
In other words, we have $\tau_{\alpha}(\alpha(\gamma))=\alpha(\tau_{\alpha}(\gamma))^{-1}$ for every $\gamma\in\bfS(F_{\alpha})$.
Thus, if $\gamma\in\bfS(F)$, then we get
\[
\Nr_{F_{\alpha}/F_{\pm\alpha}}\bigl(\alpha(\gamma)\bigr)
=
\alpha(\gamma)\cdot \tau_{\alpha}\bigl(\alpha(\gamma)\bigr)
=
\alpha(\gamma)\cdot \alpha\bigl(\tau_{\alpha}(\gamma)\bigr)^{-1}
=
\alpha(\gamma)\cdot\alpha(\gamma)^{-1}
=1.
\]
We define $\epsilon^{\sym}$ and $\epsilon_{\sym,\ur}$ as products of these $\epsilon_{\alpha}$'s:
\[
\epsilon^{\sym}
:=
\prod_{\alpha\in\Gamma_{F}\times\{\pm1\}\backslash\Phi(\bfS,\G)^{\sym}} \epsilon_{\alpha}
\quad\text{and}\quad
\epsilon_{\sym,\ur}
:=
\prod_{\alpha\in\Gamma_{F}\backslash\Phi(\bfS,\G)_{\sym,\ur}} \epsilon_{\alpha},
\]
where the index set of the first product denotes the set of orbits of asymmetric roots via the action of $\Gamma_{F}\times\{\pm1\}$ (the action of $-1\in\{\pm1\}$ is given by $\alpha\mapsto-\alpha$).

\begin{rem}
In DeBacker--Spice's original paper (\cite[Definition 4.14]{MR3849622}) or Kaletha's paper (\cite[1123 page, (4.3.3)]{MR4013740}), the character $\epsilon^{\sym}\cdot\epsilon_{\sym,\ur}$ is defined as the product of $\epsilon_{\alpha}$'s only for roots $\alpha$ belonging to $\Phi_{d}^{d-1}=\Phi(\bfS,\G^{d})-\Phi(\bfS,\G^{d-1})$.
This is because Adler--DeBacker--Spice's character formula has an inductive nature and there (\cite[Section 4.3]{MR3849622} or \cite[Section 4.3]{MR4013740}) they only describe a sign appearing in a ``reduction step from $\G^{d}$ to $\G^{d-1}$''.
Thus, in a final form of the full character formula (cf.\ \cite[Theorem 7.1]{MR2543925}), the product of $\epsilon_{\alpha}$'s over all roots $\alpha$ in $\Phi(\bfS,\G)$ as above appears.
Also, when this sign character is used in Kaletha's construction of regular supercuspidal $L$-packets (\cite[1154 page, Step 3]{MR4013740}, this will be explained in Section \ref{subsec:LLC-Kal}), the product is taken over all roots (i.e., the sign character $\epsilon^{\sym}\cdot\epsilon_{\sym,\ur}$ defined as in this paper is used).
\end{rem}

The second invariant which we have to consider is the character $\epsilon_{f,\ram}$ defined in \cite[Definition 4.7.3]{MR4013740}.
As explained in \cite[Lemma 4.7.4]{MR4013740}, this is a character of $\bfS(F)$ expressed as the product of ``\textit{toral invariants}'' $f_{(\G,\bfS)}$ for symmetric ramified roots, which are introduced in \cite[Section 4.1]{MR3402796}:
\[
\epsilon_{f,\ram}(\gamma)
=
\prod_{\begin{subarray}{c}\alpha\in\Gamma_{F}\backslash\Phi(\bfS,\G)_{\sym,\ram}\\ \alpha(\gamma)\neq1 \\ \ord(\alpha(\gamma)-1)=0\end{subarray}}
f_{(\G,\bfS)}(\alpha)
\quad
\text{for }
\gamma\in\bfS(F).
\]
In fact, as a special feature of general linear groups, we can show that this invariant is always trivial:

\begin{prop}\label{prop:toral}
When $\G=\GL_{n}$, for every symmetric root $\alpha\in\Phi(\bfS,\G)_{\sym}$, its toral invariant is trivial, i.e., $f_{(\G,\bfS)}(\alpha)=1$.
\end{prop}

\begin{proof}
We first recall the definition of the toral invariant.
Let $\alpha$ be a symmetric root of $\bfS$ in $\G$.
Then we have the corresponding quadratic extension $F_{\alpha}/F_{\pm\alpha}$.
We take an element $\tau_{\alpha}\in\Gamma_{\pm\alpha}-\Gamma_{\alpha}$.
If we take a nonzero root vector $X_{\alpha}\in\mfg_{\alpha}(F_{\alpha})$ for the root $\alpha$, then $\tau_{\alpha}(X_{\alpha})$ is a root vector for the root $\tau_{\alpha}(\alpha)=-\alpha$.
Thus their Lie bracket product $[X_{\alpha},\tau_{\alpha}(X_{\alpha})]$ gives a nonzero element of the maximal torus $\mfs(F_{\alpha})\subset\mfg(F_{\alpha})$.
On the other hand, the differential $d\alpha^{\vee}$ of the coroot $\alpha^{\vee}$ corresponding to $\alpha$ defines a nonzero element $H_{\alpha}:=d\alpha^{\vee}(1)$ of $\mfs(F_{\alpha})$.
By using these elements, we define the toral invariant $f_{(\G,\bfS)}(\alpha)$ of $\alpha$ by
\[
f_{(\G,\bfS)}(\alpha):=
\kappa_{\alpha}\biggl(\frac{[X_{\alpha},\tau_{\alpha}(X_{\alpha})]}{H_{\alpha}}\biggr),
\]
where $\kappa_{\alpha}$ is the unique quadratic character of $F_{\pm\alpha}^{\times}$ corresponding to the quadratic extension $F_{\alpha}/F_{\pm\alpha}$.
Thus, for our purpose, it suffices to find a root vector $X_{\alpha}$ satisfying $[X_{\alpha},\tau_{\alpha}(X_{\alpha})]=H_{\alpha}$.

Let $E$ be a tamely ramified extension of $F$ of degree $n$ corresponding to the tamely ramified elliptic maximal torus $\bfS$ of $\G$.
By taking an $F$-basis of $E$, we obtain an isomorphism
\[
\mfg(F)
=
\mathfrak{gl}_{n}(F)
\cong
\End_{F}(E).
\]
Then we may assume that $\mfs(F)\subset\mfg(F)$ corresponds to $E$, which is canonically embedded into $\End_{F}(E)$ as multiplication maps, on the right-hand side.
With this identification in our mind, in order to describe the (absolute) root space decomposition of $\mfg$ via $\bfS$, we consider the action of $(E\otimes_{F}\ol{F})^{\times}$ on $\End_{F}(E)\otimes_{F}{\ol{F}}$.

Recall that, in Section \ref{subsec:Tam}, we took a set $\{g_{1},\ldots,g_{n}\}$ of representatives of $\Gamma_{F}/\Gamma_{E}$ and that we have an isomorphism
\[
E\otimes_{F}\ol{F}
\xrightarrow{\cong}
\bigoplus_{i=1}^{n}\ol{F}
\colon\quad
x\otimes y
\mapsto
\bigl(g_{i}(x)y\bigr)_{i}.
\]
Thus the set of characters and cocharacters of $\bfS$ is described as
\[
X^{\ast}(\bfS)=\bigoplus_{i=1}^{n}\Z \delta_{i}
\quad\text{and}\quad
X_{\ast}(\bfS)=\bigoplus_{i=1}^{n}\Z \delta^{\vee}_{i},
\]
where $\delta_{i}$ and $\delta_{i}^{\vee}$ are given by (on $\ol{F}$-valued points)
\begin{align*}
\delta_{i}&\colon \prod_{i=1}^{n}\ol{F}^{\times}\rightarrow\ol{F}^{\times};\quad (x_{i})_{i}\mapsto x_{i},\\
\delta_{i}^{\vee}&\colon \ol{F}^{\times}\rightarrow\prod_{i=1}^{n}\ol{F}^{\times};\quad x\mapsto (1,\ldots,1,\underbrace{x}_{i},1,\ldots,1),
\end{align*}
respectively, and the set of roots of $\bfS$ in $\mfg$ is given by $\{\delta_{i}-\delta_{j}\mid 1\leq i\neq j\leq n\}$.
We put $d_{i}$ and $d_{i}^{\vee}$ to be their differentials, i.e., 
\begin{align*}
d_{i}&\colon \bigoplus_{i=1}^{n}\ol{F}\rightarrow\ol{F};\quad (x_{i})_{i}\mapsto x_{i},\\
d^{\vee}_{i}&\colon \ol{F}\rightarrow\bigoplus_{i=1}^{n}\ol{F};\quad x\mapsto (0,\ldots,0,\underbrace{x}_{i},0,\ldots,0).
\end{align*}
Then we can check that $d^{\vee}_{i}\circ d_{j}$ is a root vector for the root $\delta_{i}-\delta_{j}$.

Let us assume that our symmetric root $\alpha$ is given by $\delta_{i}-\delta_{j}$ and show that $X_{\alpha}:=d^{\vee}_{i}\circ d_{j}$ is a desired root vector.
First, this element $X_{\alpha}\in\mfg_{\alpha}(\ol{F})$ belongs to $\mfg_{\alpha}(F_{\alpha})$.
Indeed, in order to check this, it suffices to show that $\tau(X_{\alpha})=X_{\alpha}$ for every $\tau\in\Gamma_{\alpha}$.
By considering how the $\Gamma_{F}$-action on $E\otimes_{F}\ol{F}$ is converted on $\oplus_{i=1}^{n}\ol{F}$, we can easily check that the $\Gamma_{F}$-action on characters and cocharacters is given by $\tau(\delta_{i})=\delta_{\tau(i)}$ (hence $\tau(d_{i})=d_{\tau(i)}$) and $\tau(\delta^{\vee}_{i})=\delta^{\vee}_{\tau(i)}$ (hence $\tau(d^{\vee}_{i})= d^{\vee}_{\tau(i)}$).
Here $\tau(i)$ denotes the unique index satisfying $g_{\tau(i)}\Gamma_{E}=(\tau\circ g_{i})\Gamma_{E}$.
Thus, for $\tau\in\Gamma_{F}$, $\tau$ belongs to $\Gamma_{\alpha}=\Stab_{\Gamma_{F}}(\delta_{i}-\delta_{j})$ if and only if we have $\tau(i)=i$ and $\tau(j)=j$.
On the other hand, we can also check that $\tau(d^{\vee}_{i}\circ d_{j})=d^{\vee}_{\tau(i)}\circ d_{\tau(j)}$ easily.
Thus $\tau\in\Gamma_{\alpha}$ implies that $\tau(X_{\alpha})=X_{\alpha}$.

Now let us show that $[X_{\alpha}, \tau_{\alpha}(X_{\alpha})]=H_{\alpha}$.
By noting that $\mfs(\ol{F})$ is embedded into $\End_{F}(E)\otimes_{F}\ol{F}$ as
\begin{align*}
\mfs(\ol{F})
\cong
\bigoplus_{i=1}^{n}\ol{F}
&\hookrightarrow
\End_{\ol{F}}\Bigl(\bigoplus_{i=1}^{n}\ol{F}\Bigr)
\cong
\End_{F}(E)\otimes_{F}\ol{F}\\
(x_{i})_{i}
&\mapsto
\sum_{i=1}^{n} x_{i}\cdot(d^{\vee}_{i}\circ d_{i}),
\end{align*}
we get $H_{\alpha}=d_{i}^{\vee}\circ d_{i}-d_{j}^{\vee}\circ d_{j}$.
On the other hand, as we have 
\[
\tau_{\alpha}(\alpha)
=\delta_{\tau_{\alpha}(i)}-\delta_{\tau_{\alpha}(j)}
=-\alpha
=\delta_{j}-\delta_{i},
\]
we get $\tau_{\alpha}(i)=j$ and $\tau_{\alpha}(j)=i$.
Hence we have
\[
\tau_{\alpha}(X_{\alpha})
=
d^{\vee}_{\tau_{\alpha}(i)}\circ d_{\tau_{\alpha}(j)}
=d^{\vee}_{j}\circ d_{i}.
\]
By noting that $d_{i}\circ d^{\vee}_{i}=\id$ for any $i$, we get
\begin{align*}
[X_{\alpha},\tau_{\alpha}(X_{\alpha})]
&=
X_{\alpha}\circ\tau_{\alpha}(X_{\alpha})-\tau_{\alpha}(X_{\alpha})\circ X_{\alpha}\\
&=
d^{\vee}_{i}\circ d_{j}\circ d^{\vee}_{j}\circ d_{i}-d^{\vee}_{j}\circ d_{i}\circ d^{\vee}_{i}\circ d_{j}\\
&=
d^{\vee}_{i}\circ d_{i}-d^{\vee}_{j}\circ d_{j}.
\end{align*}
This completes the proof.
\end{proof}

\subsection{Kaletha's Construction of $L$-packets and $L$-parameters}\label{subsec:LLC-Kal}
In \cite{{MR4013740}}, Kaletha attached an $L$-packet and an $L$-parameter to each ``regular supercuspidal $L$-packet datum'' of $\G$.
Let us recall his construction briefly.
Here we suppose that $\G$ is a quasi-split tamely ramified connected reductive group temporarily (later we will focus on the case where $\G=\GL_{n}$).

First recall that a \textit{regular supercuspidal $L$-packet datum of $\G$} is a tuple $(\bfS,\widehat{j},\chi,\xi)$ consisting of 
\begin{itemize}
\item
a tamely ramified torus $\bfS$ over $F$ whose absolute rank is the same as that of $\G$,
\item
an embedding $\widehat{j}\colon\widehat{\bfS}\hookrightarrow\widehat{\G}$ whose $\widehat{\G}$-conjugacy class is $\Gamma_{F}$-stable,
\item
a set $\chi$ of $\chi$-data for $(\bfS,\G)$, and
\item
a character $\xi\colon\bfS(F)\rightarrow\C^{\times}$
\end{itemize}
satisfying several conditions (see \cite[Definition 5.2.4]{MR4013740}).
Here the meaning of ``a set of $\chi$-data for $(\bfS,\G)$'' is as follows (see \cite[Section 5.1]{MR4013740} for the details):
First, by considering the converse of the procedure in Section \ref{subsec:LS}, we can get a $\Gamma_{F}$-stable $\G(\overline{F})$-conjugacy class $J$ of embeddings of $\bfS$ into $\G$.
More precisely, by taking a conjugation, we may assume that the image $\widehat{j}(\widehat{\bfS})$ of $\widehat{j}$ in $\widehat{\G}$ is equal to $\mathcal{T}$, which is the maximal torus belonging to the fixed splitting $\mathbf{spl}_{\widehat{\G}}$ for $\widehat{\G}$.
Then, by taking the dual, we get an isomorphism between $\bfS$ and $\bfT$, which is the maximal torus belonging to $\mathbf{spl}_{\G}$, over $\overline{F}$.
Hence we get an embedding of $\bfS$ into $\G$ (which is over $\overline{F}$) well-defined up to $\G(\overline{F})$-conjugation.
In other words, we get a $\G(\overline{F})$-conjugacy class of embeddings of $\bfS$ into $\G$, for which we write $J$ (see \cite[Section 5.1]{MR4013740} for an explanation on the $\Gamma_{F}$-stability of $J$).
Then, as explained in \cite[Section 5.1]{MR4013740}, the class $J$ has a $\Gamma_{F}$-fixed element, i.e., an embedding of $\bfS$ into $\G$ defined over $F$ because of the quasi-splitness of $\G$ (\cite[Corollary 2.2]{MR683003}).
By taking such an embedding $j\colon\bfS\hookrightarrow\G$ defined over $F$, $j(\bfS)$ gives a maximal torus of $\G$ defined over $F$.
In particular, by pulling back the set $\Phi(j(\bfS),\G)$ via $j$, we get a subset $\Phi(\bfS,\G)$ of $X^{\ast}(\bfS)$ which is $\Gamma_{F}$-stable.
Since this set with $\Gamma_{F}$-action does not depend on the choice of $j$, we can define the notion of a set of $\chi$-data in the usual way (see Section \ref{subsec:LS}).
We note that, in the definition of a regular supercuspidal $L$-packet datum, an ellipticity condition on $\bfS$ is also imposed (i.e., if we regard $\bfS$ as a subtorus of $\G$ by using $j$ as above, then it should be elliptic).
We also note that we can extend $\widehat{j}$ to an $L$-embedding ${}^{L}\!j_{\chi}\colon{}^{L}\bfS\hookrightarrow{}^{L}\G$ by using a set of $\chi$-data $\chi$ in the same manner as in Section \ref{subsec:LS}.

Next let us recall the construction of $L$-packets and $L$-parameters.
We take a regular supercuspidal $L$-packet datum $(\bfS,\widehat{j},\chi,\xi)$ of $\G$.
First, we obtain an $L$-parameter from this regular supercuspidal $L$-packet datum $(\bfS,\widehat{j},\chi,\xi)$ just by using the local Langlands correspondence for tori:

\begin{description}
\item[Construction of $L$-parameters ({\cite[Proof of Proposition 5.2.7]{MR4013740}})]
By applying the local Langlands correspondence for $\bfS$ to $\xi$, we get an $L$-parameter $\phi_{\xi}$ of $\bfS$, which is a homomorphism from $W_{F}$ to ${}^{L}\bfS$.
On the other hand, by the Langlands--Shelstad construction, we can extend $\widehat{j}$ to an $L$-embedding ${}^{L}\!j_{\chi}$ from ${}^{L}\bfS$ to ${}^{L}\G$ by using the set $\chi$ of $\chi$-data.
Thus, by composing these two homomorphisms, we get an $L$-parameter $\phi$ of $\G$:
\[
\phi
\colon
W_{F}
\xrightarrow{\phi_{\xi}}
{}^{L}\bfS
\xrightarrow{{}^{L}\!j_{\chi}}
{}^{L}\G.
\]
\end{description}
Next we recall the construction of the $L$-packet attached to $(\bfS,\widehat{j},\chi,\xi)$:

\begin{description}
\item[Parametrization of members in each $L$-packet ({\cite[1154 page]{MR4013740}})]
As explained in the beginning of this section, from $\widehat{j}$, we get a $\Gamma_{F}$-stable $\G(\overline{F})$-conjugacy class $J$ of embeddings of $\bfS$ into $\G$.
We put $\mathcal{J}$ to be the set of $\G(F)$-conjugacy classes of embeddings which belong to $J$ and defined over $F$.
Then members of the $L$-packet $\Pi_{\phi}$ for the $L$-parameter $\phi$ are parametrized by the elements of $\mathcal{J}$.
\item[Construction of members ({\cite[1153-1154 page]{MR4013740})}]
For each $j\in\mathcal{J}$, we call a tuple $(\bfS,\widehat{j},\chi,\xi,j)$ a \textit{regular supercuspidal datum} of $\G$ (mapping to the regular supercuspidal $L$-packet datum $(\bfS,\widehat{j},\chi,\xi)$).
For each regular supercuspidal datum $(\bfS,\widehat{j},\chi,\xi,j)$, the pair
\[
\bigl(j(\bfS),\xi_{j}\bigr)
:=
\bigl(
j(\bfS),
\epsilon\cdot(\xi\cdot\zeta_{\bfS}^{-1})\circ j^{-1}
\bigr)
\] 
is a ($\G(F)$-conjugacy class of) tame elliptic regular pair of $\G$.
Here 
\[
\epsilon
:=
\epsilon^{\sym}\cdot\epsilon_{\sym,\ur}\cdot\epsilon_{f,\ram}
\]
and $\zeta_{\bfS}$ is the character determined by the ``$\zeta$-data'' measuring the difference between $\chi$ and ``Kaletha's $\chi$-data'' $\chi_{\Kal}$.
Kaletha's $\chi$-data $\chi_{\Kal}$ is defined in \cite[the paragraph before Definition 4.7.3]{MR4013740} (note that in \cite{MR4013740} $\chi_{\Kal}$ is written as $\chi'$). 
Later we will recall its definition (Sections \ref{subsec:asymm}, \ref{subsec:symmur}, and \ref{subsec:symmram}).
On the other hand, we do not recall the definition of the character $\zeta_{\bfS}$, but point out that it is trivial when $\chi$ is equal to $\chi_{\Kal}$ (see Step 2 and Definition 4.6.5 in \cite{MR4013740} for details).
We put 
\[
\Pi_{\phi}
:=
\{\pi_{(j(\bfS),\xi_{j})}^{\KY}\mid j\in\mathcal{J}\}.
\]
\end{description}

\begin{rem}
In \cite{MR4013740}, $L$-packets are constructed according to his formalism of ``rigid inner forms'', which was established in \cite{MR3548533}, and compounded in the sense that they consist not only of representations of $\G(F)$ but also those of all rigid inner forms of $\G$.
Thus, strictly speaking, we also have to add a datum of a rigid inner twist to the above explanation of the definition of a regular supercuspidal datum (see \cite[Definition 5.3.2]{MR4013740}).
However, since we treat the local Langlands correspondence only in the case of $\G$ itself in this paper, we always take a rigid inner twist $(\G',\psi,z)$ in each regular supercuspidal datum $(\bfS,\widehat{j},\chi,\xi,(\G',\psi,z),j)$ to be the trivial twist $(\G,\id,1)$, and write shortly $(\bfS,\widehat{j},\chi,\xi,j)$ for $(\bfS,\widehat{j},\chi,\xi,(\G,\id,1),j)$.
\end{rem}

Here we note that the above map
\[
(\bfS,\widehat{j},\chi,\xi,j)
\mapsto
\bigl(j(\bfS),\xi_{j}\bigr)
:=
\bigl(
j(\bfS),
\epsilon\cdot(\xi\cdot\zeta_{\bfS}^{-1})\circ j^{-1}
\bigr)
\]
gives a bijection from the set of isomorphism classes of regular supercuspidal data of $\G$ to the set of $\G(F)$-conjugacy classes of tame elliptic regular pairs of $\G$.
Indeed, the well-definedness of this map is explained (implicitly) in Steps 2 and 3 in \cite[Section 5.3]{MR4013740}.
Moreover we can check that the inverse map is given by 
\[
(\bfS,\xi)
\mapsto
(\bfS,\widehat{j},\chi_{\Kal},\epsilon^{-1}\cdot\xi, j),
\]
where $j$ is the inclusion of $\bfS$ into $\G$ and $\widehat{j}$ is the embedding constructed in the manner as explained in Section \ref{subsec:LS}.
As a consequence, we can conclude that Kaletha's $L$-packets are disjoint from each other and exhaust all regular supercuspidal representations of $\G(F)$.

Now let us focus on the case where $\G=\GL_{n}$.
An important observation coming from the speciality of the group $\GL_{n}$ is that every $L$-packet constructed in this way is a singleton.
Indeed, as already mentioned before, for every regular supercuspidal $L$-packet datum $(\bfS,\widehat{j},\chi,\xi)$, the set $\mathcal{J}$ is nonempty because of the (quasi-)splitness of $\G$.
Moreover, in general, if we take an element $j$ of $\mathcal{J}$, then we can check easily that $\mathcal{J}$ is parametrized by the set
\[
\Ker\bigl(H^{1}(F,\bfS)\xrightarrow{j} H^{1}(F,\G)\bigr),
\]
where the map between two cohomology groups is the one induced from $j$.
As $\bfS$ is an induced torus when $\G$ is $\GL_{n}$, this set is always trivial by Shapiro's lemma and Hilbert's theorem 90.
Thus $\mathcal{J}$ is a singleton, hence so is the corresponding $L$-packet.

In conclusion, Kaletha's construction of the local Langlands correspondence for regular supercuspidal representations of $\G(F)=\GL_{n}(F)$ is summarized as follows:

\begin{prop}\label{prop:Kaletha}
We take a tame elliptic regular pair $(\bfS,\xi)$ of $\G$ and consider the regular supercuspidal representation $\pi_{(\bfS,\xi)}$ of $\G(F)$ arising from this pair $(\bfS,\xi)$.
Let $E$ be a tamely ramified extension of $F$ of degree $n$ corresponding to $\bfS$.
We write $\LLC_{\G}^{\Kal}(\pi_{(\bfS,\xi)})$ for the $L$-parameter corresponding to $\pi_{(\bfS,\xi)}$ by Kaletha's construction.
Then we have
\[
\LLC_{\G}^{\Kal}(\pi_{(\bfS,\xi)})
=
\Ind_{W_{E}}^{W_{F}}(\epsilon^{-1}\xi\mu_{\chi_{\Kal}}),
\]
where $\mu_{\chi_{\Kal}}$ is the character of $E^{\times}$ determined by the set $\chi_{\Kal}$ of Kaletha's $\chi$-data according to the recipe of Proposition \ref{prop:Tam}.
\end{prop}

\begin{proof}
Let $(\bfS,\xi)$ be a pair as in the statement.
Then, as explained in this section, the $L$-packet containing $\pi_{(\bfS,\xi)}$ is a singleton and arises from the regular supercuspidal $L$-packet datum $(\bfS,\widehat{j},\chi_{\Kal},\epsilon^{-1}\xi)$.
Thus the corresponding $L$-parameter is given by ${}^{L}\!j_{\chi_{\Kal}}\circ\phi_{\epsilon^{-1}\xi}$.
If we apply Proposition \ref{prop:Tam} to this $L$-parameter, we obtain the induced representation as desired.
\end{proof}

\section{Prerequisites for a comparison of $\chi$-data}\label{sec:pre}

\subsection{Classification of roots of elliptic tori of $\GL_{n}$}\label{subsec:root}
Let $E/F$ be a tamely ramified extension of degree $n$.
We simply write $e$ (resp.\ $f$) for the ramification index $e(E/F)$ (resp.\ residue degree $f(E/F)$).
We take an elliptic maximal torus $\bfS$ in $\G$ which is isomorphic to $\Res_{E/F}\Gm$.
In this section we classify the symmetric unramified roots and ramified roots in $\Phi(\bfS,\G)$.
Especially, we will determine the relation between symmetric unramified/ramified roots in the sense of Adler--DeBacker--Spice, which we adopt in this paper, and those in the sense of Tam.
We note that Tam's definition of the symmetricity of roots coincides with our definition.
However, his definition of unramifiedness and ramifiedness for symmetric roots is different from ours.
See \cite[1710 page]{MR3509939} for the definition of symmetric unramifiedness and ramifiedness of Tam.

We first recall an explicit choice of a set of representatives of 
$\Gamma_F/\Gamma_E$, following \cite[Section 3.2]{MR3509939}.
We take uniformizers $\varpi_{E}$ and $\varpi_{F}$ of $E$ and $F$, respectively,
so that
\[
\varpi_{E}^e=\zeta _{E/F}\varpi_{F}
\]
for some $\zeta_{E/F} \in \mu_{E}$.
We fix a primitive $e$-th root $\zeta_{e}$ of unity 
and an $e$-th root $\zeta_{E/F, e}$ of $\zeta_{E/F}$,
and put $\zeta_{\phi}:=\zeta_{E/F, e}^{q-1}$.
Then $L:=E[\zeta_{e}, \zeta_{E/F, e}]$ is a tamely ramified extension of $F$
which contains the Galois closure of $E/F$ and is unramified over $E$.
The Galois group $\Gamma_{L/F}$ of the extension $L/F$ is given by the semi-direct product
$\langle \sigma \rangle \rtimes \langle \phi \rangle$, where
\begin{align*}
&\sigma \colon \zeta \mapsto \zeta \quad (\zeta \in \mu_{L}), \quad \varpi_{E} \mapsto \zeta_e \varpi_{E} \\
&\phi \colon \zeta \mapsto \zeta ^q \quad (\zeta \in \mu_{L}), \quad \varpi_{E} \mapsto \zeta_{\phi}\varpi_{E}
\end{align*}
and
$\phi \sigma \phi ^{-1}=\sigma ^q$.
Moreover, as explained in \cite[Proposition 3.3 (i)]{MR3509939}, we can take a set of representatives of $\Gamma_{F}/\Gamma_{E}$ to be
\[
\{\Gamma_{F}/\Gamma_{E}\}
:=
\{
\sigma^{k}\phi^{i}
\mid
0\leq k \leq e-1,\, 0\leq i \leq f-1
\}.
\]
Here we implicitly regard each $\sigma^{k}\phi^{i}\in\Gamma_{L/F}$ as an element of $\Gamma_{F}$ by taking its extension to $\overline{F}$ from $L$.
In the rest of this paper, we always adopt this set of representatives and use the notations defined in Section \ref{subsec:Tam}.
We note that, as $L/E$ is unramified, there exists an integer $c$ such that 
$\Gamma_{L/E}=\langle \sigma^{c} \phi^{f}\rangle$.

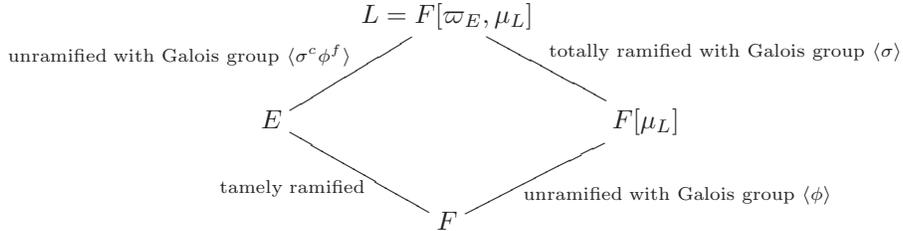
\[
\xymatrix{
&L=F[\varpi_{E},\mu_{L}]&\\
E\ar@{-}[ur]^-{\text{unramified with Galois group $\langle\sigma^{c}\phi^{f}\rangle$}\quad\quad\quad}&&F[\mu_{L}]\ar@{-}[ul]_-{\quad\quad\quad\quad\text{totally ramified with Galois group $\langle\sigma\rangle$}}\\
&F\ar@{-}[ul]^-{\text{tamely ramified}\quad}\ar@{-}[ur]_-{\quad\quad\quad\text{unramified with Galois group $\langle\phi\rangle$}}&
}
\]

The following criterion of the symmetricity of roots is a slight enhancement of \cite[Proposition 3.2]{MR3509939}:

\begin{prop} \label{prop:Tam3.2}
Let $\alpha\in\Phi(\bfS,\G)$ be a root of the form $\begin{bmatrix}1\\g\end{bmatrix}$ for some $g\in\{\Gamma_{F}/\Gamma_{E}\}$.
	\begin{enumerate}
		\item \label{item:Eg} We have $\Gamma_{\alpha}=\Gamma_{E} \cap g\Gamma_{E}g^{-1}$.
		\item	\label{item:xg} The root $\alpha$ is symmetric if and only if there exists $x_g\in \Gamma_{E}$ such that $gx_gg\in \Gamma_{E}$. 
		Moreover, in this case
		we have $(gx_g)^2\in \Gamma_{\alpha}$ and $\Gamma_{\pm\alpha}=\langle gx_g, \Gamma_{\alpha}	\rangle$.
\end{enumerate}
\end{prop}

\begin{proof}
	As in the proof of \cite[Proposition 3.1]{MR3509939},
	the set $\Phi(\bfS, \G)$ is identified with the set of off-diagonal elements in $\Gamma_{F}/\Gamma_{E} \times \Gamma_{F}/\Gamma_{E}$
	by the map $\begin{bmatrix}	g_1 \\ g_2\end{bmatrix} \mapsto (g_1\Gamma_{E}, g_2\Gamma_{E})$ and
	the induced $\Gamma_{F}$-action on the latter set is described as $g(g_1\Gamma_{E}, g_2\Gamma_{E})=(gg_1\Gamma_{E}, gg_2\Gamma_{E})$.
	The assertion \eqref{item:Eg} follows immediately from this.
	
	Let us prove the assertion \eqref{item:xg}.
	The same description shows that the $\Gamma_{F}$-orbit of $-\alpha$ contains $\alpha'=\begin{bmatrix}	1 \\ g^{-1}	\end{bmatrix}$.
	Therefore, $\alpha$ is symmetric if and only if $\Gamma_{E}g\Gamma_{E}=\Gamma_{E}g^{-1}\Gamma_{E}$ holds.
	This latter condition amounts to the existence of $x_g, y_g\in \Gamma_{E}$ such that $y_ggx_g=g^{-1}$,
	or equivalently, the existence of $x_g\in \Gamma_{E}$ such that 
	$gx_gg\in \Gamma_{E}$.
Moreover, for such a choice of $x_g\in \Gamma_{E}$, we have 
\[
(gx_g)^2=(gx_gg)x_g\in \Gamma_{E}
\quad
\text{and}
\quad
(gx_g)^2=gx_g(gx_gg)g^{-1}\in g\Gamma_{E}g^{-1},
\]
	which shows $(gx_g)^2\in \Gamma_{\alpha}$ by \eqref{item:Eg}.
	Now again by the description of $\Gamma_{F}$-action above
	we see that $gx_g$ sends $\alpha$ to $-\alpha$.
	Hence $\Gamma_{\pm \alpha}=\langle gx_g, \Gamma_{\alpha}\rangle$ as required.
\end{proof}

We next give a sufficient condition for a given symmetric root to be unramified.

\begin{lem} \label{lem:symm}
Let $\alpha\in\Phi(\bfS,\G)$ be a root of the form $\begin{bmatrix}1\\g\end{bmatrix}$ for some $g\in\{\Gamma_{F}/\Gamma_{E}\}$.
Moreover we assume that $\alpha$ is symmetric.
		\begin{enumerate}
		\item \label{item:compfield} Suppose that $g\notin \Gamma_{K}$ for a subfield $F\subset K\subset E$.
		Then we have $F_{\alpha}=F_{\pm\alpha}\cdot K$.
		\item \label{item:unramroot} Let $K=F[\mu_{E}]$ be the maximal unramified extension of $F$ in $E$.
		If $g\notin \Gamma_{K}$ then $g$ is symmetric unramified.
	\end{enumerate}
\end{lem}

\begin{proof}
	Let us prove \eqref{item:compfield}.
	Since $F_{\alpha}\supset E\supset K$, 
	we clearly have $F_{\alpha}\supset F_{\pm\alpha}\cdot K$.	
	Let $x_g\in \Gamma_{E}$ be as in Proposition \ref{prop:Tam3.2}.
	As $g$ does not fix $K$ while $x_g$ does,
	the product $gx_g$ does not fix $K$ either.
	That is, $K$ is not contained in the fixed field $F_{\pm\alpha}=F_{\alpha}^{gx_g}$.
	Since $F_{\alpha}/F_{\pm\alpha}$ is quadratic, this implies $F_{\alpha}=F_{\pm\alpha}\cdot K$.

	Let us prove \eqref{item:unramroot}.
	By \eqref{item:compfield} we have $F_{\alpha}=F_{\pm\alpha}\cdot K$.
	Now $K/F$ is unramified and so is $F_{\alpha}/F_{\pm\alpha}$.
\end{proof}

We next investigate symmetric ramified roots.
In fact, we can show that they are of a very limited form as follows:

\begin{prop}\label{prop:symmram}
The set $\Phi (\bfS, \G)$ has a symmetric ramified root if and only if $e$ is even.
In this case there exists a unique $\Gamma_{F}$-orbit of symmetric ramified roots and it is represented by $\begin{bmatrix}1 \\ \sigma^{\frac{e}{2}}\end{bmatrix}$.
\end{prop}

\begin{proof}
We first prove that if $e$ is even, then the root $\alpha=\begin{bmatrix}1\\\sigma^{\frac{e}{2}}\end{bmatrix}$ is a symmetric ramified root.
Let us apply Proposition \ref{prop:Tam3.2} \eqref{item:xg} to $g:=\sigma^{\frac{e}{2}}$.
Note that the behaviour of $g$ on $E$ is particularly simple:
\[
g=\sigma^{\frac{e}{2}} \colon
\zeta \mapsto \zeta \quad (\zeta \in \mu_{E})
\quad\text{and}\quad 
\varpi_{E} \mapsto -\varpi_{E}.
\]
Thus $g$ maps $E$ onto itself and we can take $x_g$ as in Proposition \ref{prop:Tam3.2} \eqref{item:xg} to be the identity.
In particular, $\alpha$ is symmetric.
Moreover, by noting that $F_{\alpha}=E\cdot g(E)=E$, we have $F_{\pm\alpha}=E^{gx_{g}}=E^{g}=F[\mu_{E}, \varpi_{E}^2]$.
Hence $F_{\alpha}$ is a ramified quadratic extension of $F_{\pm\alpha}$.
In other words, $\alpha$ is ramified.

We next prove the converse.
To be more precise, let $\alpha\in\Phi(\bfS,\G)$ be a symmetric ramified root of the form $\begin{bmatrix}1\\g\end{bmatrix}$ for some $g=\sigma^{k}\phi^{i}\in\{\Gamma_{F}/\Gamma_{E}\}$.
Then it is enough to show that $e$ is even and $g$ is necessarily equal to $\sigma^{\frac{e}{2}}$.

First, since our $\alpha$ is symmetric ramified, $g$ fixes $F[\mu_{E}]$ by Lemma \ref{lem:symm} \eqref{item:unramroot}.
In other words, we may assume $g=\sigma^{k}$ with $0< k \leq e-1$.
Let us show that $e$ is even and that this $k$ is in fact given by $\frac{e}{2}$.

We take $x_{g}\in\Gamma_{E}$ as in Proposition \ref{prop:Tam3.2} \eqref{item:xg}, that is, $gx_{g}g$ belongs to $\Gamma_{E}$.
Here, since we have $\Gamma_{L/E}=\langle \sigma^{c} \phi^{f}\rangle$, we may assume that $x_{g}$ is given by $(\sigma^{c}\phi^{f})^{r}$ for some integer $r\in\Z$.
Then we can easily compute that
\begin{align*}
x_{g}=\sigma ^{c\frac{q^{fr}-1}{q^f-1}}\phi ^{fr} \quad\text{and}\quad
gx_{g}g=\sigma ^{k(1+q^{fr})+c\frac{q^{fr}-1}{q^f-1}}\phi ^{fr}.
\end{align*}
Thus the condition that $gx_{g}g\in\Gamma_{E}$ can be rephrased as
\begin{equation} \label{eq:symmcond}
k(1+q^{fr})\equiv 0 \pmod e.
\end{equation}

On the other hand, recall that the unique nontrivial element of the Galois group $\Gamma_{F_{\alpha}/F_{\pm\alpha}}$ is represented by $gx_{g}$.
Thus our symmetric root $\alpha$ is ramified if and only if 
\begin{equation} \label{eq:ramcond}
gx_{g}(\zeta)=\zeta \quad (\text{for any }\zeta \in \mu_{F_{\alpha}}).
\end{equation}
As we have $F_{\alpha}=E\cdot g(E)=F[\varpi_{E}, \zeta_{q^f-1}, \zeta_{e}^{k}]$ (here $\zeta_{q^f-1}\in \mu_{E}$ is a primitive $(q^{f}-1)$st root of unity), the condition \eqref{eq:ramcond} is equivalent to the condition that
\[
gx_{g}(\zeta_{q^f-1})=\zeta_{q^f-1}
\quad\text{and}\quad 
gx_{g}(\zeta_{e}^k)=\zeta_{e}^k.
\]
Since $x_{g}$ lies in $\Gamma_{L/E}$ and $\sigma$ fixes $\mu_{L}$, the product $gx_{g}$ fixes $\mu_{E}$.
In particular, the first equation is automatic.
Therefore, by noting that we have $gx_{g}(\zeta_{e}^k)=\zeta_{e}^{kq^{fr}}$, we see that the condition \eqref{eq:ramcond} holds if and only if 
\[
k(q^{fr}-1)\equiv 0 \pmod e,
\]
    which in turn is equivalent to $-2k\equiv 0\pmod e$
    in view of \eqref{eq:symmcond}.
    Recalling that $0< k \leq e-1$ we conclude that
    if $g$ is symmetric ramified then
    $e$ is even and $k=\frac{e}{2}$.	
\end{proof}

We finally describe a condition for a given symmetric root to be unramified.

\begin{prop}\label{prop:symmur}
Let $\alpha\in\Phi(\bfS,\G)$ be a root of the form $\begin{bmatrix}1\\g\end{bmatrix}$ for some $g=\sigma^{k}\phi^{i}\in\{\Gamma_{F}/\Gamma_{E}\}$.
If $\alpha$ is symmetric unramified, then $i$ is given by either $0$ or $\frac{f}{2}$ (in the latter case, $f$ must be even).
Moreover, in the case where $i=0$, $k$ is not equal to $\frac{e}{2}$.
\end{prop}

\begin{proof}
The first assertion is nothing but \cite[Proposition 3.3 (iii)]{MR3509939}.
The latter assertion follows from Proposition \ref{prop:symmram}.
\end{proof}

Now let us summarize the relation between Adler--DeBacker--Spice's and Tam's notions of unramifiedness/ramifiedness for symmetric roots.
Let $\alpha=\begin{bmatrix}1\\\sigma^{k}\phi^{i}\end{bmatrix}\in\Phi(\bfS,\G)$ be a symmetric root with $0\leq k \leq e-1$ and $0\leq i \leq f-1$.
Then $\alpha$ is ramified in the sense of Adler--DeBacker--Spice only when $(k,i)=(\frac{e}{2},0)$.
In this case, $\alpha$ is ramified also in the sense of Tam (see \cite[1710 page]{MR3509939}).
Next let us assume $\alpha$ is unramified in the sense of Adler--DeBacker--Spice.
Then, by Proposition \ref{prop:symmur}, $i$ equals either $0$ or $\frac{f}{2}$.
Moreover, the root $\alpha$ is
\[
\begin{cases}
\text{ramified in the sense of Tam}& \text{when $i=0$},\\
\text{unramified in the sense of Tam}& \text{when $i=\frac{f}{2}$}.
\end{cases}
\]

\[
    \begin{tabular}{|c|c|c|} \hline
      symmetric root $\alpha=\begin{bmatrix}1\\g\end{bmatrix}$ & the sense of ADS & the sense of Tam \\ \hline
      $g=\sigma^{\frac{e}{2}}$ &  ramified &  ramified \\ \hline
      $g=\sigma^{k} (k\neq \frac{e}{2})$ &  unramified &  ramified \\ \hline
      $g=\sigma^{k}\phi^{\frac{f}{2}}$ &  unramified &  unramified \\ \hline
    \end{tabular}
\]

\subsection{Howe factorization}\label{subsec:Howe}
In this section, we recall the notion of a \textit{Howe factorization} according to Kaletha's sophisticated definition.
Temporarily let $\G$ be a general tamely ramified connected reductive group over $F$.

Let $\bfS$ be a tamely ramified elliptic maximal torus of $\G$ defined over $F$ and $\xi$ a character of $\bfS(F)$.
For each positive real number $r\in\R_{>0}$, we define a subset $\Phi_{r}$ of $\Phi(\bfS,\G)$ as in \cite[1107 page, (3.6.1)]{MR4013740} by
\[
\Phi_{r}
:=
\{
\alpha\in\Phi(\bfS,\G) 
\mid
\xi\equiv \mathbbm{1} \text{ on }\Nr_{\bfS(L)/\bfS(F)}\circ\alpha^{\vee}(L_{r}^{\times})
\},
\]
where $L$ is the splitting field of the torus $\bfS$.
We let $r_{d-1}>\cdots >r_{1}>r_{0}$ be the real numbers satisfying 
\[
\Phi_{r}
\subsetneq
\Phi_{r+}:=\bigcap_{r'>r}\Phi_{r'}
\]
and call them the \textit{jumps} of $\xi$.
We put $r_{d}:=\depth(\xi)$ and $r_{-1}:=0$ (note that $r_{d}\geq r_{d-1}$).
Since each $\Phi_{r_{i}}$ is a Levi subsystem of $\Phi(\bfS,\G)$ by \cite[Lemma 3.6.1]{MR4013740}, we get a sequence $\G^{0}\subsetneq\cdots\subsetneq\G^{d-1}$ of tame twisted Levi subgroups of $\G$ corresponding to the sequence $\Phi_{r_{0}}\subsetneq\cdots\subsetneq\Phi_{r_{d-1}}$ (i.e., each $\G^{i}$ is the tame twisted Levi subgroup of $\G$ which contains $\bfS$ and has $\Phi_{r_{i}}$ as its roots).
We put $\G^{-1}:=\bfS$ and $\G^{d}:=\G$ so that we have $\G^{-1}\subset\G^{0}\subsetneq\cdots\subsetneq\G^{d-1}\subsetneq\G^{d}$.

\begin{defn}[Howe factorization, {\cite[Definition 3.6.2]{MR4013740}}]\label{defn:Howe}
A sequence of characters $\phi_{i}\colon\G^{i}(F)\rightarrow\C^{\times}$ for each $i=-1,\ldots,d$ is called a \textit{Howe factorization} of $\xi$ if it satisfies the following conditions:
\begin{itemize}
\item[(i)]
Each $\phi_{i}$ is trivial on $\G^{i}_{\mathrm{sc}}(F)$ (the image of the set of $F$-valued points of the simply-connected cover of the derived group of $\G^{i}$ in $\G^{i}(F)$).
\item[(ii)]
\begin{itemize}
\item
For $i=-1$, we have
$\begin{cases}
\phi_{-1}=\mathbbm{1}&\text{if $\bfS=\G^{0}$},\\
\phi_{-1}|_{\bfS(F)_{0+}}=\mathbbm{1}&\text{if $\bfS\subsetneq\G^{0}$}.
\end{cases}$
\item
For all $0\leq i \leq d-1$, the character $\phi_{i}$ is $\G^{i+1}$-generic of depth $r_{i}$.
\item
For $i=d$, we have
$\begin{cases}
\phi_{d}=\mathbbm{1}&\text{if $r_{d-1}=r_{d}$},\\
\depth(\phi_{d})=r_{d}&\text{if $r_{d-1}<r_{d}$}.
\end{cases}$
\end{itemize}
\item[(iii)]
We have
\[
\xi=\prod_{i=-1}^{d} \phi_{i}|_{\bfS(F)}.
\]
\end{itemize}
Here see \cite[Section 9]{MR1824988} for the definition of being 
$\G^{i+1}$-generic of depth $r_{i}$ for each character $\phi_{i}$.
\end{defn}

As proved in \cite[Proposition 3.6.7]{MR4013740}, we can always take a Howe factorization for any given character $\xi$ of $\bfS(F)$.
This fact has been known by Howe and Moy classically (\cite{MR0492087,MR853218}) in the $\GL_{n}$-case, and the connection between their definition of a Howe factorization and Kaletha's one in the $\GL_{n}$-case can be roughly explained as follows:
Let $E$ be a tamely ramified extension of $F$ of degree $n$ corresponding to $\bfS\subset\GL_{n}$.
In the definition of Howe and Moy, a Howe factorization of $\xi$ consists of a sequence of subfields 
\[
F=E_{d}
\subsetneq
E_{d-1}
\subsetneq
\cdots
\subsetneq
E_{0}
\subset
E_{-1}=E
\]
and characters $\xi_{i}$ of $E_{i}^{\times}$ for each $-1\leq i\leq d$ satisfying several conditions such as 
\[
\xi
=
\xi_{-1}
\cdot
(\xi_{0}\circ\Nr_{E/E_{0}})
\cdots
(\xi_{d}\circ\Nr_{E/E_{d}})
\]
(see, for example, \cite[Definition 3.33]{MR2431732}).
Suppose that we have such sequences of fields and characters.
Then, by considering the centralizers of 
\[
F^{\times}=E_{d}^{\times}
\subsetneq
E_{d-1}^{\times}
\subsetneq
\cdots
\subsetneq
E_{0}^{\times}
\subset
E_{-1}^{\times}\cong\bfS(F)
\]
in $\GL_{n}$, we get a sequence of tame twisted Levi subgroups of $\G$:
\[
\GL_{n}=\G^{d}
\supsetneq
\G^{d-1}
\supsetneq
\cdots
\supsetneq
\G^{0}
\supset
\G^{-1}=\bfS.
\]
Here note that if we put $n_{i}:=n/[E_{i}:F]$, then $\G^{i}$ is isomorphic to $\Res_{E_{i}/F}\GL_{n_{i},E_{i}}$.
In fact, this sequence of tame twisted Levi subgroups $\G^{0}\subsetneq\cdots\subsetneq\G^{d-1}$ has $\Phi_{r_{0}}\subsetneq\cdots\subsetneq\Phi_{r_{d-1}}$ as its roots.
Moreover, if we put $\phi_{i}:=\xi_{i}\circ\det$, we get a Howe factorization in the sense of Kaletha.
Note that, since $\G^{i}_{\mathrm{sc}}(F)$ is given by $\SL_{n_{i}}(E_{i})$, the triviality condition for $\phi_{i}$ on $\G^{i}_{\mathrm{sc}}(F)$ is equivalent to the condition that it factors through the determinant map.
Furthermore, as the determinant map is nothing but the norm map $\Nr_{E/E_{i}}$ on $\bfS(F)\cong E^{\times}$, we have $\phi_{i}|_{\bfS(F)}=\xi_{i}\circ\Nr_{E/E_{i}}$.
\[
\xymatrix{
\bfS(F)\cong E^{\times}\ar[rd]_-{\Nr_{E/E_{i}}}\ar@{^{(}->}[r]&\G^{i}(F)=\GL_{n_{i}}(E_{i})\ar[r]^-{\phi_{i}}\ar[d]^-{\det}&\C^{\times}\\
&E_{i}^{\times}\ar[ru]_-{\xi_{i}}&
}
\]

See \cite[Section 3.5]{MR2431732} for more detailed explanation on the relation between these two apparently different definitions of a Howe factorization.

\begin{rem}\label{rem:jump}
In the language of Bushnell--Henniart, the jumps are defined to be the $E$-levels of $\xi_{i}\circ\Nr_{E/E_{i}}$ for $0\leq i\leq d-1$ (see, for example, \cite[Section 5.1]{MR3509939}).
If we put the $E$-levels of $\xi_{i}\circ\Nr_{E/E_{i}}$ to be $t_{i}$, then the relation between $t_{i}$ and the depth $r_{i}$ of $\phi_{i}$ is expressed as 
\[
r_{i}=\frac{t_{i}}{e}.
\]
This follows from the tameness assumption on $E/F$ and the fact that the depth of $\phi_{i}|_{\bfS(F)}$ is again given by $r_{i}$ (see \cite[Lemma 2.52]{MR2431732}).
\end{rem}

The following ``descent'' property of a Howe factorization will be needed later (in Section \ref{subsec:symmram}).

\begin{prop}\label{prop:Howe-rest}
Let $\G$ be a tamely ramified connected reductive group over $F$ and $\bfS$ a tamely ramified elliptic maximal torus of $\G$.
Let $\xi$ be a character of $\bfS(F)$.
We consider a tame twisted Levi subgroup $\bfH$ of $\G$ containing $\bfS$.
Then the set of jumps $r'_{d'-1}>\cdots>r'_{0}$ of $\xi$ with respect to $(\bfS,\bfH)$ is contained in that $r_{d-1}>\cdots>r_{0}$ with respect to $(\bfS,\G)$ (note that we always have $r'_{d'}=r_{d}$ and $r'_{-1}=r_{-1}$).
Moreover we can take
\begin{itemize}
\item
a Howe factorization $(\phi_{-1},\ldots,\phi_{d})$ of $\xi$ with respect to $\bfS\subset\G$ and
\item
a Howe factorization $(\phi'_{-1},\ldots,\phi'_{d'})$ of $\xi$ with respect to $\bfS\subset\bfH$
\end{itemize}
such that we have
\[
\phi_{[i]}|_{\bfS(F)_{r_{[i]}}}=\phi'_{i}|_{\bfS(F)_{r'_{i}}}
\]
for any $0\leq i\leq d'$.
Here $[i]$ denotes the unique index satisfying $r_{[i]}=r'_{i}$.
\end{prop}

\begin{proof}
Let $\Phi_{r}$ and $\Phi'_{r}$ be the $r$-th filtration on the sets $\Phi(\bfS,\G)$ and $\Phi(\bfS,\bfH)$ of roots, respectively.
Then, as $\Phi(\bfS,\bfH)$ is a subset of $\Phi(\bfS,\G)$, the first assertion on the jumps immediately follows from the definitions of $\{\Phi_{r}\}_{r\in\R_{>0}}$ and $\{\Phi'_{r}\}_{r\in\R_{>0}}$.
Moreover, for any index $0\leq i\leq d'$, we have $\Phi'_{r'_{i-1}}\subset\Phi_{r_{[i-1]}}$ and $\Phi'_{r'_{i}}=\Phi'_{r'_{i-1}+}\subset\Phi_{r_{[i-1]+1}}$ by the definition of jumps.
Thus each $\bfH^{i-1}$ is contained in $\G^{[i-1]}$ and that $\bfH^{i}$ is contained in $\G^{[i-1]+1}$.
\[
\xymatrix{
\G^{[i-1]}\ar@{^{(}->}[r]&\G^{[i-1]+1}\ar@{^{(}->}[r]&\cdots\ar@{^{(}->}[r]&\G^{[i]}\\
\bfH^{i-1}\ar@{^{(}->}[rrr]\ar@{^{(}->}[u]&&&\bfH^{i}\ar@{_{(}->}[llu]\ar@{^{(}->}[u]
}
\]
Note that, for $i=-1$, we always have $\G^{-1}=\bfH^{-1}=\bfS$ by definition.

Now we take a Howe factorization $(\phi_{-1},\ldots,\phi_{d})$ of $\xi$ with respect to $\bfS\subset\G$.
By noting that $r_{d}=r'_{d'}$ and $r_{-1}=r'_{-1}$ (i.e., $[d']=d$ and $[-1]=-1$), we define $\phi'_{i}$ as follows:
\begin{align*}
\phi'_{i}
&:=
\phi_{[i-1]+1}|_{\bfH^{i}(F)}\cdot\phi_{[i-1]+2}|_{\bfH^{i}(F)}\cdots\phi_{[i]}|_{\bfH^{i}(F)} \quad\text{for $1\leq i \leq d'$},\\
\phi'_{0}
&:=
\begin{cases}
\phi_{-1}|_{\bfS(F)}\cdot\phi_{0}|_{\bfS(F)}\cdots\phi_{[0]}|_{\bfS(F)}  & \text{if $\bfS=\bfH^{0}$},\\
\phi_{0}|_{\bfH^{0}(F)}\cdot\phi_{1}|_{\bfH^{0}(F)}\cdots\phi_{[0]}|_{\bfH^{0}(F)}  & \text{if $\bfS\subsetneq\bfH^{0}$},
\end{cases}\\
\phi'_{-1}
&:=
\begin{cases}
\mathbbm{1} & \text{if $\bfS=\bfH^{0}$},\\
\phi_{-1} & \text{if $\bfS\subsetneq\bfH^{0}$}.
\end{cases}
\end{align*}
Then $(\phi'_{-1},\ldots,\phi'_{d'})$ gives a desired Howe factorization.
Indeed, by \cite[Lemma 2.52]{MR2431732}, the depth of $\phi_{j}|_{\bfS(F)}$ is equal to that of $\phi_{j}$ for any $j$.
Note that, although only positive depth characters are treated in \cite[Lemma 2.52]{MR2431732}, it also holds that the depth of the restriction of a depth zero character is again zero by the tame descent property of Moy--Prasad filtrations (see, e.g., \cite[Lemma 3.8]{MR2431235}).
Hence the restrictions of $\phi_{[i-1]+1},\ldots,\phi_{[i]-1}$ to $\bfS(F)_{r'_{i}}$ are trivial and we have
\[
\phi_{[i]}|_{\bfS(F)_{r_{[i]}}}=\phi'_{i}|_{\bfS(F)_{r'_{i}}}.
\]
Moreover, we can check each condition in Definition \ref{defn:Howe} so that $(\phi'_{-1},\ldots,\phi'_{d'})$ is a Howe factorization of $\xi$ with respect to $(\bfS,\bfH)$ as follows:
\begin{itemize}
\item[(i)]
In general, for any sequence $\bfH_{1}\subset\bfH_{2}$ of connected reductive groups, the map $\bfH_{1,\mathrm{sc}}\rightarrow\bfH_{1}\subset\bfH_{2}$ factors through $\bfH_{2,\mathrm{sc}}$.
In particular, the image of $\bfH_{1,\mathrm{sc}}(F)$ in $\bfH_{2}(F)$ is contained in that of $\bfH_{2,\mathrm{sc}}(F)$ in $\bfH_{2}(F)$.
Thus the condition (i) for $(\phi_{-1},\ldots,\phi_{d})$ implies that for $(\phi'_{-1},\ldots,\phi'_{d'})$.
\item[(ii)]
\begin{itemize}
\item
The condition for $i=-1$ immediately follows from the definition of $\phi'_{-1}$.
\item
We assume that $0\leq i \leq d'-1$.
Recall that being $\bfH^{i+1}$-generic of depth $r'_{i}$ is a condition on the restriction $\phi'_{i}|_{\bfH^{i}(F)_{x,r'_{i}}}$ (see \cite[Section 9]{MR1824988}).
Here $x$ is the point of the reduced Bruhat--Tits building of $\bfH^{i}$ determined by $\bfS$ (see Section \ref{subsec:Heisen}) and $\bfH^{i}(F)_{x,r'_{i}}$ denotes the $r'_{i}$-th Moy--Prasad filtration of $\bfH^{i}(F)$ at $x$.
By the same reason as above, we have
\[
\phi'_{i}|_{\bfH^{i}(F)_{x,r'_{i}}}
=
\phi_{[i]}|_{\bfH^{i}(F)_{x,r'_{i}}}
\]
by \cite[Lemma 2.52]{MR2431732}.
Then, as $\bfH^{i}$ is a tame twisted Levi subgroup of $\G^{[i]}$ and we have $r'_{i}=r_{[i]}$, being $\G^{[i]+1}$-generic of depth $r_{[i]}$ for $\phi_{[i]}$ implies being $\bfH^{i+1}$-generic of depth $r'_{i}$ for $\phi'_{i}$.
\item
Finally, we consider the case where $i=d'$.
First, if we have $r'_{d'-1}=r'_{d'}$, then we have $r_{[d'-1]}=r_{[d']} (=r_{d})$.
Therefore, by noting that the sequence $r_{-1}<r_{0}<\cdots<r_{d-1}$ is strictly increasing, we necessarily have $r_{[d'-1]}=r_{d-1}=r_{d}$.
Thus the condition (ii) for $\phi_{d}$ and the definition of $\phi'_{d'}$ implies that $\phi'_{d'}=\phi_{d}|_{\bfH(F)}=\mathbbm{1}$.
Next let us suppose that $r'_{d'-1}<r'_{d'}$.
\begin{itemize}
\item
When $r_{d-1}<r_{d}$, we have $\depth(\phi_{d})=r_{d}$, hence the sequence $r_{[d'-1]+1}<\cdots<r_{d}$ of depths of $\phi_{[d'-1]+1},\ldots,\phi_{d}$ is strictly increasing.
Thus, again by noting that the depth behaves well under the tame descent (\cite[Lemma 2.52]{MR2431732}), we get
\[
\depth(\phi'_{d'})
=
\depth\bigl(\phi_{[d'-1]+1}|_{\bfH(F)}\cdots\phi_{d}|_{\bfH(F)}\bigr)\\
=
r_{d}
=
r'_{d'}.
\]
\item
When $r_{d-1}=r_{d}$, we have $\phi_{d}=\mathbbm{1}$.
Thus, similarly to above, we get
\[
\depth(\phi'_{d'})
=
\depth\bigl(\phi_{[d'-1]+1}|_{\bfH(F)}\cdots\phi_{d}|_{\bfH(F)}\bigr)\\
=
r_{d-1}.
\]
As we have $r_{d-1}=r_{d}=r'_{d'}$, we conclude that $\depth(\phi'_{d'})=r'_{d'}$.
\end{itemize}
\end{itemize}
\item[(iii)]
This follows immediately by the definition of $(\phi'_{-1},\ldots,\phi'_{d'})$ and the condition (iii) for $(\phi_{-1},\ldots,\phi_{d})$.
\end{itemize}
\end{proof}

\subsection{Notation around Heisenberg groups}\label{subsec:Heisen}
In this section, we recall briefly the notion of a t-factor which will be needed in the calculation of the factors of $\chi_{\Tam}$ associated to symmetric roots.
The content of this section is a summary of \cite[Sections 4 and 5.2]{MR3509939}.

Let $E$ be a tamely ramified extension of $F$ of degree $n$.
Then we get a hereditary order of $\End_{F}(E)$ defined by
\[
\mfA
:=
\{
X\in\End_{F}(E)
\mid
X\cdot\mfp_{E}^{k}\subset\mfp_{E}^{k} \text{ for any } k\in\Z
\}.
\]
We put $\mfP$ to be the Jacobson radical of $\mfA$, which is given by
\[
\mfP
=
\{
X\in\End_{F}(E)
\mid
X\cdot\mfp_{E}^{k}\subset\mfp_{E}^{k+1} \text{ for any } k\in\Z
\}.
\]
We consider the quotient
\[
\mfU:=\mfA/\mfP
\]
of $\mfA$ by $\mfP$, which admits a natural $k_{F}$-vector space structure.
Then the following natural action of $E^{\times}$ on $\End_{F}(E)$ induces an $E^{\times}$-action on $\mfU$:
\[
X\in \End_{F}(E),\quad 
t\in E^{\times},\quad
t\cdot X := t\circ X\circ t^{-1},
\]
where $t$ and $t^{-1}$ on the right-hand side are the multiplication maps via $t$ and $t^{-1}$, respectively.
As this action factors through the finite quotient $E^{\times}/F^{\times}U_{E}^{1}$ of order prime to $p$, the space $\mfU$ decomposes into a direct sum of eigenspaces with respect to the $E^{\times}$-action.
This decomposition is described as follows:
\begin{prop}[{\cite[Proposition 4.4]{MR3509939}}]\label{prop:U-decomp}
We have a decomposition of $\mfU$ into a direct sum of $E^{\times}$-isotypic subspaces
\[
\mfU
=
\bigoplus_{[g]\in\Gamma_{E}\backslash\Gamma_{F}/\Gamma_{E}}
\mfU_{[g]}.
\]
Here, for each $[g]\in\Gamma_{E}\backslash\Gamma_{F}/\Gamma_{E}$, the subspace $\mfU_{[g]}$ is canonically identified with the residue field $k_{E\cdot g(E)}$ of the composite field of $E$ and $g(E)$.
Moreover, the action of $t\in E^{\times}$ on $\mfU_{[g]}$ is given by the multiplication by $tg(t)^{-1}\in \mcO_{E\cdot g(E)}^{\times}$ mod $\mfp_{E\cdot g(E)}$.
\end{prop}

\begin{rem}\label{rem:U-decomp}
Recall that we can identify $(\Gamma_{E}\backslash\Gamma_{F}/\Gamma_{E})'$ with the set $\Gamma_{F}\backslash\Phi(\bfS,\G)$ of $\Gamma_{F}$-orbits of roots.
If a double-$\Gamma_{E}$-coset $[g]$ corresponds to a root $\alpha$ under this identification, then we can understand the action of $t\in E^{\times}$ on $\mfU_{[g]}$ as the multiplication by $\alpha(t)\in \mcO_{F_{\alpha}}^{\times}$ mod $\mfp_{F_{\alpha}}$.
In other words, $\mfU_{[g]}$ is the $\alpha|_{\bfS(F)}$-isotypic part of $\mfU$.
Moreover, if we take $\alpha$ to be $\begin{bmatrix}1\\g\end{bmatrix}$ as in Section \ref{subsec:Tam}, then $E\cdot g(E)$ is nothing but $F_{\alpha}$.
\end{rem}

We next consider an $F$-admissible character $\xi$ of $E^{\times}$.
Recall that then we can attach a sequence 
\[
F=E_{d}
\subsetneq
E_{d-1}
\subsetneq
\cdots
\subsetneq
E_{0}
\subset
E_{-1}=E
\]
of subfields and a sequence
\[
r_{d}\geq r_{d-1}>\cdots>r_{0}>r_{-1}=0
\]
of jumps to $\xi$.
Note that, in this paper, we put $E_{d}$ (not $E_{d+1}$ as in \cite{MR3509939}) to be $F$.
We put $t_{i}:=e\cdot r_{i}$ (this is the $E$-level of the $i$-th character $\xi_{i}$ for $0\leq i\leq d-1$, see Remark \ref{rem:jump}) and define integers $h_{i}$ and $j_{i}$ by
\[
h_{i}
:=
\biggl\lfloor\frac{t_{i}}{2}\biggr\rfloor+1
=
\begin{cases}
\frac{t_{i}}{2}+1& \text{if $t_{i}$ is even},\\
\frac{t_{i}+1}{2}&\text{if $t_{i}$ is odd},
\end{cases}\quad
j_{i}
:=
\biggl\lfloor\frac{t_{i}+1}{2}\biggr\rfloor
=
\begin{cases}
\frac{t_{i}}{2}& \text{if $t_{i}$ is even},\\
\frac{t_{i}+1}{2}&\text{if $t_{i}$ is odd}.
\end{cases}
\]
For each $i$, we set
\[
\mfA_{i}
:=
\{
X\in \End_{E_{i}}(E)
\mid
X\cdot\mfp_{E}^{k}\subset\mfp_{E}^{k} \text{ for any } k\in\Z
\}.
\]
We put $\mfP_{i}$ to be the Jacobson radical of $\mfA_{i}$, which is given by
\[
\mfP_{i}
=
\{
X\in\End_{E_{i}}(E)
\mid
X\cdot\mfp_{E}^{k}\subset\mfp_{E}^{k+1} \text{ for any } k\in\Z
\}.
\]
For each positive integer $k\in\Z_{>0}$, we set
\[
U_{\mfA_{i}}^{k}
:=
1+\mfP_{i}^{k}.
\]
Now we define subgroups $H^{1}$ and $J^{1}$ as follows:
\begin{align*}
H^{1}&:=U_{\mfA_{0}}^{1}U_{\mfA_{1}}^{h_{0}}\cdots U_{\mfA_{d-1}}^{h_{d-2}}U_{\mfA_{d}}^{h_{d-1}},\\
J^{1}&:=U_{\mfA_{0}}^{1}U_{\mfA_{1}}^{j_{0}}\cdots U_{\mfA_{d-1}}^{j_{d-2}}U_{\mfA_{d}}^{j_{d-1}}.
\end{align*}
Let $\mfV$ denote the quotient of $J^{1}$ by $H^{1}$.
As both of $H^{1}$ and $J^{1}$ are $E^{\times}$-stable, we can consider a decomposition of this $k_{F}$-vector space $\mfV$ into a direct sum of $E^{\times}$-isotypic parts:
\[
\mfV
=
\bigoplus_{[g]\in(\Gamma_{E}\backslash\Gamma_{F}/\Gamma_{E})'}
\mfV_{[g]}.
\]
On the other hand, by the definitions of $H^{1}$ and $J^{1}$, we have a decomposition
\[
\mfV
=\bigoplus_{i=0}^{d-1}\mfV_{i}
\]
of $\mfV$ into a direct sum of $k_{F}$-vector spaces $\mfV_{i}$, which are isomorphic to $U_{\mfA_{i+1}}^{j_{i}}/U_{\mfA_{i}}^{j_{i}}U_{\mfA_{i+1}}^{h_{i}}$.
Note that each component $\mfV_{i}$ is nonzero only when $t_{i}$ is even.
Furthermore, in this case, we can identify $\mfV_{i}$ with $\mfA_{i+1}/(\mfA_{i}+\mfP_{i+1})$ by using the fixed uniformizer $\varpi_{E}$ of $E$.
As we have
\[
\mfA_{i+1}/(\mfA_{i}+\mfP_{i+1})
\cong
\bigl((\mfA_{i+1}+\mfP)/\mfP\bigr) \big/ \bigl((\mfA_{i}+\mfP)/\mfP\bigr),
\]
we may furthermore regard $\mfV_{i}$ as a subspace of $\mfU$ which is stable under the action of $E^{\times}$.
In this sense, the $[g]$-isotypic part $\mfV_{[g]}$ of $\mfV$ can be identified with $\mfU_{[g]}$ if it is not zero.
We define $\mfW_{[g]}$ by
\[
\mfW_{[g]}:=
\begin{cases}
\mfU_{[g]}& \text{if $\mfV_{[g]}=0$,}\\
0& \text{if $\mfV_{[g]}\neq0$ (hence $\mfV_{[g]}=\mfU_{[g]}$).}
\end{cases}
\]
Thus, by definition, we have
\[
\mfU_{[g]}=\mfV_{[g]}\oplus\mfW_{[g]},
\]
and exactly one of $\mfV_{[g]}$ or $\mfW_{[g]}$ is zero.

By taking an $F$-basis of $E$, the sequence $F=E_{d}\subsetneq E_{d-1}\subsetneq\cdots\subsetneq E_{0}\subset E_{-1}=E$ gives rise to a sequence of tame twisted Levi subgroups of $\G$:
\[
\G=
\G^{d}
\supsetneq
\G^{d-1}
\supsetneq
\cdots
\supsetneq
\G^{0}
\supset
\G^{-1}=\bfS.
\]
Let us suppose that a double-$\Gamma_{E}$-coset $[g]\in\Gamma_{E}\backslash\Gamma_{F}/\Gamma_{E}$ corresponds to the $\Gamma_{F}$-orbit of a root $\alpha\in\Phi(\bfS,\G)$.
Then we can describe the condition so that $\mfV_{[g]}$ is zero (or, equivalently, $\mfW_{[g]}$ is not zero) as follows:
\begin{itemize}
\item
when $\alpha$ belongs to $\Phi(\bfS,\G^{0})$, we always have $\mfV_{[g]}=0$ and $\mfW_{[g]}=\mfU_{[g]}$, and
\item
when $\alpha$ belongs to $\Phi_{i+1}^{i}=\Phi(\bfS,\G^{i+1})-\Phi(\bfS,\G^{i})$ for $0\leq i \leq d-1$, we have
\[
\mfV_{[g]}
=
\begin{cases}
\mfU_{[g]}& \text{if $t_{i}$ is even,}\\
0&\text{if $t_{i}$ is odd,}
\end{cases}
\quad
\mfW_{[g]}
=
\begin{cases}
0 & \text{if $t_{i}$ is even},\\
\mfU_{[g]} & \text{if $t_{i}$ is odd}.
\end{cases}
\]
\end{itemize}
Note that, if $t_{i}$ is even in the latter case, $\mfV_{[g]}$ is contained in $\mfV_{i}$.
Thus $\mfV_{i}$ can be identified with
\[
\bigoplus_{\begin{subarray}{c}[g]\in(\Gamma_{E}\backslash\Gamma_{F}/\Gamma_{E})'\\ [g]\leftrightarrow\alpha\in\Phi_{i+1}^{i}\end{subarray}}\mfU_{[g]}
\]
when $t_{i}$ is even.

Now suppose that we have a finite group $\boldsymbol{\Gamma}$ and a symplectic $\F_{p}[\boldsymbol{\Gamma}]$-module $V$ (i.e., $\F_{p}[\boldsymbol{\Gamma}]$-module equipped with a $\boldsymbol{\Gamma}$-invariant non-degenerate alternating $\F_{p}$-bilinear form).
Then we can attach invariants called \textit{t-factors} $t_{\boldsymbol{\Gamma}}^{0}(V)$ and $t_{\boldsymbol{\Gamma}}^{1}(V)$;
\begin{itemize}
\item
a sign $t_{\boldsymbol{\Gamma}}^{0}(V)\in\{\pm1\}$ and
\item
a sign character $t_{\boldsymbol{\Gamma}}^{1}(V)\colon\boldsymbol{\Gamma}\rightarrow\{\pm1\}$
\end{itemize}
to $V$.
We do not recall the definitions of these objects, see \cite[Section 3.4]{MR2679700} or \cite[Section 4.2]{MR3509939}.
(We note that if $V=0$, then $t_{\boldsymbol{\Gamma}}^{0}(V)=1$ and $t_{\boldsymbol{\Gamma}}^{1}(V)$ is trivial.)
We put $t_{\boldsymbol{\Gamma}}(V)$ to be a map from $\boldsymbol{\Gamma}$ to $\{\pm1\}$ defined by taking their product:
\[
t_{\boldsymbol{\Gamma}}(V):=t_{\boldsymbol{\Gamma}}^{0}(V)\cdot t_{\boldsymbol{\Gamma}}^{1}(V)\colon \boldsymbol{\Gamma}\rightarrow\{\pm1\}.
\]
If we put
\[
\boldsymbol{\mfV}_{[g]}:=
\begin{cases}
\mfV_{[g]}\oplus\mfV_{[g^{-1}]}&\text{if $[g]$ corresponds to an asymmetric root,}\\
\mfV_{[g]}&\text{if $[g]$ corresponds to a symmetric root},
\end{cases}
\]
then it admits a symplectic $\F_{p}[E^{\times}/F^{\times}U_{E}^{1}]$-module structure (see \cite[Sections 4.2.1, 4.2.2, and 4.2.3]{MR3509939}).
In particular, we can consider t-factors for $\boldsymbol{\mfV}_{[g]}$ with 
\begin{itemize}
\item
$\boldsymbol{\Gamma}=\boldsymbol{\mu}:=\mu_{E}/\mu_{F}$ or 
\item
$\boldsymbol{\Gamma}=$ the subgroup $\varpi$ of $E^{\times}/F^{\times}U_{E}^{1}$ generated by the image of $\varpi_{E}$.
\end{itemize}

On the other hand, for an orthogonal $k_{F}[E^{\times}/F^{\times}U_{E}^{1}]$-module $W$ (i.e., $k_{F}[E^{\times}/F^{\times}U_{E}^{1}]$-module equipped with a $E^{\times}/F^{\times}U_{E}^{1}$-invariant non-degenerate symmetric $k_{F}$-bilinear form), we can attach a fourth root of unity $t(W)\in\C^{\times}$ in the manner of \cite[Sections 4.3]{MR3509939}.
(We note that if $W=0$, then $t(W)=1$.)
Since each
\[
\boldsymbol{\mfW}_{[g]}:=
\begin{cases}
\mfW_{[g]}\oplus\mfW_{[g^{-1}]}&\text{if $[g]$ corresponds to an asymmetric root,}\\
\mfW_{[g]}&\text{if $[g]$ corresponds to a symmetric root},
\end{cases}
\]
has an orthogonal $k_{F}[E^{\times}/F^{\times}U_{E}^{1}]$-module structure as explained in \cite[Section 4.3]{MR3509939}, we can consider $t(\boldsymbol{\mfW}_{[g]})$.
Here we remark that, only in the case where $[g]$ corresponds to the $\Gamma_{F}$-orbit of a symmetric ramified root, we have to make a special choice of a quadratic form on $\boldsymbol{\mfW}_{[g]}$.
In Tam's paper, this choice is implicitly specified in \cite[Remark 7.7]{MR3509939} (see also a comment in \cite[1755 page, (ii)]{MR3509939}).
Since it is necessary to compute the t-factor in this case explicitly for our purpose, we will recall the precise definition of the quadratic form and the corresponding t-factor later (see the proof of Proposition \ref{prop:main} in the symmetric ramified case in Section \ref{subsec:symmram}).

Finally, we also have to explain the relation between these notions and the ingredients used in Kaletha's construction of the local Langlands correspondence.
Thus let us take a tamely ramified elliptic maximal torus $\bfS$ of $\G=\GL_{n}$ which corresponds to the tamely ramified extension $E$ of $F$ of degree $n$.
Then, by a tamely ramified descent property of Bruhat--Tits building (\cite{MR1871292}), we can regard the Bruhat--Tits building $\mathcal{B}(\bfS,F)$ of $\bfS$ as a subset of that $\mathcal{B}(\G,F)$ for $\G$.
Moreover, as $\bfS$ is elliptic, the image of $\mathcal{B}(\bfS,F)$ in the reduced building $\mathcal{B}_{\red}(\G,F)$ consists only of one point.
We let $x$ denote this unique point.

For each $\alpha\in\Phi(\bfS,\G)$ and $r\in\R$, we put 
\[
\mfg_{\alpha}(F_{\alpha})_{r}
:=
\mfg_{\alpha}(F_{\alpha})\cap\mfg(F_{\alpha})_{x,r},
\]
where 
\begin{itemize}
\item
$\mfg_{\alpha}$ is the root subspace of $\mfg$ attached to $\alpha$ and 
\item
$\mfg(F_{\alpha})_{x,r}$ is the $r$-th Moy--Prasad filtration of $\mfg(F_{\alpha})$ at $x$.
\end{itemize}
By using this filtration on $\mfg_{\alpha}(F_{\alpha})$, we define a subset $\ord_{x}(\alpha)$ of $\R$ as follows:
\begin{defn}[{\cite[Definition 3.6]{MR3849622}}]\label{defn:ord}
We put
\[
\ord_{x}(\alpha)
:=
\{r\in\R
\mid
\mfg_{\alpha}(F_{\alpha})_{r+}\subsetneq\mfg_{\alpha}(F_{\alpha})_{r}\}.
\]
\end{defn}

The following proposition is encoded in the work of Adler--Spice (\cite{MR2543925}) and DeBacker--Spice (\cite{MR3849622}), especially the proofs of \cite[Proposition 3.8]{MR2543925} and \cite[Proposition 4.21]{MR3849622}.
For the sake of completeness, we give an explanation (in the following proof, we temporarily use the notation of \cite{MR2543925, MR3849622}):
\begin{prop}\label{prop:Heisenberg}
Let $x\in\mathcal{B}_{\red}(\G,F)$ be the point determined by $\bfS$.
We take a double-$\Gamma_{E}$-coset $[g]$ and a root $\alpha\in\Phi(\bfS,\G)$ whose $\Gamma_{F}$-orbit corresponds to $[g]$.
We put $i$ to be the unique index $-1\leq i\leq d-1$ such that $\alpha\in\Phi_{i+1}^{i}$.
When $i=-1$, we have $\mathfrak{V}_{[g]}=0$.
When $0\leq i\leq d-1$, we have $\mathfrak{V}_{[g]}\neq0$ (recall that this is equivalent to that $t_{i}=e\cdot r_{i}$ is even) if and only if we have $\frac{r_{i}}{2}\in \ord_{x}(\alpha)$.
\end{prop}

\begin{proof}
Since the first assertion was already explained in this section, we prove the second assertion by assuming that $i$ is not equal to $-1$.

By the definition of $\ord_{x}(\alpha)$, we have $\frac{r_{i}}{2}\in\ord_{x}(\alpha)$ if and only if the quotient $\mfg_{\alpha}(F_{\alpha})_{x,\frac{r_{i}}{2}}/\mfg_{\alpha}(F_{\alpha})_{x,\frac{r_{i}}{2}+}$ is not zero.
Let $L$ be the splitting field of $\bfS$.
Note that the $V_{\alpha}$ in \cite[Proof of Proposition 3.8]{MR2543925}, which is the $\Gamma_{\alpha}$-fixed part of the quotient $\mfg_{\alpha}(L)_{x,\frac{r_{i}}{2}:\frac{r_{i}}{2}+}:=\mfg_{\alpha}(L)_{x,\frac{r_{i}}{2}}/\mfg_{\alpha}(L)_{x,\frac{r_{i}}{2}+}$ by definition, is nothing but the quotient $\mfg_{\alpha}(F_{\alpha})_{x,\frac{r_{i}}{2}:\frac{r_{i}}{2}+}:=\mfg_{\alpha}(F_{\alpha})_{x,\frac{r_{i}}{2}}/\mfg_{\alpha}(F_{\alpha})_{x,\frac{r_{i}}{2}+}$ by the compatibility of taking the quotient with tamely ramified Galois descent (\cite[Corollary 2.3]{MR1824988}).
Thus it suffices to show that $V_{\alpha}$ is isomorphic to $\mfV_{[g]}$.

As explained in the first paragraph of \cite[Proof of Proposition 3.8]{MR2543925}, we can identify $V_{\alpha}$ with the $\alpha|_{\bfS(F)}$-isotypic part of the $\Gamma_{F}$-fixed part of 
\[
\Lie(\G^{i},\G^{i+1})(L)_{x,(r_{i},\frac{r_{i}}{2}):(r_{i},\frac{r_{i}}{2}+)}
\]
by the map 
\[
X_{\alpha}\mapsto\sum_{\tau\in\Gamma_{F}/\Gamma_{\alpha}} \tau(X_{\alpha})
\]
(note that this isomorphism is $\bfS(F)$-equivariant).
Again by the tamely ramified descent property of taking the quotient (\cite[Corollary 2.3]{MR1824988}), the latter space is identified with the $\alpha|_{\bfS(F)}$-isotypic part of $\Lie(\G^{i},\G^{i+1})(F)_{x,(r_{i},\frac{r_{i}}{2}):(r_{i},\frac{r_{i}}{2}+)}$, which is furthermore isomorphic to $\Lie\G^{i+1}_{x,\frac{r_{i}}{2}}/(\Lie\G^{i}_{x,\frac{r_{i}}{2}}+\Lie\G^{i+1}_{x,\frac{r_{i}}{2}+})$.

On the other hand, by definition, $\mfV_{[g]}$ is the $\alpha|_{\bfS(F)}$-isotypic part of the group $U_{\mfA_{i+1}}^{j_{i}}/U_{\mfA_{i}}^{j_{i}}U_{\mfA_{i+1}}^{h_{i}}$, which is isomorphic to $\mfP_{i+1}^{j_{i}}/(\mfP_{i}^{j_{i}}+\mfP_{i+1}^{h_{i}})$.
Recall that we identify $\Lie\G=\mathfrak{gl}_{n}$ with $\End_{F}(E)$ by taking an $F$-basis of $E$.
Under this identification, the comparison result of Broussous--Lemaire on the lattice filtrations and Moy--Prasad filtrations (\cite{MR1888474}) gives
\[
\Lie\G^{i}_{x,r}
\cong
\mfP_{i}^{\lceil er\rceil}
\]
(recall that $e$ is the ramification index of $E/F$) for any $r\in\R$.
Hence we have an isomorphism
\[
\Lie\G^{i+1}_{x,\frac{r_{i}}{2}}\big/\bigl(\Lie\G^{i}_{x,\frac{r_{i}}{2}}+\Lie\G^{i+1}_{x,\frac{r_{i}}{2}+}\bigr)
\cong
\mfP_{i+1}^{j_{i}}/(\mfP_{i}^{j_{i}}+\mfP_{i+1}^{h_{i}}),
\]
which is $\bfS(F)$-equivariant.
Thus, in particular, $V_{\alpha}$ is isomorphic to $\mfV_{[g]}$.
\end{proof}

\section{Comparison of $\chi$-data}\label{sec:main}

\subsection{Main theorem}\label{subsec:main}

\begin{thm}\label{thm:main}
We assume that the residual characteristic $p$ is odd.
Then Kaletha's and Harris--Taylor's local Langlands correspondences coincide for regular supercuspidal representations of $\GL_{n}(F)$.
That is, for every regular supercuspidal representation $\pi$ of $\GL_{n}(F)$, we have
\[
\LLC_{\G}^{\Kal}(\pi)
=
\LLC_{\G}^{\HT}(\pi).
\]
\end{thm}

\begin{proof}
Let $\pi$ be a regular supercuspidal representation of $\GL_{n}(F)$ arising from a tame elliptic regular pair $(\bfS,\xi)$, i.e., $\pi\cong\pi_{(\bfS,\xi)}^{\KY}$.
Here if we let $E$ be a finite tamely ramified extension of $F$ of degree $n$ such that $\bfS$ is isomorphic to $\Res_{E/F}\Gm$, then we have $\pi_{(\bfS,\xi)}^{\KY}\cong\pi_{(E,\xi)}^{\BH}$ (see Section \ref{subsec:rsc} and Appendix \ref{sec:BH-Kal}).
Then, by Corollary \ref{cor:BH} and Proposition \ref{prop:Kaletha}, we have
\[
\LLC_{\G}^{\HT}(\pi)=\Ind_{W_{E}}^{W_{F}}(\xi\mu_{\chi_{\Tam}}^{-1})
\quad\text{and}\quad
\LLC_{\G}^{\Kal}(\pi)=\Ind_{W_{E}}^{W_{F}}(\epsilon^{-1}\xi\mu_{\chi_{\Kal}}).
\]
Thus it suffices to show the equality
\[
\mu_{\chi_{\Tam}}^{-1}
=
\epsilon^{-1}\mu_{\chi_{\Kal}}
\]
of characters on $E^{\times}$.
Here let us recall that each character appearing in this equality is defined as follows:
\begin{itemize}
\item
For a set $\chi$ of $\chi$-data, $\mu_{\chi}$ is defined to be a character of $E^{\times}$ given by
\[
\mu_{\chi}:=
\prod_{[\alpha]\in\Gamma_{F}\backslash\Phi(\bfS,\G)} \chi_{\alpha}|_{E^{\times}},
\]
where, for each $[\alpha]\in\Gamma_{F}\backslash\Phi(\bfS,\G)$, we take its representative $\alpha\in\Phi(\bfS,\G)$ to be of the form $\begin{bmatrix}1\\g\end{bmatrix}$ for some (unique) $g\in\mathcal{D}$ (see Section \ref{subsec:Tam}).
\item
The character $\epsilon$ is defined as the product $\epsilon^{\sym}\cdot\epsilon_{\sym,\ur}\cdot\epsilon_{f,\ram}$, where
\begin{itemize}
\item
$\epsilon^{\sym}$ is the product of $\epsilon_{\alpha}$ over the set of $\Gamma_{F}\times\{\pm1\}$-orbits of asymmetric roots, 
\item
$\epsilon_{\sym,\ur}$ is the product of $\epsilon_{\alpha}$ over the set of $\Gamma_{F}$-orbits of symmetric unramified roots, and
\item
$\epsilon_{f,\ram}$ is the contribution of symmetric ramified roots, which is in fact trivial when specialized to $\GL_{n}$ (Proposition \ref{prop:toral}).
\end{itemize}
\end{itemize}
Therefore the above equality follows from the following proposition.
\end{proof}

\begin{prop}\label{prop:main}
Let $(\bfS,\xi)$ be a tame elliptic regular pair of $\GL_{n}$ and $E$ a finite tamely ramified extension of $F$ of degree $n$ corresponding to $\bfS$.
Let $\alpha\in\Phi(\bfS,\G)$ be a root of the form $\begin{bmatrix}1\\g\end{bmatrix}$ for some $g\in\mathcal{D}$.
When $\alpha$ is asymmetric, we put $\alpha'$ to be the root $\begin{bmatrix}1\\g'\end{bmatrix}$ (for some $g'\in\mathcal{D}$) whose $\Gamma_{F}$-orbit contains $-\alpha$.
Then we have
\[
\begin{cases}
(\chi_{\Tam,\alpha}\cdot\chi_{\Tam,\alpha'})|_{E^{\times}}^{-1}=\epsilon_{\alpha}^{-1}\cdot(\chi_{\Kal,\alpha}\cdot\chi_{\Kal,\alpha'})|_{E^{\times}} & \text{if $\alpha\in\Phi(\bfS,\G)^{\sym}$},\\
\chi_{\Tam,\alpha}|_{E^{\times}}^{-1}=\epsilon_{\alpha}^{-1}\cdot\chi_{\Kal,\alpha}|_{E^{\times}} & \text{if $\alpha\in\Phi(\bfS,\G)_{\sym,\ur}$},\\
\chi_{\Tam,\alpha}|_{E^{\times}}^{-1}=\chi_{\Kal,\alpha}|_{E^{\times}} & \text{if $\alpha\in\Phi(\bfS,\G)_{\sym,\ram}$}.
\end{cases}
\]
\end{prop}

We will prove this proposition by a case-by-case computation according to whether a root $\alpha$ is symmetric or asymmetric (and furthermore whether unramified or ramified when it is symmetric). 
In the rest of this section, we fix a tame elliptic regular pair $(\bfS,\xi)$ of $\GL_{n}$.
We let $E$ be the finite tamely ramified extension of $F$ of degree $n$ which corresponds to $\bfS$.
Moreover, according to Tam (see \cite[1735 page, (5.11)]{MR3509939}), we add the following condition to our fixed uniformizer $\varpi_{E}$:
\[
\varpi_{E}\in E_{0},
\]
where $E_{0}$ is the subfield of $E$ appearing in a Howe factorization of $(\bfS,\xi)$ (in the sense of Howe and Moy, see Section \ref{subsec:Howe}).
Note that we can always assume this condition since $E$ is unramified over $E_{0}$ (see, e.g., \cite[Section 5.1]{MR3509939}).

\subsection{Asymmetric roots}\label{subsec:asymm}
Let $\alpha =\begin{bmatrix}1 \\g\end{bmatrix}\in \Phi(\bfS,\G)^{\sym}$ be an asymmetric root.
We take $\alpha'$ as in Proposition \ref{prop:main}.
We recall the definitions of the terms on both sides of the equality in Proposition \ref{prop:main}.
\begin{description}
	\item[Kaletha's $\chi$-data at $\alpha$] 
	We take $\chi_{\Kal,\alpha}$ to be the trivial character of $F_{\alpha}^{\times}$.	
\item[Tam's $\chi$-data at $\alpha$] 
Since what we need is only the restriction of $\chi_{\Tam,\alpha}\cdot\chi_{\Tam,\alpha'}$ to $E^{\times}$, here we recall the value of $\chi_{\Tam,\alpha}\cdot\chi_{\Tam,\alpha'}$ at each element of $E^{\times}$: 
\begin{itemize}
\item
$\chi_{\Tam,\alpha}\cdot\chi_{\Tam,\alpha'}$ is trivial on $1+\mfp_{E}$,
\item
$(\chi_{\Tam,\alpha}\cdot\chi_{\Tam,\alpha'})(\gamma)=\sgn_{\gamma}(\mfV_{[g]})\cdot\sgn_{\gamma}(\mfV_{[g^{-1}]})$ for any $\gamma\in\mu_{E}$, and
\item
$(\chi_{\Tam,\alpha}\cdot\chi_{\Tam,\alpha'})(\varpi_{E})=t_{\varpi}(\boldsymbol{\mfV}_{[g]})(\varpi_{E})$.
\end{itemize}
Here $\sgn_{\gamma}(\mfV_{[g]})$ (resp.\ $\sgn_{\gamma}(\mfV_{[g^{-1}]})$) is the signature of the permutation of $\mfV_{[g]}$ (resp.\ $\mfV_{[g^{-1}]}$) given by the multiplication by $\gamma$ (recall that we have identifications $\mfV_{[g]}\cong k_{F_{\alpha}}$ and $\mfV_{[g^{-1}]}\cong k_{F_{\alpha'}}$, see Proposition \ref{prop:U-decomp} and Remark \ref{rem:U-decomp}).

\item[$\varepsilon$ at $\alpha$] 
Recall that, by supposing that $\alpha$ belongs to $\Phi_{i+1}^{i}$ for $-1\leq i\leq d-1$, we put
\[
\epsilon_{\alpha}(\gamma)
:=
\begin{cases}
\Jac{\overline{\alpha(\gamma)}}{k_{F_{\alpha}}^{\times}} &\text{if $i\neq-1$ and $\frac{r_{i}}{2}\in\ord_{x}(\alpha)$,}\\
1&\text{otherwise},
\end{cases}
\]
for any $\gamma\in E^{\times}$.
\end{description}

\begin{proof}[Proof of Proposition \ref{prop:main}: asymmetric case]
If $\mfV_{[g]}$ is trivial, then $\chi_{\Tam,\alpha}\cdot \chi_{\Tam,\alpha'}|_{E^{\times}}$ is trivial.
On the other hand, by Proposition \ref{prop:Heisenberg}, $\epsilon_{\alpha}$ is also trivial.
Thus we are done in this case.

In the following, we assume that $\mfV_{[g]}\neq0$, or equivalently by Proposition \ref{prop:Heisenberg}, $i\neq-1$ and $\frac{r_{i}}{2}\in\ord_{x}(\alpha)$.
Since $\epsilon_{\alpha}$ and $\chi_{\Tam,\alpha}\cdot \chi_{\Tam,\alpha'}|_{E^{\times}}$ are trivial on $1+\mfp_{E}$, it is enough to check the equality 
\[
\epsilon_{\alpha}(\gamma)=(\chi_{\Tam,\alpha}\cdot\chi_{\Tam,\alpha'})(\gamma)
\]
for $\gamma=\varpi_{E}$ and each $\gamma\in\mu_{E}$.

First, as noted in \cite[Theorem 7.1 (i) (b)]{MR3509939}, we have
\[
t_{\varpi}(\boldsymbol{\mfV}_{[g]})(\varpi_{E})
=
\sgn_{\alpha(\varpi_{E})}(\mfV_{[g]})
\]
by \cite[Proposition 4.9 (ii) (a)]{MR3509939} (note that $\boldsymbol{\mfV}_{[g]}$ on the right-hand side of the last equality in \cite[Theorem 7.1 (i) (b)]{MR3509939} should be correctly $\mfV_{[g]}$).
Second, by \cite[1758 page, (7.2)]{MR3509939} and the equality preceding \cite[1758 page, (7.2)]{MR3509939}, we have
\[
(\chi_{\Tam,\alpha}\cdot\chi_{\Tam,\alpha'})(\gamma)
=
\sgn_{\alpha(\gamma)}(\mfV_{[g]})
\]
for any $\gamma\in\mu_{E}$.

Thus, in any case, it suffices to show the equality
\[
\sgn_{\zeta}(\mfV_{[g]})
=
\Jac{\zeta}{k_{F_{\alpha}}^{\times}}
\]
for any $\zeta\in k_{F_{\alpha}}^{\times}$.
Furthermore, by noting that $k_{F_{\alpha}}^{\times}$ is a cyclic group, it is enough to show the above equality only for a generator $\zeta$ of $k_{F_{\alpha}}^{\times}$.
Then the right-hand side of the above equality is given by $-1$.
On the other hand, since $k_{F_{\alpha}}^{\times}$ acts on $\mfV_{[g]}$ via the identification $\mfV_{[g]}\cong k_{F_{\alpha}}$, the action of $\zeta$ is given by a cyclic permutation of $(|k_{F_{\alpha}}|-1)$-letters (note that only $0\in k_{F_{\alpha}}$ is fixed).
Thus it can be written as a product of $(|k_{F_{\alpha}}|-2)$ transpositions.
As $|k_{F_{\alpha}}|$ is odd by the oddness assumption on $p$, we finally get
\[
\sgn_{\zeta}(\mfV_{[g]})
=
(-1)^{|k_{F_{\alpha}}|-2}
=-1.
\]
\end{proof}

\subsection{Symmetric unramified roots}\label{subsec:symmur}
Let $\alpha =\begin{bmatrix}1\\g\end{bmatrix}\in \Phi (\bfS,\G)_{\sym, \ur}$ be a symmetric unramified root with $g=\sigma ^{k}\phi^{i}$.
Note that, as explained in Section \ref{subsec:root} (Proposition \ref{prop:symmur}), $i$ equals either $0$ or $\frac{f}{2}$.
Moreover, when $i=0$, $k$ is not equal to $\frac{e}{2}$.
We recall the definitions of the terms on both sides of the equality in Proposition \ref{prop:main}.
\begin{description}
\item[Kaletha's $\chi$-data at $\alpha$] 
We take $\chi_{\Kal,\alpha}$ to be the unique nontrivial unramified quadratic character of $F_{\alpha}^{\times}$.
\item[Tam's $\chi$-data at $\alpha$]
As in the asymmetric case, we recall only the values of $\chi_{\Tam,\alpha}$ on $E^{\times}$.
By \cite[Theorem 7.1 (i)]{MR3509939}, $\chi_{\Tam,\alpha}$ is a character of $F_{\alpha}^{\times}$ satisfying the following conditions:
\begin{itemize}
\item
$\chi_{\Tam,\alpha}$ is trivial on $1+\mfp_{E}$,
\item
$\chi_{\Tam,\alpha}|_{\mu_{E}}=t^1_{\boldsymbol{\mu}}(\mfV_{[g]})$, and
\item
$\chi_{\Tam,\alpha}(\varpi_{E})=t^0_{\boldsymbol{\mu}}(\mfV_{[g]}^{\varpi})\cdot t_{\varpi}(\mfV_{[g]})\cdot t(\mfW_{[g]})$.
\end{itemize}
Here $\mfV_{[g]}^{\varpi}$ denotes the $\varpi$-fixed part of $\mfV_{[g]}$.

\item[$\epsilon$ at $\alpha$] 
Recall that, by supposing that $\alpha$ belongs to $\Phi_{i+1}^{i}$ for $-1\leq i\leq d-1$, we put
\[
\epsilon_{\alpha}(\gamma)
:=
\begin{cases}
\Jac{\overline{\alpha(\gamma)}}{k_{F_{\alpha}}^{1}}&\text{if $i\neq-1$ and $\frac{r_{i}}{2}\in\ord_{x}(\alpha)$,}\\
1&\text{otherwise}.
\end{cases}
\]
\end{description}

\begin{prop} \label{prop:symmunrammu}
We have the equality in Proposition \ref{prop:main}, which is
\[
\chi_{\Tam,\alpha}^{-1}(\zeta)=\epsilon_{\alpha}^{-1}(\zeta)\cdot\chi_{\Kal,\alpha}(\zeta),
\]
for every $\zeta \in \mu_{E}$.
\end{prop}
\begin{proof}
As $\chi_{\Kal,\alpha}|_{\mu_{E}}$ is trivial,
we are going to prove
\begin{equation} \label{eq:symmunrammu}
\epsilon_{\alpha} (\zeta)=\chi_{\Tam,\alpha}(\zeta)
\end{equation}
for $\zeta\in \mu_{E}$.

If $\mfV_{[g]}$ is trivial, then $\chi_{\Tam,\alpha}|_{\mu_{E}}$ is trivial.
On the other hand, by Proposition \ref{prop:Heisenberg}, $\epsilon_{\alpha}$ is also trivial.
Thus the equation \eqref{eq:symmunrammu} follows.

In the following, we assume that $\mfV_{[g]}\neq0$, or equivalently by Proposition \ref{prop:Heisenberg}, $i\neq-1$ and $\frac{r_{i}}{2}\in\ord_{x}(\alpha)$.
Then by \cite[Proposition 4.9 (i) (b) and (c)]{MR3509939} we have
\[
\chi_{\Tam,\alpha}(\zeta)=t^1_{\boldsymbol{\mu}}(\mfV_{[g]})(\zeta)=
\begin{cases}
1 &\text{if $g=\sigma^{k}$ ($k\neq \frac{e}{2}$),} \\
\Jac{\zeta}{\mu_{E}} &\text{if $g=\sigma^{k}\phi^{\frac{f}{2}}$ ($f$ is even)}.
\end{cases}
\]

If $g=\sigma^{k}$ with $k\neq \frac{e}{2}$, then we have $\alpha (\zeta)=\zeta \cdot \sigma^{k}(\zeta)^{-1}=1$ 
and the equality \eqref{eq:symmunrammu} holds.

If $g=\sigma^{k}\phi^{\frac{f}{2}}$, then we have $\alpha (\zeta)=\zeta \cdot \sigma^{k}\phi^{\frac{f}{2}}(\zeta)^{-1}=\zeta^{1-q^{\frac{f}{2}}}$. 
Thus we have
\begin{align*}
\epsilon_{\alpha} (\zeta)&=\Jac{\alpha(\zeta)}{k_{F_{\alpha}}^1} 
=\Jac{\zeta^{1-q^{\frac{f}{2}}}}{k_{F_{\alpha}}^1}=\zeta^{\frac{\bigl(1-q^{\frac{f}{2}}\bigr)\bigl(q^{\frac{fr}{2}}+1\bigr)}{2}}, \\
\chi_{\Tam,\alpha}(\zeta)&=\Jac{\zeta}{\mu_{E}}=\zeta^{\frac{q^f-1}{2}},
\end{align*}
where $r=[F_{\alpha}:E]$.
Hence, in order to check the equality (\ref{eq:symmunrammu}), it suffices to show that
\[
\frac{\bigl(1-q^{\frac{f}{2}}\bigr)\bigl(q^{\frac{fr}{2}}+1\bigr)}{2}
-
\frac{q^f-1}{2}
\equiv0
\mod{q^{f}-1}.
\]
Here we note that $r$ is odd by \cite[Proposition 3.4]{MR3509939} (as explained in  Section \ref{subsec:root}, our root $\alpha$ is symmetric unramified also in the sense of Tam).
Thus we get
\[
q^{\frac{fr}{2}}+1
=
\bigl(1-q^{\frac{f}{2}}+q^{f}-\cdots+q^{\frac{f(r-1)}{2}}\bigr)\bigl(1+q^{\frac{f}{2}}\bigr),
\]
hence
\[
\frac{\bigl(1-q^{\frac{f}{2}}\bigr)\bigl(q^{\frac{fr}{2}}+1\bigr)}{2}
-
\frac{q^f-1}{2}
=
\frac{1-q^{f}}{2}
\bigl(
1-q^{\frac{f}{2}}+q^{f}-\cdots+q^{\frac{f(r-1)}{2}}+1
\bigr).
\]
As $1-q^{\frac{f}{2}}+q^{f}-\cdots+q^{\frac{f(r-1)}{2}}+1$ is even, this is equal to $0$ modulo $q^{f}-1$.
\end{proof}

\begin{prop} \label{prop:symmunramvarpi}
We have the equality in Proposition \ref{prop:main} for $\varpi_{E}$, that is,
\[
\chi_{\Tam,\alpha}^{-1}(\varpi_{E})=\epsilon_{\alpha}^{-1}(\varpi_{E})\cdot\chi_{\Kal,\alpha}(\varpi_{E}).
\]
\end{prop}

\begin{proof}
Since the extension $F_{\alpha}/E$ is unramified, we have $\chi_{\Kal,\alpha}(\varpi_{E})=-1$.
Thus it is enough to show that
\[
(\epsilon_{\alpha} \cdot \chi_{\Tam,\alpha}^{-1})(\varpi_{E})=-1.
\]
Recall that, by Proposition \ref{prop:Heisenberg}, we have 
\[
\epsilon_{\alpha}(\varpi_{E})
=
\begin{cases}
1 &\text{if $\mfV_{[g]}$ is trivial}, \\
\Jac{\alpha(\varpi_{E})}{k_{F_{\alpha}}^{1}}  & \text{if $\mfV_{[g]}$ is nontrivial}.
\end{cases}
\]
On the other hand, by noting that $\mfW_{[g]}$ is trivial if and only if $\mfV_{[g]}$ is nontrivial and that we defined $t_{\varpi}(\mfV_{[g]})$ to be the product $t^{0}_{\varpi}(\mfV_{[g]})\cdot t^{1}_{\varpi}(\mfV_{[g]})$, we have
\[
\chi_{\Tam,\alpha}(\varpi_{E})
=\begin{cases}
t(\mfW_{[g]}) &\text{if $\mfV_{[g]}$ is trivial,} \\
t^0_{\boldsymbol{\mu}}(\mfV_{[g]}^{\varpi})\cdot t^{0}_{\varpi}(\mfV_{[g]})\cdot t^{1}_{\varpi}(\mfV_{[g]}) &\text{if $\mfV_{[g]}$ is nontrivial}.
\end{cases}
\]

If $\mfV_{[g]}$ is trivial, then we have $t(\mfW_{[g]})=t(\mfU_{[g]})=-1$ by \cite[1728 page, (4.17)]{MR3509939} and get the desired equality.

In the rest of the proof, we assume that $\mfV_{[g]}$ is nontrivial and put
\[
\beta :=\alpha (\varpi_{E})=\begin{cases}
\zeta_{e}^{k} &\text{if $g=\sigma^{k}$ ($k\neq\frac{e}{2}$),} \\
\zeta_{e}^{k}\zeta_{\phi}^{\frac{q^{\frac{f}{2}}-1}{q-1}} &\text{if $g=\sigma^{k}\phi^{\frac{f}{2}}$ ($f$ is even)}.
\end{cases}
\]
Note that $\beta \neq \pm 1$ in the first case.
As explained in \cite[Remark 7.3]{MR3509939}, we have
\[
t^0_{\boldsymbol{\mu}}(\mfV_{[g]}^{\varpi})
=
\begin{cases}
-1 &\text{if $g=\sigma^{k}\phi^{\frac{f}{2}}$, $\beta=1$},\\
1 &\text{otherwise}.
\end{cases}
\]
On the other hand, by \cite[Proposition 4.9 (ii) (b) and (c)]{MR3509939}, we have 
\begin{align*}
t^{0}_{\varpi}(\mfV_{[g]})&=
\begin{cases}
1 &\text{if $g=\sigma^{k}\phi^{\frac{f}{2}}$, $\beta=\pm 1$}, \\
-1 &\text{otherwise},
\end{cases} \\
t^{1}_{\varpi}(\mfV_{[g]})&=
\begin{cases}
1 &\text{if $g=\sigma^{k}\phi^{\frac{f}{2}}$, $\beta =1$}, \\
(-1)^{\frac{q^{\frac{f}{2}}-1}{2}} &\text{if $g=\sigma^{k}\phi^{\frac{f}{2}}$, $\beta =-1$}, \\
\Jac{\beta}{\F_{p}[\beta]^1} &\text{otherwise}.
\end{cases}
\end{align*}

First, if $g=\sigma ^k\phi ^{\frac{f}{2}}$ and $\beta=1$,
then we can easily check the desired equality.

Next suppose that $g=\sigma ^{k}\phi ^{\frac{f}{2}}$ and $\beta =-1$.
Then putting $r=[F_{\alpha}:E]$, we have
\begin{align*}
(\epsilon_{\alpha} \cdot \chi_{\Tam,\alpha}^{-1})(\varpi_{E})&=\Jac{-1}{k_{F_{\alpha}}^{1}}\cdot (-1)^{\frac{1-q^{\frac{f}{2}}}{2}} \\
&=(-1)^{\frac{q^{\frac{fr}{2}}+1}{2}}\cdot (-1)^{\frac{1-q^{\frac{f}{2}}}{2}}\\
&=(-1)^{\frac{q^{\frac{fr}{2}}-q^{\frac{f}{2}}}{2}+1}.
\end{align*}
Here, again by noting that $r$ is odd as explained in the proof of Proposition \ref{prop:symmunrammu}, we have
\[
q^{\frac{fr}{2}}-q^{\frac{f}{2}}
=q^{\frac{f}{2}}(q^{\frac{f}{2}}-1)\bigl(q^{\frac{f(r-2)}{2}}+\cdots+q^{\frac{f}{2}}+1\bigr)
\equiv0
\mod{4}.
\]
This implies that $(\epsilon_{\alpha} \cdot \chi_{\Tam,\alpha}^{-1})(\varpi_{E})=-1$, as desired.

Finally suppose that $\beta \neq \pm 1$.
Then we have $t^0_{\boldsymbol{\mu}}(\mfV_{[g]}^{\varpi})\cdot t^{0}_{\varpi}(\mfV_{[g]})=-1$ and thus we are going to show that
\begin{equation} \label{eq:symurpi}
\Jac{\beta}{k_{F_{\alpha}}^{1}}=\Jac{\beta}{\F_{p}[\beta]^1}.
\end{equation}
Before we start our computation, we remark that the fact $\beta$ belongs to $\F_{p}[\beta]^{1}$ can be checked as follows.
First, as already explained in Section \ref{subsec:epsilon}, $\beta=\alpha(\varpi_{E})$ belongs to $k_{F_{\alpha}}^{1}$.
By combining this fact with our assumption that $\beta$ is not equal to $\pm1$, we know that it does not lie in $k_{F_{\pm\alpha}}$.
Thus $\F_{p}[\beta]':=\F_{p}[\beta]\cap k_{F_{\pm\alpha}}$ gives the unique subfield of $\F_{p}[\beta]$ such that $[\F_{p}[\beta]:\F_{p}[\beta]']=2$.
Moreover, as the Galois group for $k_{F_{\alpha}}/k_{F_{\pm\alpha}}$ can be naturally identified with that for $\F_{p}[\beta]/\F_{p}[\beta]'$ by restriction, we have $k_{F_{\alpha}}^{1}\cap\F_{p}[\beta]=\F_{p}[\beta]^{1}$.
Hence $\beta$ belongs to $\F_{p}[\beta]^{1}$.

Now let us show the above equality (\ref{eq:symurpi}).
If we put $t:=[k_{F_{\alpha}}:\F_{p}[\beta]]$ and $q_{\beta}:=|\F_p[\beta]'|$, then we have $|k_{F_{\alpha}}|=q_{\beta}^{2t}$.
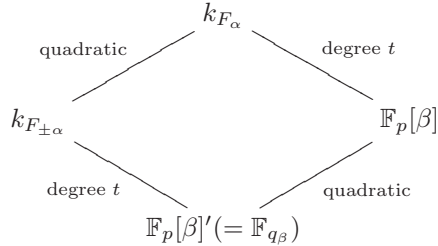
\[
\xymatrix{
&k_{F_{\alpha}}&\\
k_{F_{\pm\alpha}}\ar@{-}[ur]^-{\text{quadratic}\quad}&&\F_{p}[\beta]\ar@{-}[ul]_-{\quad\text{degree $t$}}\\
&\F_{p}[\beta]' (=\F_{q_{\beta}})\ar@{-}[ul]^-{\text{degree $t$}\quad}\ar@{-}[ur]_-{\quad\text{quadratic}}&
}
\]
Thus we get
\[
\Jac{\beta}{k_{F_{\alpha}}^{1}}=\beta^{\frac{q_{\beta}^{t}+1}{2}}
\quad\text{and}\quad
\Jac{\beta}{\F_{p}[\beta]^1}=\beta^{\frac{q_{\beta}+1}{2}}.
\]
Since $\beta \in \F_{p}[\beta]^1$ satisfies $\beta^{q_{\beta}+1}=1$, it suffices to prove 
\[
\frac{q_{\beta}^{t}+1}{2}
-
\frac{q_{\beta}+1}{2}
\equiv0
\mod{q_{\beta}+1}.
\]
We have
\[
\frac{q_{\beta}^{t}+1}{2}
-
\frac{q_{\beta}+1}{2}
=\frac{q_{\beta}(q_{\beta}-1)}{2}\cdot(q_{\beta}^{t-2}+\cdots+q_{\beta}+1).	
\]
Here we note that $t$ is an odd integer.
Indeed, 
if we suppose that $t=2s$ were even, then we would have
\[
\beta^{q_{\beta}^{t}+1}=\beta^{q_{\beta}^{2s}+1}=\beta^{(q_{\beta}^{2})^{s}}\cdot \beta=\beta^2.
\]
However, as $\beta$ belongs to $k_{F_{\alpha}}^{1}$, we would have $\beta^{q_{\beta}^{t}+1}=1$ and this contradicts the assumption that $\beta \neq \pm 1$.
Therefore, by the oddness of $t$, we can conclude that $q_{\beta}+1$ divides $q_{\beta}^{t-2}+\cdots+q_{\beta}+1$ and this completes the proof.
\end{proof}

\begin{proof}[Proof of Proposition \ref{prop:main}: symmetric unramified case]
Since the characters on both sides of the equality in Proposition \ref{prop:main} are tamely ramified, they are equal by Propositions \ref{prop:symmunrammu} and \ref{prop:symmunramvarpi}.
\end{proof}

\subsection{Symmetric ramified roots}\label{subsec:symmram}
In this section, we consider the case of symmetric ramified roots.
First we recall that the elliptic torus $\bfS$ in $\G$ has a symmetric ramified root only when the ramification index $e$ of the extension $E/F$ is even (see Proposition \ref{prop:symmram}).
Thus, in this section, we assume that the ramification index $e$ is even.
Under this assumption, there exists a unique $\Gamma_{F}$-orbit of symmetric ramified roots and it is represented by $\alpha:=\begin{bmatrix}1\\\sigma^{\frac{e}{2}}\end{bmatrix}$.
As explained in the proof of Proposition \ref{prop:symmram}, the element $g=\sigma^{\frac{e}{2}}$ of $\Gamma_{F}$ preserves $E$.
Note that thus the field $F_{\alpha}$ is nothing but $E$.

We first take a Howe factorization $(\phi_{-1},\ldots,\phi_{d})$ of $(\bfS,\xi)$ (see Section \ref{subsec:Howe}).
Let $\bfi$ be the unique index such that $-1\leq \bfi \leq d-1$ and the symmetric ramified root $\alpha$ lies in $\Phi_{\bfi+1}^{\bfi}=\Phi(\bfS, \G^{\bfi+1})-\Phi(\bfS, \G^{\bfi})$.
Note that $\bfi\neq-1$ since every symmetric root belonging to $\Phi(\bfS,\G^{0})$ is unramified (see \cite[the paragraph before Corollary 4.10.1]{MR4013740}).
We simply write $r$ for the depth $r_{\bfi}$ of the character $\phi_{\bfi}$.

Now let us recall the definitions of $\chi_{\Kal,\alpha}$ and $\chi_{\Tam,\alpha}$ for our symmetric ramified root $\alpha$:

\begin{description}
\item[Kaletha's $\chi$-data at $\alpha$]
We take $\chi_{\Kal,\alpha}$ to be a character of $E^{\times}$ satisfying the following conditions:
\begin{itemize}
\item
$\chi_{\Kal,\alpha}$ is trivial on $1+\mfp_{E}$,
\item
$\chi_{\Kal,\alpha}|_{\mu_{E}}$ is the unique nontrivial quadratic character of $\mu_{E}$, and
\item
$\chi_{\Kal,\alpha}(2a_{\alpha})=\lambda_{E/F_{\pm\alpha}}$,
\end{itemize}
where $a_{\alpha}$ is an (any) element of $E_{-r}$ which is a lift of the unique element $\bar{a}_{\alpha}$ of $E_{-r}/E_{-r+}$ satisfying the equality
\[
\phi_{\bfi}\bigl(\Nr_{\bfS(E)/\bfS(F)}(\alpha^{\vee}(X+1))\bigr)
=
\psi_{E}(\bar{a}_{\alpha}\cdot X)
\]
for every $X\in E_{r}/E_{r+}$ (see \cite[1133 page, (4.7.3)]{MR4013740}).
Here $\psi_{E}$ is an additive character of $E$ taken as in Section \ref{sec:notation}, $\lambda_{E/F_{_{\pm\alpha}}}$ is the Langlands constant (see Section \ref{sec:notation}), and $\Nr_{\bfS(E)/\bfS(F)}$ is the norm map from $\bfS(E)$ to $\bfS(F)$.
Note that the third condition on $\chi_{\Kal,\alpha}$, when combined with the first two, is enough to characterize $\chi_{\Kal,\alpha}$ since $r$ is given by $\frac{2s+1}{e}$ with some odd integer $2s+1$ (see \cite[Section 4.7]{MR4013740}) and $\chi_{\Kal,\alpha}|_{F_{\pm\alpha}^{\times}}$ should be the quadratic character corresponding to the extension $E/F_{\pm\alpha}$ by the definition of $\chi$-data.

\item[Tam's $\chi$-data at $\alpha$]
We take $\chi_{\Tam,\alpha}$ to be the unique character of $E^{\times}$ satisfying the following conditions:
\begin{itemize}
\item
$\chi_{\Tam,\alpha}$ is trivial on $1+\mfp_{E}$,
\item
$\chi_{\Tam,\alpha}|_{\mu_{E}}$ is the unique nontrivial quadratic character of $\mu_{E}$, and
\item
$\chi_{\Tam,\alpha}(\varpi_{E})=t_{\boldsymbol{\mu}}^{0}(\mfV_{[g]}^{\varpi})\cdot t_{\varpi}(\mfV_{[g]})\cdot t(\mfW_{[g]})$.
\end{itemize}
Note that, by Proposition \ref{prop:Heisenberg} and the oddness of the numerator of $r$ explained above, the module $\mfV_{[g]}$ is zero.
In particular, two t-factors $t_{\boldsymbol{\mu}}^{0}(\mfV_{[g]}^{\varpi})$ and $t_{\varpi}(\mfV_{[g]})$ appearing in the definition of $\chi_{\Tam,\alpha}(\varpi_{E})$ are in fact trivial.
Thus the key in our calculation of Tam's $\chi$-data is to determine the factor $t(\mfW_{[g]})$ explicitly.
\end{description}

\begin{rem}\label{rem:Howe}
A priori the definition of $\chi_{\Kal,\alpha}$ may depend on the choice of a Howe factorization of $\xi$.
In fact, also for $\chi_{\Tam,\alpha}$, a priori there may be such a dependence although we cannot see it at this point since we have not recalled the precise definition of $\chi_{\Tam,\alpha}$ (i.e., the definition of $t(\mfW_{[g]})$) up to now.
However, in fact, they are independent of the choice of a Howe factorization.
We can easily check this by noting that any two different Howe factorizations are ``refactorizations'' of each other (see \cite[Lemma 3.6.6]{MR4013740}, or also \cite[Proposition 1.1 (2)]{MR2148193}).
\end{rem}

In order to compare $\chi_{\Kal,\alpha}$ with $\chi_{\Tam,\alpha}$, we next recall the following factorization of $E/F$, which is used in Bushnell--Henniart's computation of the rectifier:
\[
F=K_{-1}
\subsetneq
K_{0}
\subsetneq
K_{1}
\subsetneq
\cdots
\subsetneq
K_{l-1}
\subsetneq
K_{l}
\subsetneq
K_{l+1}=E,
\]
where
\begin{itemize}
\item
$K_{0}/K_{-1}$ is an unramified extension,
\item
$K_{i+1}/K_{i}$ is a quadratic ramified extension for each $0\leq i\leq l-1$, and
\item
$K_{l+1}/K_{l}$ is a totally ramified extension of odd degree (we put $m$ to be this degree).
\end{itemize}
Note that such a factorization is unique and given explicitly by
\[
K_{0}=F[\mu_{E}],\,\,\ldots,\,\,
K_{l-1}=F[\mu_{E},\varpi_{E}^{2m}],\,\,
K_{l}=F[\mu_{E},\varpi_{E}^{m}],\,\,
K_{l+1}=F[\mu_{E},\varpi_{E}].
\]
Here, by our assumption on the ramification index, we have $l\geq1$, i.e., we have at least one quadratic ramified extension in the above sequence.
We also note that, since the action of $\sigma^{\frac{e}{2}}$ on $E$ is given by 
\[
\zeta \mapsto \zeta \quad (\zeta \in \mu_{E})
\quad\text{and}\quad 
\varpi_{E} \mapsto -\varpi_{E},
\]
$\sigma^{\frac{e}{2}}$ belongs to $\Gamma_{K_{l-1}}-\Gamma_{K_{l}}$.

Now we consider the tame twisted Levi subgroup $\bfH$ of $\G$ defined to be the centralizer of $K_{l-1}^{\times}\subset E^{\times}\cong\bfS(F)$ (note that then $\bfH$ is isomorphic to $\Res_{K_{l-1}/F}\GL_{2m}$) and a Howe factorization of $\xi$ with respect to $\bfS\subset\bfH$.
Then, by Proposition \ref{prop:Howe-rest}, we can construct a Howe factorization $(\phi'_{-1},\ldots,\phi'_{d'})$ of $\xi$ with respect to $\bfS\subset\bfH$ by using the Howe factorization $(\phi_{-1},\ldots,\phi_{d})$ of $\xi$ with respect to $\bfS\subset\G$ satisfying the following equality for each $0\leq i \leq d'$:
\begin{equation}\label{eq:i=i'}
\phi_{[i]}|_{\bfS(F)_{r_{[i]}}}
=
\phi'_{i}|_{\bfS(F)_{r'_{i}}}.
\end{equation}
Here recall that $r'_{0}<\cdots<r'_{d'}$ (resp.\ $r_{0}<\cdots<r_{d}$) is a sequence of jumps of $\xi$ with respect to $\bfS\subset\bfH$ (resp.\ $\bfS\subset\G$) and $[i]$ denotes the unique index satisfying $r_{[i]}=r'_{i}$.
If we put $-1\leq\bfi'\leq d'-1$ to be the unique index satisfying $\alpha\in\Phi(\bfS,\bfH^{\bfi'+1})-\Phi(\bfS,\bfH^{\bfi'})$, then we have $[\bfi']=\bfi$.
Note that the existence of such an index (or, in other words, the fact that $\alpha$ lies in $\Phi(\bfS,\bfH)$) is guaranteed by that $\sigma^{\frac{e}{2}}$ belongs to $\Gamma_{K_{l-1}}$.

Here we remark that, in general, a Howe factorization of $\xi$ with respect to $\bfS\cong\Res_{E/F}\Gm\subset\bfH\cong\Res_{K_{l-1}/F}\GL_{2m,K_{l-1}}$ can be regarded as that with respect to $\Res_{E/K_{l-1}}\Gm\subset\GL_{2m,K_{l-1}}$ (and vice versa).
This can be easily checked by noting that the root system $\Phi(\bfS,\bfH)$ for $\bfS\subset\bfH$ is given by $\Phi(\Res_{E/K_{l-1}}\Gm,\GL_{2m,K_{l-1}})\rtimes_{\Gamma_{K_{l-1}}}\Gamma_{F}$.
Then, by seeing our descended Howe factorization $(\phi'_{-1},\ldots,\phi'_{d'})$ as the latter one, it can be also understood as a classical Howe factorization in the sense of Howe and Moy (see Section \ref{subsec:Howe}).
In the language of Howe and Moy, they are described as follows.
We have a sequence of subfields
\[
K_{l-1}=E'_{d'}
\subsetneq
\cdots
\subsetneq
E'_{0}
\subset
E'_{-1}=E,
\]
and characters $\xi'_{i}$ of ${E'_{i}}^{\times}$ satisfying
\[
\xi
=
\xi'_{-1}
\cdot
(\xi'_{0}\circ\Nr_{E/E'_{0}})
\cdots
(\xi'_{d'}\circ\Nr_{E/E'_{d'}}).
\]
Then the index $\bfi'$ is characterized as the unique index satisfying $\sigma^{\frac{e}{2}}\in \Gamma_{E'_{\bfi'+1}}- \Gamma_{E'_{\bfi'}}$.
The depth of the character $\phi'_{\bfi'}=\xi'_{\bfi'}\circ\Nr_{E/E'_{\bfi'}}$ is given by $r'_{\bfi'}=r_{\bfi} (=:r=\frac{2s+1}{e})$.

\begin{proof}[Proof of Proposition \ref{prop:main}: symmetric ramified case]
Our task is to show that 
\[
\chi_{\Tam,\alpha}^{-1}=\chi_{\Kal,\alpha}.
\]
For this, by the definitions of $\chi_{\Kal,\alpha}$ and $\chi_{\Tam,\alpha}$, it suffices to check that 
\[
\chi_{\Tam,\alpha}^{-1}(2a_{\alpha})=\chi_{\Kal,\alpha}(2a_{\alpha}).
\]

We start from rewriting $\chi_{\Kal,\alpha}(2a_{\alpha})$.
By definition, $\chi_{\Kal,\alpha}(2a_{\alpha})$ is equal to the Langlands constant $\lambda_{E/F_{\pm\alpha}}$.
As $E/F_{\pm\alpha}$ is quadratic ramified, we have
\[
\lambda_{E/F_{\pm\alpha}}=\mfn(\psi_{F_{\pm\alpha}})
\]
(see, for example, \cite[Lemma 1.5 (3)]{MR2148193}).
By noting that $K_{l-1}$ is contained in $F_{\pm \alpha}=F[\mu_{E}, \varpi_{E}^{2}]$ and that the extension $F_{\pm\alpha}/K_{l-1}$ is totally ramified of degree $m$, the additive character of $k_{F_{\pm\alpha}}=k_{E}$ induced by $\psi_{F_{\pm\alpha}}$ is equal to the composition of the additive character induced by $\psi_{K_{l-1}}$ with the multiplication map via $m$.
Thus we get
\begin{align*}
\mfn(\psi_{F_{\pm\alpha}})
&=q^{-\frac{f}{2}}\sum_{x\in k_{E}^{\times}}\begin{pmatrix}\frac{x}{k_{E}^{\times}}\end{pmatrix}\psi_{F_{\pm\alpha}}(x)\\
&=q^{-\frac{f}{2}}\sum_{x\in k_{E}^{\times}}\begin{pmatrix}\frac{x}{k_{E}^{\times}}\end{pmatrix}\psi_{K_{l-1}}(mx)\\
&=q^{-\frac{f}{2}}\begin{pmatrix}\frac{m^{-1}}{k_{E}^{\times}}\end{pmatrix}\sum_{x\in k_{E}^{\times}}\begin{pmatrix}\frac{x}{k_{E}^{\times}}\end{pmatrix}\psi_{K_{l-1}}(x)
=\begin{pmatrix}\frac{m}{q^{f}}\end{pmatrix}^{-1}\cdot\mfn(\psi_{K_{l-1}}).
\end{align*}
In summary, we get
\begin{equation}\label{eq:chi-Kal}
\chi_{\Kal,\alpha}(2a_{\alpha})
=\mfn(\psi_{K_{l-1}})\begin{pmatrix}\frac{m}{q^{f}}\end{pmatrix}.
\end{equation}

We next rewrite $\chi_{\Tam,\alpha}^{-1}(\varpi_{E})$.
First, as we already explained in the definition of Tam's $\chi$-data, we have $\chi_{\Tam,\alpha}(\varpi_{E})=t(\mfW_{[g]})$ by the oddness of the numerator of $r$.
Next, we recall from \cite[Section 5.3.4 and Remark 7.7]{MR3509939} the expression of the t-factor $t(\mfW_{[g]})$ in terms of a quadratic form and the Gauss sum.
For this, observe that for any subextension $L_1/L_2$ of $E/K_{l-1}$ the modules defined by
\[
\mfW_{L_1/L_2}:=\bigoplus_{[h]\in \Gamma_{E}\backslash (\Gamma_{L_2}- \Gamma_{L_1})/\Gamma_{E}}\mfW_{[h]} \\
\subset \mfW_{E/K_{l-1}}:=\bigoplus_{[h]\in(\Gamma_{E}\backslash\Gamma_{K_{l-1}}/\Gamma_{E})'}\mfW_{[h]},
\]
as well as $\mfW_{[g]} \subset \mfW_{E/K_{l-1}}$, are naturally vector spaces over $k_{K_{l-1}}=k_{E}$.
Recall that these spaces can be regarded as subspaces of $\mfU=\mfA/\mfP$.
Let $\mathbf{q}_{K_{l-1}}^{(\bfi')}$ be the quadratic form defined by
\begin{align*}
\mathbf{q}_{K_{l-1}}^{(\bfi')} \colon \mfW_{E'_{\bfi'}/E'_{\bfi'+1}} &\to k_{E}; \\
x&\mapsto \overline{\tr_{A_{K_{l-1}}/K_{l-1}}\bigl((x-\varpi_{E}^{-1}x\varpi_{E})\cdot \varpi_{E}^{s}x\varpi_{E}^{-s}\bigr)}.
\end{align*}
where we write $A_{K_{l-1}}:=\End_{K_{l-1}}(E)$ and $\tr_{A_{K_{l-1}}/K_{l-1}}$ denotes the reduced trace of $A_{K_{l-1}}/K_{l-1}$.
Here note that this definition makes sense since the space $\mfW_{E/K_{l-1}}$ can be identified with $(A_{K_{l-1}}\cap\mfA)/(A_{K_{l-1}}\cap\mfP)\subset \mfA/\mfP$.
We also note that we have $\End_{K_{l-1}}(E)\cong \mathrm{Mat}_{2m}(K_{l-1})$ and that $\tr_{A_{K_{l-1}}/K_{l-1}}$ is nothing but the trace as a matrix in $\mathrm{Mat}_{2m}(K_{l-1})$.
This is the quadratic form $\mathbf{q}_{\mathbf{F}}^{(j)}$ in \cite[1739 page, (5.19)]{MR3509939} for $\mathbf{F}=K_{l-1}$ and $j=\bfi'$ 
(as remarked before we start this proof, the depth of $\phi'_{\bfi'}=\xi'_{\bfi'} \circ \Nr_{E/E'_{\bfi'}}$ is equal to $r=\frac{2s+1}{e}$ and thus $\mathbf{h}_{\bfi'}$ in \cite[1738 page]{MR3509939} is equal to $s+1$).
Now as explained in \cite[Remark 7.7]{MR3509939}, the t-factor $t(\mfW_{[g]})$ is given by
\begin{equation} \label{eq:Tam-Rem7.7}
\Bigl(\begin{pmatrix}\frac{\zeta^{(\bfi')}(\varpi_{E})}{k_{E}^{\times}}\end{pmatrix}\mfn(\psi_{K_{l-1}})\Bigr)^{\dim_{k_E}\mfW_{[g]}}
\begin{pmatrix}\frac{\det\bigl(\mathbf{q}_{K_{l}/K_{l-1}}^{(\bfi')}|_{\mfW_{[g]}}\bigr)}{k_{E}^{\times}}\end{pmatrix}.
\end{equation}
Here the meanings of the symbols $\zeta^{(\bfi')}(\varpi_{E})$ and $\det(\mathbf{q}_{K_{l}/K_{l-1}}^{(\bfi')}|_{\mfW_{[g]}})$ are as follows:
\begin{itemize}
\item
We consider an element $\alpha^{(\bfi')}(\varpi_{E})\in {E'_{\bfi'}}^{\times}$ satisfying 
\[
\xi'_{\bfi'}(1+X)=\psi_{E'_{\bfi'}}\bigl(\alpha^{(\bfi')}(\varpi_{E})\cdot X\bigr)
\]
for any $X\in E'_{\bfi',r}$.
Since such an element is unique up to $U_{E'_{\bfi'}}^{1}$, we can find a unique root of unity $\zeta^{(\bfi')}(\varpi_{E})\in\mu_{E}$ satisfying
\[
\alpha^{(\bfi')}(\varpi_{E})
\equiv
\varpi_{E}^{-(2s+1)}\zeta^{(\bfi')}(\varpi_{E})
\mod U_{E}^{1}.
\]
\item
We put $\mfW^{(\bfi')}=\mfW_{K_{l}/K_{l-1}}\cap \mfW_{E'_{\bfi'}/E'_{\bfi'+1}}$.
Then we define a quadratic form $\mathbf{q}_{K_{l}/K_{l-1}}^{(\bfi')}$ to be the restriction $\mathbf{q}_{K_{l-1}}^{(\bfi')}|_{\mfW^{(\bfi')}}$ of $\mathbf{q}_{K_{l-1}}^{(\bfi')}$ to $\mfW^{(\bfi')}$.
The symbol $\det(\mathbf{q}_{K_{l}/K_{l-1}}^{(\bfi')}|_{\mfW_{[g]}})$ denotes the discriminant of the restriction of $\mathbf{q}_{K_{l}/K_{l-1}}^{(\bfi')}$ to $\mfW_{[g]}$.
\end{itemize}
We remark that, strictly speaking, the quadratic form $\mathbf{q}_{K_{l}/K_{l-1}}^{(\bfi')}$ in the last quadratic character in \cite[Remark 7.7]{MR3509939} should be written as $\mathbf{q}_{K_{l}/K_{l-1}}^{(\bfi')}|_{\mfW_{[g]}}$.
Here note that, by definition, $\mathbf{q}_{K_{l}/K_{l-1}}^{(\bfi')}|_{\mfW_{[g]}}$ is nothing but $\mathbf{q}_{K_{l-1}}^{(\bfi')}|_{\mfW_{[g]}}$.
Now let us compute the discriminant of $\mathbf{q}_{K_{l-1}}^{(\bfi')}|_{\mfW_{[g]}}$.
Since we have $\dim_{k_E}\mfW_{[g]}=1$ (recall Proposition \ref{prop:U-decomp} and the definition of $\mfW_{[g]}$), this is easily done by an explicit computation as follows (cf.\ the proof of \cite[Proposition 8.3]{MR2148193}).
We start by extending $\mathbf{q}_{K_{l-1}}^{(\bfi')}$ from $\mfW_{E'_{\bfi'}/E'_{\bfi'+1}}$ to $\mfW_{E/K_{l-1}}$ by the same formula:
\begin{align*}
\mfW_{E/K_{l-1}} &\to k_{E}; \\
x&\mapsto \overline{\tr_{A_{K_{l-1}}/K_{l-1}}\bigl((x-\varpi_{E}^{-1}x\varpi_{E}) \cdot \varpi_{E}^{s}x\varpi_{E}^{-s}\bigr)},
\end{align*}
for which we still write $\mathbf{q}_{K_{l-1}}^{(\bfi')}$.
If we take a basis $\{\varpi_{E}^{2m-1}, \varpi_{E}^{2m-2}, \dots, \varpi_{E}, 1\}$ of $E$ over $K_{l-1}=F[\mu_{E}, \varpi_{E}^{2m}]$
to identify $A_{K_{l-1}}=\End_{K_{l-1}}(E)$ with $\mathrm{Mat}_{2m}(K_{l-1})$,
then $\mfA_{K_{l-1}}:=\mfA \cap A_{K_{l-1}}$ (resp.\ $\mfP_{K_{l-1}}:=\mfP \cap A_{K_{l-1}}$) is identified with the subset of matrices with entries in $\mcO_{K_{l-1}}$ that are upper triangular (resp.\ strictly upper triangular) when reduced modulo $\mfp_{K_{l-1}}$.
We have an isomorphism of $k_{E}$-algebras:
\begin{align*}
&\mfW_{E/K_{l-1}}\cong\mfA_{K_{l-1}}/\mfP_{K_{l-1}} \cong k_{E}^{2m}; \\
&(x_{ij})_{1\leq i, j \leq 2m}\mapsto (\overline{x_{11}}, \overline{x_{22}}, \dots, \overline{x_{2m, 2m}}).
\end{align*}
The action of $\varpi_{E}\in \varpi$ on the right-hand side induced by this isomorphism is given by the translation 
\[
(a_1, a_2, \dots, a_{2m}) \mapsto (a_2, \dots, a_{2m}, a_1).
\]
Thus the submodule $\mfW_{[g]}\subset \mfW_{E/K_{l-1}}$ is mapped to
\[
(k_{E}^{2m})_{[g]}:=\{ v(a):=(a, -a, \dots, a, -a)\in k_{E}^{2m} \mid a\in k_{E}\}
\]
as it can be characterized as the subspace where $\varpi_{E}$ acts as multiplication via $-1$.
Moreover, the trace map $\tr_{A_{K_{l-1}}/K_{l-1}}$ on $A_{K_{l-1}}$ induces the summation map
\[
\nu\colon k_{E}^{2m}\to k_{E}; \quad (a_1, a_2, \dots, a_{2m})\mapsto \sum_{i=1}^{2m}a_i.
\]
Therefore the quadratic form $\mathbf{q}_{K_{l-1}}^{(\bfi')}$ on $\mfW_{[g]}$ induces the following quadratic form on $(k_{E}^{2m})_{[g]}$:
\[
(k_{E}^{2m})_{[g]}\to k_{E}; \quad v(a)\mapsto \nu \bigl(2v(a)\cdot (-1)^sv(a)\bigr)
=(-1)^s4ma^2,
\]
whose discriminant is clearly $(-1)^s4m\equiv (-1)^sm \mod{(k_{E}^{\times})^2}$.
Now, by substituting $\dim_{k_E} \mfW_{[g]}=1$ and the discriminant just computed into \eqref{eq:Tam-Rem7.7},
we get
\begin{equation}\label{eq:Tam-varpi}
\chi_{\Tam,\alpha}(\varpi_{E})=
\begin{pmatrix}\frac{\zeta^{(\bfi')}(\varpi_{E})}{k_{E}^{\times}}\end{pmatrix}\mfn(\psi_{K_{l-1}})
\begin{pmatrix}\frac{(-1)^{s}m}{q^{f}}\end{pmatrix}.
\end{equation}

Before we start our comparison of $\chi_{\Kal,\alpha}$ and $\chi_{\Tam,\alpha}$, we investigate the relation between $a_{\alpha}$ and $\zeta^{(\bfi')}(\varpi_{E})$.
Recall that $a_{\alpha}\in E_{-r}$ is a lift of the unique element $\bar{a}_{\alpha}\in E_{-r}/E_{-r+}$ satisfying the equality
\[
\phi_{\bfi}\bigl(\Nr_{\bfS(E)/\bfS(F)}\circ\alpha^{\vee}(X+1)\bigr)
=
\psi_{E}(\bar{a}_{\alpha}\cdot X)
\]
for any $X\in E_{r}/E_{r+}$.
Since we have $\phi_{\bfi}|_{\bfS(F)_{r}}=\phi'_{\bfi'}|_{\bfS(F)_{r}}$ as explained in (\ref{eq:i=i'}) and $\phi'_{\bfi'}|_{\bfS(F)}=\xi'_{\bfi'}\circ\Nr_{E/E'_{\bfi'}}$, this is equivalent to
\[
\xi'_{\bfi'}\circ\Nr_{E/E'_{\bfi'}}\bigl(\Nr_{\bfS(E)/\bfS(F)}\circ\alpha^{\vee}(X+1)\bigr)
=
\psi_{E}(\bar{a}_{\alpha}\cdot X).
\]
Here we note that the image of $X+1$ under the map $\Nr_{\bfS(E)/\bfS(F)}\circ\alpha^{\vee}$ from $E^{\times}$ to $\bfS(F)\cong E^{\times}$ is given by $(X+1)(\sigma^{\frac{e}{2}}(X)+1)^{-1}$.
Indeed, if we consider an isomorphism from $\bfS(\overline{F})$ to $\prod_{i=1}^{n}\overline{F}^{\times}$ as in Section \ref{subsec:Tam}, the coroot $\alpha^{\vee}$ is described as
\[
E^{\times}\rightarrow \bfS(E);\quad x\mapsto (x,1,\ldots,1,x^{-1},1,\ldots,1).
\]
Here the component having $x^{-1}$ corresponds to $\sigma^{\frac{e}{2}}\in\{\Gamma_{F}/\Gamma_{E}\}$.
Note that $\alpha^{\vee}$ maps $E^{\times}$ into $\bfS(E)$ as $E$ is the splitting field $F_{\alpha}$ of $\alpha$.
On the other hand, for each $z\in\bfS(E)$, its norm $\Nr_{\bfS(E)/\bfS(F)}(z)$ is defined as the product of $g(z)$ over $g\in\Gamma_{F}/\Gamma_{E}$.
Thus, by considering a description of the Galois action on $\prod_{i=1}^{n}\overline{F}^{\times}$, we can easily check that $\Nr_{\bfS(E)/\bfS(F)}\circ\alpha^{\vee}(x)$ is given by $x\cdot\sigma^{\frac{e}{2}}(x)^{-1}$ for each $x\in E^{\times}$.

As we have
\[
(X+1)\bigl(\sigma^{\frac{e}{2}}(X)+1\bigr)^{-1}
\equiv
X-\sigma^{\frac{e}{2}}(X)+1
\]
in $E^{\times}_{r}/E^{\times}_{r+}$, we have 
\[
\Nr_{E/E'_{\bfi'}}\bigl((X+1)(\sigma^{\frac{e}{2}}(X)+1)^{-1}\bigr)
\equiv
\Tr_{E/E'_{\bfi'}}\bigl(X-\sigma^{\frac{e}{2}}(X)\bigr)+1
\]
in ${E'}_{\bfi',r}^{\times}/{E'}_{\bfi',r+}^{\times}$.
By recalling that $r$ is given by $\frac{2s+1}{e}$ and $\sigma^{\frac{e}{2}}$ acts on $E$ via 
\[
\sigma^{\frac{e}{2}}|_{\mu_{E}}\equiv \mathrm{id}
\quad
\text{and}
\quad
\sigma^{\frac{e}{2}}(\varpi_{E})=-\varpi_{E},
\]
we get 
\[
\Tr_{E/E'_{\bfi'}}\bigl(X-\sigma^{\frac{e}{2}}(X)\bigr)+1
\equiv
2\Tr_{E/E'_{\bfi'}}(X)+1
\]
in ${E'}_{\bfi',r}^{\times}/{E'}_{\bfi',r+}^{\times}$.
Therefore, for any $X\in E_{r}/E_{r+}$, we have
\[
\xi'_{\bfi'}\bigl(2\Tr_{E/E'_{\bfi'}}(X)+1\bigr)
=
\psi_{E}(\bar{a}_{\alpha}\cdot X).
\]
On the other hand, by the definitions of $\alpha^{(\bfi')}(\varpi_{E})$ and $\zeta^{(\bfi')}(\varpi_{E})$, the left-hand side equals 
\begin{align*}
\psi_{E'_{\bfi'}}\bigl(\alpha^{(\bfi')}(\varpi_{E})\cdot 2\Tr_{E/E'_{\bfi'}}(X)\bigr)
&=
\psi_{E}\bigl(\alpha^{(\bfi')}(\varpi_{E})\cdot 2X\bigr)\\
&=
\psi_{E}\bigl(\varpi_{E}^{-(2s+1)}\zeta^{(\bfi')}(\varpi_{E})\cdot 2X\bigr).
\end{align*}
In summary, we have the equality
\[
\psi_{E}(\bar{a}_{\alpha}\cdot X)
=
\psi_{E}\bigl(\varpi_{E}^{-(2s+1)}\zeta^{(\bfi')}(\varpi_{E})\cdot 2X\bigr)
\]
for any $X\in E_{r}/E_{r+}$.
Therefore we have
\begin{equation}\label{eq:a-zeta}
\bar{a}_{\alpha}
\equiv
2\varpi_{E}^{-(2s+1)}\zeta^{(\bfi')}(\varpi_{E})
\quad\text{in}\quad E_{-r}/E_{-r+}.
\end{equation}

Now let us complete the proof.
By recalling that $\chi_{\Tam,\alpha}$ is given by the nontrivial quadratic character on $\mu_{E}$, the equality (\ref{eq:a-zeta}) tells us
\begin{align*}
\chi_{\Tam,\alpha}^{-1}(2a_{\alpha})
&=
\chi_{\Tam,\alpha}^{-1}\bigl(4\varpi_{E}^{-(2s+1)}\zeta^{(\bfi')}(\varpi_{E})\bigr)\\
&=
\chi_{\Tam,\alpha}(\varpi_{E})^{2s+1}\begin{pmatrix}\frac{\zeta^{(\bfi')}(\varpi_{E})}{k_{E}^{\times}}\end{pmatrix}^{-1}.
\end{align*}
Then, by the equality (\ref{eq:Tam-varpi}) and noting that 
\[
\chi_{\Tam,\alpha}(\varpi_{E})^{2}
=
\begin{pmatrix}\frac{-1}{q^{f}}\end{pmatrix}
\]
(this also can be deduced from (\ref{eq:Tam-varpi})), we get
\begin{align*}
\chi_{\Tam,\alpha}^{-1}(2a_{\alpha})
&=
\begin{pmatrix}\frac{-1}{q^{f}}\end{pmatrix}^{s}
\begin{pmatrix}\frac{\zeta^{(\bfi')}(\varpi_{E})}{k_{E}^{\times}}\end{pmatrix}\mfn(\psi_{K_{l-1}})
\begin{pmatrix}\frac{(-1)^{s}m}{q^{f}}\end{pmatrix}
\begin{pmatrix}\frac{\zeta^{(\bfi')}(\varpi_{E})}{k_{E}^{\times}}\end{pmatrix}^{-1}\\
&=
\mfn(\psi_{K_{l-1}})
\begin{pmatrix}\frac{m}{q^{f}}\end{pmatrix}.
\end{align*}
Thus we get $\chi_{\Tam,\alpha}^{-1}(2a_{\alpha})=\chi_{\Kal,\alpha}(2a_{\alpha})$ by the equality (\ref{eq:chi-Kal}) and this completes the proof.
\end{proof}

\appendix
\section{Comparison of the two constructions of supercuspidal representations}\label{sec:BH-Kal}
In this appendix we show that if a tame elliptic regular pair $(\bfS, \xi)$ of $\G=\GL_{n}$ corresponds to an $F$-admissible pair $(E, \xi)$
then the regular supercuspidal representation $\pi_{(\bfS,\xi)}^{\KY}$
is isomorphic to the essentially tame supercuspidal representation $\pi_{(E,\xi)}^{\BH}$.
This should be well-known to experts, but we record some details here.

In both the constructions, one uses $\xi$ to define an open compact-modulo-center subgroup of $\G(F)$ and an irreducible representation,
and then takes the compact induction to obtain the irreducible supercuspidal representation associated to the pair.
Apart from the difference in the languages used---one based on orders and radicals defined by means of lattices, the other based on the Bruhat--Tits theory and the Moy--Prasad filtrations---each step of the two constructions proceeds
in almost the same way, except at one place.
Thus the main point of this appendix is the comparison of the ``only debatable step" (\cite[695 page, Remark 1]{MR2138141}) in the construction of $\pi_{(E,\xi)}^{\BH}$ with the corresponding step in that of $\pi_{(\bfS,\xi)}^{\KY}$ using the Heisenberg--Weil representation (see Remark \ref{rem:debatable}).

We remark that we found \cite[Section 5.2]{MR3509939}, \cite{Mayeux:2017aa}, and \cite{Mayeux:2020aa} particularly helpful in preparing this appendix.

Let $(\bfS,\xi)$ be a tame elliptic regular pair of $\G$ which corresponds to an $F$-admissible pair $(E, \xi)$.

We first review the construction of $\pi_{(E,\xi)}^{\BH}$ (\cite[Section 2]
{MR2138141} and \cite[Section 4]{MR2679700}), largely following \cite[Section 5.2]{MR3509939}.

\begin{description}
	\item[Definition of open subgroups] 
	We fix an $F$-basis of $E$ to regard $\mathrm{Mat}_n(F)\cong\End_{F}(E)$ 
	(in particular $E\subset \mathrm{Mat}_n(F)$).
	Recall from Section \ref{subsec:Heisen} that 
	the pair $(E, \xi)$ gives rise to sequences of subfields and integers:
	\[
	F=E_d\subsetneq E_{d-1}\subsetneq \cdots \subsetneq E_0 \subset E_{-1}=E,
	\quad t_d \geq t_{d-1}> \cdots >t_0>t_{-1}=0,
	\] 
	orders and radicals:
	\[
	\End_{E_i}(E)\supset \mfA_{i} \supset \mfP_{i},
	\]
	and various subgroups of $\G(F)\cong\Aut_F(E)$:
	\begin{align*}
	U_{\mfA_{i}}^k&:=1+\mfP_{i}^{k} \quad (k\in \Z_{>0}), \\
	H^1&:=U_{\mfA_0}^1U_{\mfA_1}^{h_0}\cdots U_{\mfA_{d-1}}^{h_{d-2}}U_{\mfA_{d}}^{h_{d-1}}, \\
	J^1&:=U_{\mfA_0}^1U_{\mfA_1}^{j_0}\cdots U_{\mfA_{d-1}}^{j_{d-2}}U_{\mfA_{d}}^{j_{d-1}}
	\end{align*}
	(see Section \ref{subsec:Heisen} for the definition of $h_i$ and $j_i$ using $t_i$).
	We need two more subgroups:
	\begin{align*}
	J^0:=\mfA_0^{\times}J^1, \quad \mathbf{J}:=E^{\times}J^0=E_{0}^{\times}J^0.
	\end{align*}
	\item[Definition of a character $\theta$ on $H^1$]
	From $\xi$ (or more precisely its restriction to $U_E^1$)
	we construct a character $\theta$ on $H^1$.
	Bushnell--Henniart's construction
	is somewhat abstract
	but here we follow
	\cite[1731 page, (i)]{MR3509939} and express $\theta$ in more down-to-earth terms.
	
	Recall from Section \ref{subsec:Howe} that we have a factorization of $\xi$
	in terms of characters $\xi_i$ on $E_i^{\times}$:
	\[
	\xi=\xi_{-1}\cdot (\xi_0 \circ \Nr_{E/E_0}) \cdots (\xi_d \circ \Nr_{E/E_d}).
	\]
	Among other properties, these characters satisfy the following:
	\begin{itemize}
		\item For $i=-1$, the character $\xi_{-1}$ is tamely ramified.
		\item For $0\leq i\leq d-1$, the $E$-level of $\xi_i \circ \Nr_{E/E_i}$ is $t_i$.
	\end{itemize}
	We use $\xi_i$ for $0\leq i\leq d$ to define $\theta$.
	Let $c_i\in \mfp_E^{-t_i}\cap E_i$ be an element such that 
	\[
	\xi_i(1+x)=\psi_F(\Tr_{E_{i}/F}(c_ix))
	\]
	for $x\in \mfp_E^{h_i}\cap E_i$.
	Note that the coset $c_i+(\mfp_E^{1-h_i}\cap E_i)$ is well-defined.
	Putting $A_i:=\End_{E_i}(E)\subset \End_{F}(E)$,
	we define a character $\psi_{c_i}$ on $U_{\mfA_{d}}^{h_i}$ by
	\[
	\psi_{c_i}(1+x):=\psi_{F}(\tr_{A_d/F}(c_ix))
	\]
	for $x\in \mfP_{d}^{h_i}$.
	With $\psi_{c_i}$, we define characters $\theta_i$ on $U_{\mfA_{i}}^{h_{i-1}}U_{\mfA_{i+1}}^{h_i}\cdots U_{\mfA_{d}}^{h_{d-1}}$ for $0\leq i\leq d$ inductively from $i=d$ as follows (here we put $h_{-1}:=1$).
	We define $\theta_d$ by
	\[
	\theta_d:=\xi_d \circ \det |_{U_{\mfA_{d}}^{h_{d-1}}}.
	\]
	Inductively, $\theta_i$ is defined using $\theta_{i+1}$ by
	\begin{align*}
	\theta_i&=\psi_{c_i}\theta_{i+1} \quad \text{on $U_{\mfA_{i+1}}^{h_i}\cdots U_{\mfA_{d}}^{h_{d-1}}$}, \\
	\theta_i&=(\xi_i \circ {\textstyle\det_{i}})\cdot (\xi_{i+1} \circ {\textstyle\det_{i+1}}) \cdots (\xi_d \circ {\textstyle\det_{d}}) \quad \text{on $U_{\mfA_{i}}^{h_{i-1}}$},
	\end{align*}
	where, letting $n_i:=[E:E_i]$, we write $\det_i\colon A_i^{\times} \cong \GL_{n_{i}}(E_i)\to E_{i}^{\times}$ for 
	the determinant map.
	Finally, we put $\theta:=\theta_0$.
	
	It can be checked that $\theta$ does not depend on the choice of a factorization of $\xi$ and only depends on $\xi|_{U_E^1}$.
	
	\item[Definition of an irreducible representation $\eta$ of $J^1$]
	There exists a unique representation $\eta$ of $J^1$ containing $\theta$ (see \cite[1732 page, (ii)]{MR3509939}).
	We only remark that the proof is based on the fact that $J^1/\Ker \theta$ is a Heisenberg $p$-group and $H^1/\Ker \theta$ is its center
	(see, for example,  \cite[Section 2.3]{MR2138141}). 
	\item[Definition of an irreducible representation $\Lambda_{\mathrm{w}}$ of $\mathbf{J}$]
	Recall that we fixed a uniformizer $\varpi_{F}$ of $F$.
	We extend $\eta$ to an irreducible representation $\Lambda_{\mathrm{w}}$ of
	$\mathbf{J}$ by the following conditions (the existence and the uniqueness of such an extension are proved in \cite[Section 2.3, Lemmas 1 and 2]{MR2138141}):
	\begin{description}
		\item[(Ext1)] \label{cond:J1} The restriction $\Lambda_{\mathrm{w}} |_{J^1}$ is isomorphic to $\eta$.
		\item[(Ext2)] The restriction $\Lambda_{\mathrm{w}}|_{J^{0}}$ is intertwined by $A_0^{\times}$, i.e., for any $a\in A_0^{\times}$, we have 
		\[\Hom_{J^{0}\cap J^{0, a}}\left( \Lambda_{\mathrm{w}}|_{J^{0}\cap J^{0, a}}, \Lambda_{\mathrm{w}}^a|_{J^{0}\cap J^{0, a}}\right) \neq 0,\]
		where $J^{0, a}=a^{-1}J^0a$ is the conjugate and $\Lambda_{\mathrm{w}}^a$
		is the representation of $J^{0, a}$ obtained as the conjugation of $\Lambda_{\mathrm{w}}$.
		\item[(Ext3)] \label{cond:varpi} We have $\varpi_{F} \in \Ker \Lambda_{\mathrm{w}}$. In particular, $\det \Lambda_{\mathrm{w}}$ has finite order.
		\item[(Ext4)] \label{cond:detorder} The determinant character $\det \Lambda_{\mathrm{w}}$ has $p$-power order.
	\end{description}
	Note that $\Lambda_{\mathrm{w}}$ only depends on $\xi|_{U_E^1}$ and the choice of $\varpi_{F}$.	
	\begin{rem}\label{rem:Ext2}
		In fact, (Ext2) is not necessary for the uniqueness.
		Indeed, suppose that $\Lambda_1$ and $\Lambda_2$ both satisfy
		the above conditions except for (Ext2).
		By (Ext1), the restrictions $\Lambda_1|_{J^0}$ and $\Lambda_2|_{J^0}$ only differ by a twist with a character $\chi$ of $J^0$ trivial on $J^1$.
		As $J^0/J^1$ is isomorphic to $\GL_{n_0}(k_{E_0})$,
		the character $\chi$ factors through the determinant map
		(note that $p$ is odd and thus the abelianization of $\GL_{n_0}(k_{E_0})$ is $k_{E_0}^{\times}$),
		and in particular has order prime to $p$.
		As $\dim \eta$ is a power of $p$,
		(Ext4) implies that $\chi$ is trivial and hence $\Lambda_1|_{J^0}$ and $\Lambda_2|_{J^0}$ are isomorphic.
		Similarly, we can further show that $\Lambda_1$ and $\Lambda_2$ are isomorphic as representations of $\J$
		by using (Ext3) and noting that $\J/\langle \varpi_{F}, J^0\rangle$ is a cyclic group of order $e(E_0/F)$, which is prime to $p$.
				
		Later we will prove that an irreducible representation of $\J$ is 
		isomorphic to $\Lambda_{\mathrm{w}}$ by checking that it satisfies (Ext1), (Ext3), and (Ext4).
		\end{rem}
	\begin{rem}\label{rem:debatable}
		This part is described as the ``only debatable step in the construction'' in \cite[695 page, Remark 1]{MR2138141}.
		The crux of this appendix lies in showing that this definition of $\Lambda_{\mathrm{w}}$ is consistent with the construction of $\pi_{(\bfS,\xi)}^{\KY}$ in a suitable sense (Proposition \ref{prop:Lambda_0}).
	\end{rem}
	\item[Definition of an irreducible representation $\Lambda_{\mathrm{t}}$ of $\mathbf{J}$ using $\xi_{\mathrm{t}}$]	
	First we take a factorization $\xi=\xi_{\mathrm{t}}\xi_{\mathrm{w}}$
	of the character $\xi$ characterized by the following conditions:\begin{itemize}
		\item The character $\xi_{\mathrm{t}}$ is tame. In other words, $\xi_{\mathrm{w}}|_{U_E^1}=\xi|_{U_E^1}$.
		\item We have $\varpi_{F} \in \Ker \xi_{\mathrm{w}}$. In particular, $\xi_{\mathrm{w}}$ has finite order.
		\item The character $\xi_{\mathrm{w}}$ has $p$-power order. In particular, $\mu_{E} \subset \Ker \xi_{\mathrm{w}}$.
	\end{itemize}
	It is easily checked that such a factorization $\xi=\xi_{\mathrm{t}}\xi_{\mathrm{w}}$ uniquely exists.
	We stress that the uniformizer $\varpi_{F}$ here is the same as the one in (Ext3).
	
	Then we construct an irreducible representation $\Lambda_{\mathrm{t}}$ 
	of $\mathbf{J}$ associated to $\xi_{\mathrm{t}}$ as follows.
	As $\xi_{\mathrm{t}}$ is tame, $\xi_{\mathrm{t}}|_{\mcO_E^{\times}}$ is
	the inflation of a character $\overline{\xi}_{\mathrm{t}}$ on $k_E^{\times}$.
	By the Green parametrization \cite{MR72878}, $\overline{\xi}_{\mathrm{t}}$ yields an irreducible cuspidal representation $\overline{\lambda}$ of $\GL_{n_0}(k_{E_0})\cong J^0/J^1$.
	Let $\lambda$ be the inflation of $\overline{\lambda}$ to $J^0$.
	We define an extension $\Lambda_{\mathrm{t}}$ of $\lambda$ to $\mathbf{J}=E_{0}^{\times}J^0$ by requiring that $\Lambda_{\mathrm{t}}|_{E_0^{\times}}$ is $\xi_{\mathrm{t}}|_{E_0^{\times}}$-isotypic.
	
	\item[Definition of $\pi_{(E,\xi)}^{\BH}$]
	We put $\Lambda_{\xi}:=\Lambda_{\mathrm{t}}\otimes \Lambda_{\mathrm{w}}$
	and finally define 
	$\pi_{(E,\xi)}^{\BH}:=\cInd_{\mathbf{J}}^{\G(F)}\Lambda_{\xi}$.
	We can easily check that $\Lambda_{\xi}$ is independent of the choice of $\varpi_{F} \in F$
	and that $\pi_{(E,\xi)}^{\BH}$ is independent of the choice of an $F$-basis of $E$.
	Thus in the following we assume that 
	the embedding $E\hookrightarrow \mathrm{Mat}_n(F)$ induced by the $F$-basis restricts to the inclusion 
	$E^{\times}\cong\bfS(F)\hookrightarrow \G(F)=\GL_n(F)$.
\end{description}

On the other hand,
Kaletha's construction starts with taking a regular Yu-datum
sent to $(\bfS, \xi)$ by the map in Proposition \ref{prop:reparametrization} (\cite[Proposition 3.7.8]{MR4013740})
and 
then Yu's construction \cite{MR1824988} yields an open subgroup $K$ of $\G(F)$ 
and an irreducible representation $\kappa$ of $K$ which give rise to a supercuspidal representation $\pi_{(\bfS,\xi)}^{\KY}:=\cInd_{K}^{\G(F)}\kappa$. 
As in the proof of the surjectivity of \cite[Proposition 3.7.8]{MR4013740}, 
we use a Howe factorization of $(\bfS, \xi)$ as in Section \ref{subsec:Howe} to construct a regular Yu-datum as above.
We write $\phi_{i} \colon \G^{i}(F)\to \C^{\times}$ ($i=-1, \dots, d$)
and $\Psi=(\G^0\subsetneq \G^1 \subsetneq \cdots \subsetneq \G^d, \pi_{-1}, (\phi_{0}, \dots, \phi_{d}))$ for the Howe factorization and the associated regular Yu-datum.
In particular, we have the following:
\begin{itemize}
	\item each $\G^i$ is as in Section \ref{subsec:Howe},
	\item $\phi_i=\xi_i \circ \textstyle{\det_i}$ for $i=-1, \dots, d$ with $\xi_i$ as in the definition of $\theta$,
	\item $\xi =\prod_{i=-1}^{d}\phi_{i}|_{\bfS(F)}$, or equivalently 
	$\xi =\prod_{i=-1}^{d}(\xi_i \circ \Nr_{E/E_{i}})$, and
	\item $\pi_{-1}:= \pi_{(\bfS, \phi_{-1})}^{\G^{0}}$, where we write $\pi_{(\bfS, \phi_{-1})}^{\G^{0}}$ for the irreducible depth zero supercuspidal representation of $\G^{0}(F)$ associated to $(\bfS, \phi_{-1})$ (denoted by $\pi_{(\bfS, \phi_{-1})}$ in \cite[Section 3.4]{MR4013740}).
\end{itemize}
According to \cite[Section 3.4]{MR2431732}, 
we may express $\kappa$ as a tensor product $\bigotimes_{i=-1}^{d} \kappa_{i}$ of irreducible representations, where each $\kappa_{i}$ only depends on $\phi_{i}$ for $0\leq i\leq d$ and $\kappa_{-1}$ only depends on $\pi_{-1}$.
As this expression is convenient for our argument, 
below we often cite \cite{MR2431732}.

In the following, we shall see that $K=\mathbf{J}$ holds 
and moreover show that if we take 
a suitable regular Yu-datum $\underline{\Psi}$ which is $\G$-equivalent to $\Psi$,
then we have $\kappa_{-1} \cong \Lambda_{\mathrm{t}}$ 
and $\bigotimes_{i=0}^{d}\kappa_{i} \cong \Lambda_{\mathrm{w}}$, 
so that
$\pi_{(\bfS,\xi)}^{\KY}=\cInd_{K}^{\G(F)}\kappa\cong\cInd_{\mathbf{J}}^{\G(F)}\Lambda_{\xi}=\pi_{(E,\xi)}^{\BH}$.

In \cite[591 page]{MR1824988} the open subgroup $K\subset \G(F)$ is defined 
in terms of the Moy--Prasad filtration as:
\[
K:=K^d:=\G^0(F)_{[x]}\G^1(F)_{x,s_0}\cdots \G^d(F)_{x,s_{d-1}},
\]
where $x$ is a point of the Bruhat--Tits building of $\G^0$ associated to $\bfS \subset \G$ (see Section \ref{subsec:Heisen}) and $s_{i}:=\frac{r_{i}}{2}=\frac{\depth(\phi_{i})}{2}$ for $i=0,\ldots,d-1$.

Recalling that we arranged the embedding $E\hookrightarrow \mathrm{Mat}_n(F)$ to be compatible with the inclusion $\bfS \subset \G$, we can check $K=\mathbf{J}$ by using the work \cite{MR1888474} of Broussous--Lemaire comparing the lattice filtration and the Moy--Prasad filtration as in the proof of Proposition \ref{prop:Heisenberg}.

Similarly we find that the compact open subgroup $K_+^d$:
\[
K_{+}^d:=\G^0(F)_{x,0+}\G^1(F)_{x,s_{0}+}\cdots \G^d(F)_{x,s_{d-1}+}
\]
is nothing but $H^1$.

The conditions required of the factorization $\xi=\xi_{\mathrm{t}}\xi_{\mathrm{w}}$ are not exactly consistent 
with those imposed on a Howe factorization in Definition \ref{defn:Howe} (\cite[Definition 3.6.2]{MR4013740}),
so we slightly modify the latter factorization in the following proposition
to eventually obtain a $\G$-equivalent regular Yu-datum.
\begin{prop} \label{prop:choiceofdatum}
	There exists a regular Yu-datum
	$\underline{\Psi}=(\G^0\subsetneq \G^1 \subsetneq \cdots \subsetneq \G^d, \underline{\pi}_{-1}, (\underline{\phi}_{0}, \dots, \underline{\phi}_{d}))
	$
	which is $\G$-equivalent to $\Psi$ and satisfies the following properties:
	\begin{itemize}
		\item[(i)] $\underline{\pi}_{-1}\cong \pi_{(\bfS, \xi_{\mathrm{t}})}^{\G^{0}}$,
		\item[(ii)] $\prod_{i=0}^{d}\underline{\phi}_{i}|_{\bfS(F)}=\xi_{\mathrm{w}}$,
		\item[(iii)] $\underline{\phi}_{i}|_{\G^i(F)_{x, 0+}}=\phi_{i}|_{\G^i(F)_{x, 0+}}$ for $i=0, \dots, d$, and
		\item[(iv)] $\underline{\phi}_{i}$ has finite $p$-power order for $i=0, \dots, d$.
	\end{itemize}
\end{prop}
\begin{proof}
	Let us define $\underline{\pi}_{-1}$ and $\underline{\phi}_{i}$ for $i=0, \dots, d$, and thus $\underline{\Psi}=(\G^0\subsetneq \G^1 \subsetneq \cdots \subsetneq \G^d, \underline{\pi}_{-1}, (\underline{\phi}_{0}, \dots, \underline{\phi}_{d}))$.
	First we put 
	\[
	\underline{\phi}_{-1}:=\xi_{\mathrm{t}}, \quad \underline{\pi}_{-1}:=\pi_{(\bfS, \underline{\phi}_{-1})}^{\G^{0}}.
	\]
	Next let us define a character $\underline{\phi}_{i}$ for each $i=0, \dots, d$ using $\phi_{i}$.
	For this we need some auxiliary characters.
	Recall that the characters $\xi_i$ on $E_i^{\times}$ appearing in the definition of $\theta$
	satisfy $\phi_i=\xi_i \circ \det_i$ 
	and $\xi =\prod_{i=-1}^{d}(\xi_i \circ \Nr_{E/E_{i}})$.
		
	We define characters $\zeta_i \colon E^{\times}\to \C^{\times}$ for $i=0, \dots, d$ by the following conditions:
	\begin{itemize}
		\item We have $\zeta_i |_{U_E^1}=(\xi_i \circ \Nr_{E/E_i})|_{U_E^1}$.
		\item We have $\varpi_{F} \in \Ker \zeta_i$.
		\item The character $\zeta_i$ has finite $p$-power order.
	\end{itemize}
	Again these conditions determine the characters $\zeta_i$ uniquely.
	
	The first and third conditions imply that each $\zeta_i$ factors through $\Nr_{E/E_i}$.
	We take $\underline{\xi}_{i}$ to be a character on $E_i^{\times}$ such that $\zeta_i=\underline{\xi}_{i} \circ \Nr_{E/E_i}$ for each $i=0, \dots, d$.
	We assume $\underline{\xi}_{i}$ to have $p$-power order.
	(Note that, unlike before, there are in general several possibilities of such $\underline{\xi}_{i}$ because the index of $\Nr_{E/E_i}(E^{\times})$ in $E_i^{\times}$ may be divisible by $p$.)
	Also, if $\phi_d=\mathbbm{1}$ and thus $\zeta_d=\mathbbm{1}$, then we arrange $\underline{\xi}_{d}=\mathbbm{1}$.
	
	Finally we put 
	\[
	\underline{\phi}_{i}:=\underline{\xi}_{i} \circ \textstyle\det_{i}
	\]
	for $i=0, \dots, d$.
	In particular, if $\phi_d=\mathbbm{1}$, then $\underline{\phi}_{d}=\mathbbm{1}$.
	
	Before proving that $\underline{\Psi}$ is indeed a regular Yu-datum
	$\G$-equivalent to $\Psi$, let us check the required properties (i)-(iv).
	The property (i) is immediate from the definition of $\underline{\pi}_{-1}$.
	For (ii) it is easy to check that the character
	\[
	\prod_{i=0}^d\underline{\phi}_{i}|_{\bfS(F)}=\prod_{i=0}^d(\underline{\xi}_{i} \circ \Nr_{E/E_{i}})=\prod_{i=0}^d\zeta_i
	\]
	satisfies the characterizing conditions of $\xi_{\mathrm{w}}$
	by using those of $\zeta_i$.
	To see (iii), we note that we have $(\underline{\xi}_{i} \circ \Nr_{E/E_i})|_{U_E^1}=\zeta_i |_{U_E^1}=(\xi_i \circ \Nr_{E/E_i})|_{U_E^1}$ and hence
	$\underline{\xi}_{i}|_{U_{E_i}^1}=\xi_i|_{U_{E_i}^1}$.
	Thus the equalities $\underline{\phi}_i=\underline{\xi}_i \circ \det_i$, $\phi_i=\xi_i \circ \det_i$ and the inclusion $\det_i(\G^i(F)_{x, 0+})\subset U_{E_i}^{1}$ imply (iii).
	The property (iv) holds since $\underline{\xi}_{i}$ has finite $p$-power order for each $i=0, \dots, d$.
	
	Finally let us prove that 
	$\underline{\Psi}$
	is a regular Yu-datum and is a refactorization of
	$\Psi$
	in the sense of \cite[Definition 4.19]{MR2431732} and reviewed in \cite[1106 page]{MR4013740}.
	Then we are done because by the definition \cite[Definition 6.3]{MR2431732} of $\G$-equivalence, if two Yu-data are refactorizations of each other, then they are in particular $\G$-equivalent.
	We check the conditions F0-F2 in \cite[1106 page]{MR4013740}:
	\begin{description}
		\item[F0] If $\phi_d=\mathbbm{1}$, then $\underline{\phi}_{d}=\mathbbm{1}$ as we arranged above.
		\item[F1] For $i=0, \dots, d$, we put
		\[
		\chi_i \colon \G^i(F) \to \C^{\times}, \quad \chi_i(g):=\prod_{j=i}^{d}\phi_j(g)\underline{\phi}_j(g)^{-1}.
		\]
		As $\xi_j$ and $\underline{\xi}_j$ agree on $U_{E_j}^1$ by (the proof of) (iii), the depth of $\chi_i$ is $0$ (thus at most $r_{i-1}$) for all $i$.
		\item[F2] As $\underline{\phi}_{-1}=\xi_{\mathrm{t}}$ by definition and $\prod_{i=0}^d\underline{\phi}_{i}|_{\bfS(F)}=\xi_{\mathrm{w}}$ by (ii),
		we have $\prod_{i=-1}^{d} \underline{\phi}_{i}|_{\bfS(F)}=\xi=\prod_{i=-1}^{d} \phi_{i}|_{\bfS(F)}$.
		Therefore, we have
		\[
		\underline{\phi}_{-1}=\phi_{-1}\cdot (\phi^{-1}_{-1}\underline{\phi}_{-1})
		=\phi_{-1}\cdot \chi_0|_{\bfS(F)}.
		\]
		Thus \cite[Lemma 3.4.28]{MR4013740} implies
		\[
		\underline{\pi}_{-1}=\pi_{(\bfS, \underline{\phi}_{-1})}^{\G^{0}}=\pi_{(\bfS, \phi_{-1})}^{\G^{0}}\otimes \chi_0=\pi_{-1}\otimes \chi_0
		\] 
		as required.
	\end{description}
	According to \cite[Lemma 4.22]{MR2431732} the conditions F0-F2 also ensure that $\underline{\Psi}$ is again a (generic cuspidal)
	Yu-datum.
	Moreover, as discussed in the paragraph before \cite[Example 3.7.4]{MR4013740}, it is regular.
	
	This completes the proof.
\end{proof}
\begin{prop}
	With the choice of a regular Yu-datum $\underline{\Psi}=(\G^0\subsetneq \G^1 \subsetneq \cdots \subsetneq \G^d, \underline{\pi}_{-1}, (\underline{\phi}_{0}, \dots, \underline{\phi}_{d}))$ as in Proposition \ref{prop:choiceofdatum}, we have $\kappa_{-1}\cong \Lambda_{\mathrm{t}}$.
\end{prop}
\begin{proof}
	This can be seen by directly comparing the construction of $\Lambda_{\mathrm{t}}$ from $\xi_{\mathrm{t}}$ with that of $\kappa_{-1}$ from $\underline{\pi}_{-1}=\pi_{(\bfS, \xi_{\mathrm{t}})}^{\G^{0}}$ (see \cite[Section 3.4]{MR4013740} and \cite[67 page]{MR2431732} for details of the latter one).
	Note that the definition of $\pi_{(\bfS, \xi_{\mathrm{t}})}^{\G^{0}}$ involves the Deligne--Lusztig construction \cite{MR393266} in place of the Green parametrization.
	However, the Deligne--Lusztig construction yields the Macdonald correspondence, which is nothing but the Green parametrization in the case of general linear groups (see \cite[Introduction]{MR393266}).
\end{proof}
\begin{prop} \label{prop:Lambda_0}
	With the choice of a regular Yu-datum $\underline{\Psi}=(\G^0\subsetneq \G^1 \subsetneq \cdots \subsetneq \G^d, \underline{\pi}_{-1}, (\underline{\phi}_{0}, \dots, \underline{\phi}_{d}))$ as in Proposition \ref{prop:choiceofdatum}, we have $\bigotimes_{i=0}^{d}\kappa_{i}\cong \Lambda_{\mathrm{w}}$.
\end{prop}
\begin{rem}
In fact, the properties (iii) and (iv) of Proposition \ref{prop:choiceofdatum} are not necessary for this proposition.
We assume them for expository convenience in the proofs of (Ext1) and (Ext4) below.
\end{rem}
\begin{proof}
	By Remark \ref{rem:Ext2},
	it suffices to check that the representation $\bigotimes_{i=0}^{d}\kappa_{i}$ of $K^{d}=\mathbf{J}$ constructed from the Yu-datum $\underline{\Psi}$ satisfies the conditions (Ext1), (Ext3), and (Ext4) in the definition of 
	$\Lambda_{\mathrm{w}}$.
	
	Let us show that $\bigotimes_{i=0}^{d}\kappa_{i}$ satisfies (Ext1).
	As the first step in constructing $\kappa_{i}$, a character $\underline{\hat{\phi}}_i$ on $K_+^d$ is constructed from $\underline{\phi}_{i}$ in \cite[591 page]{MR1824988}.
	The proof of \cite[Proposition 8.2 (i)]{Mayeux:2017aa} (or \cite[Proposition 8.12 (i)]{Mayeux:2020aa}), together with
	\cite{MR1888474}, gives the following expression of $\underline{\hat{\phi}}_i$:
	\begin{align*}
	\underline{\hat{\phi}}_i&=\psi_{\underline{c}_{i}} \quad \text{on $U_{\mfA_{i+1}}^{h_i}U_{\mfA_{i+2}}^{h_{i+1}}\cdots U_{\mfA_{d}}^{h_{d-1}}$}, \\
	\underline{\hat{\phi}}_i&=\underline{\xi}_{i} \circ {\textstyle\det_{i}} \quad \text{on $U_{\mfA_{0}}^{h_{-1}}U_{\mfA_{1}}^{h_0}\cdots U_{\mfA_{i}}^{h_{i-1}}$,}
	\end{align*}
	where
	\begin{itemize}
		\item each $\underline{\xi}_{i}$ is as in the proof of Proposition \ref{prop:choiceofdatum},
		\item $\underline{c}_{i}$ is an element in
		$\mfp_E^{-t_i}\cap E_i$ such that 
		\[
		\underline{\xi}_{i}(1+x)=\psi_F(\Tr_{E_{i}/F}(\underline{c}_{i}x))
		\]
		for $x\in \mfp_E^{h_i}\cap E_i$, and
		\item $\psi_{\underline{c}_{i}}$ is defined as $\psi_{c_i}$ 
		in the definition of $\theta$ with $\underline{c}_{i}$ in place of $c_i$.
	\end{itemize}
	By (iii) of Proposition \ref{prop:choiceofdatum} we have $\underline{\xi}_{i}|_{U_{E_i}^{1}}=\xi_{i}|_{U_{E_i}^{1}}$
	and thus $\psi_{\underline{c}_{i}}=\psi_{c_i}$.
	Therefore we see that the character $\prod_{i=0}^{d}\underline{\hat{\phi}}_i$ of $K_{+}^{d}$ agrees with the character $\theta$ of $H^{1}$ (recall that $K_{+}^{d}=H^{1}$).
	As discussed earlier,
	$\eta$ is the unique irreducible representation of $J^1$ containing
	the character $\prod_{i=0}^{d}\underline{\hat{\phi}}_i=\theta$ on $K_{+}^d=H^1$.
	Since $\bigotimes_{i=0}^{d}\kappa_{i}|_{K_{+}^{d}}$ does contain $\prod_{i=0}^{d}\underline{\hat{\phi}}_i$ by \cite[Lemma 3.27]{MR2431732},
	it is enough to show $\dim \bigotimes_{i=0}^{d}\kappa_{i}=\dim\eta$.
	This is done in (c) in the proof of \cite[Proposition 8.3]{Mayeux:2017aa} (or also \cite[Proposition 9.2]{Mayeux:2020aa}).
	
	(Ext3) is clear as we can easily check that if $Z(\G)$ denotes the center of $\G$, then $\bigotimes_{i=0}^d\kappa_{i}|_{Z(\G)(F)}$ is $\prod_{i=0}^{d}\underline{\phi}_{i}|_{Z(\G)(F)}$-isotypic 
	(cf.\ the proof of \cite[Fact 3.7.11]{MR4013740}),
	and we have $\prod_{i=0}^{d}\underline{\phi}_{i}|_{Z(\G)(F)}=\xi_{\mathrm{w}}|_{Z(\G)(F)}$ by (ii) of Proposition \ref{prop:choiceofdatum}.
	
	To prove that $\det \kappa_{i}$ has $p$-power order for $0\leq i\leq d$ 
	and thus $\bigotimes_{i=0}^d\kappa_{i}$ satisfies (Ext4),
	let us review the construction of $\kappa_{i}$ explained in \cite[Section 3.4]{MR2431732} in some details.
	
	In \cite[52-53 page]{MR2431732} certain open subgroups $K^{i}$ (resp.\ $J^{i}$) of $\G^{i}(F)$ are defined for $0\leq i\leq d$ (resp.\ $1\leq i\leq d$). 
	(Here we stick to the notation in \cite{MR2431732}. 
	The symbol $J^1$ also appeared in the construction of $\pi_{(E,\xi)}^{\BH}$,
	but what it denotes is not the same.)
	Among other properties these groups satisfy
	\[
	K^{i+1}=K^{i}J^{i+1}.
	\]
	As explained in \cite[67 page]{MR2431732},
	for $0\leq i\leq d-1$, the representation $\kappa_{i}$ of $K=K^{d}=K^{i+1}J^{i+2}\cdots J^{d}$ is defined 
	by first constructing a representation $\underline{\phi}'_{i}$ of $K^{i+1}$ from $\underline{\phi}_{i}$ and then extending it trivially on $J^{i+2}\cdots J^{d}$.
	Also, $\kappa_{d}$ is defined simply as the restriction of $\underline{\phi}_{d}$ to $K$, which has $p$-power order by (iv) of Proposition \ref{prop:choiceofdatum}.
	Thus we are reduced to showing that $\det \underline{\phi}'_{i}$ has finite $p$-power order for $0\leq i\leq d-1$.
	
	The definition of $\underline{\phi}'_{i}$ involves the following diagram:
	\[
	\xymatrix{
		K^{i}&K^{i}\ltimes J^{i+1}\ar@{->>}[l] \ar@{->>}[r] \ar@{->>}[d] & K^{i}\ltimes \mathcal{H}_i \\
		&K^{i+1},&
	}
	\]
	where $\mathcal{H}_{i}$ is a finite quotient of $J^{i+1}$,
	the two semi-direct products are induced from the conjugation action of $K^{i}$ on $J^{i+1}$
	and all the three maps are canonical surjective homomorphisms.
	The representation $\underline{\phi}'_{i}$ is defined by first constructing a representation $\omega_{i}$ of $K^{i}\ltimes \mathcal{H}_{i}$ from $\underline{\phi}_{i}$,
	then taking the tensor product of the inflations of $\underline{\phi}_{i}|_{K^i}$ and $\omega_{i}$ to $K^{i}\ltimes J^{i+1}$ and finally showing that this tensor product representation descends to a representation of $K^{i+1}$.
	As $\underline{\phi}_{i}$ has finite $p$-power order by (iv) of Proposition \ref{prop:choiceofdatum},
	it suffices to prove that $\det \omega_{i}$ has finite $p$-power order.
	
	The definition of $\omega_{i}$ is divided into two cases according to whether $\mathcal{H}_{i}$ is isomorphic to $\F_{p}$ or a non-trivial Heisenberg $p$-group.
	In the first case, 
	the semi-direct product $K^{i}\ltimes \mathcal{H}_{i}$ is in fact a direct product and $\omega_{i}$ is the pull-back of a nontrivial character on $\mathcal{H}_{i}$. Then $\det \omega_{i}=\omega_{i}$ indeed has finite $p$-power order.
	In the second case, let $W_i$ be the $\F_p$-vector space 
	obtained by dividing $\mathcal{H}_{i}$ by its center.
	The character $\underline{\phi}_{i}$ gives rise to a non-degenerate symplectic form on $W_i$ and the action of $K^{i}$ induces a symplectic action on $W_i$.
	Then $\omega_{i}$ is defined by means of the Heisenberg--Weil construction (see \cite[Section 2.3]{MR2431732} and \cite[Section 10]{MR2508725}). 
	As $\mathcal{H}_{i}$ is a $p$-group, it is enough to show that $\det \omega_{i}|_{K^{i}}$ has finite $p$-power order.
	Letting $\Sp(W_i)$ denote the symplectic group,
	we see that by the construction of the Weil representation
	the restriction $\det \omega_{i}|_{K^{i}}$ factors through $K^{i}\to \Sp(W_{i})\to \Sp(W_{i})^{\mathrm{ab}}$.
	If we write $V_{\omega_{i}}$ for the representation space of $\omega_{i}$,
	then the situation is summarized in the following diagram:
	
	\[
	\xymatrix{
		K^{i}\ltimes \mathcal{H}_i\ar[r]&\Sp(W_i)\ltimes \mathcal{H}_i \ar[r] &\GL(V_{\omega_{i}}) \ar[r]^(0.6){\det} &\C^{\times} \\
		K^{i} \ar@{^{(}->}[u] \ar[r] &\Sp(W_i) \ar@{^{(}->}[u] \ar[r] & \Sp(W_i)^{\textrm{ab}} \ar[ur]
	}
	\]
	Now we complete the proof by noting that $p$ is assumed to be odd and thus
	\[
	\Sp(W_i)^{\text{ab}}=\begin{cases}
	\Z/3\Z &\text{if $W_{i}\cong \F_{3}^2$} \\
	0 &\text{otherwise}
	\end{cases}
	\]
	is a $p$-group 
	(this is well-known as explained, for example, in \cite[Section 2.4]{MR2431732} and \cite[(1) of Proof of Theorem 2.4]{MR460477}).
\end{proof}


\end{document}